\documentclass[10pt]{amsart}
\usepackage{times,amsmath,amsbsy,amssymb,amscd,mathrsfs}
\usepackage{slashbox}\usepackage{graphicx,subfigure,epstopdf,wrapfig,chemarrow}

\usepackage{algorithm2e} 
\usepackage{multicol,multirow}
\usepackage{mathtools}
\usepackage[usenames,dvipsnames,svgnames,table]{xcolor}
\usepackage[all]{xy}
\usepackage{wrapfig}
\usepackage{tcolorbox}

\usepackage{tikz,tikz-cd}
\usepackage[utf8]{inputenc}
\usepackage{pgfplots} 
\usepackage{pgfgantt}
\usepackage{pdflscape}
\pgfplotsset{compat=newest} 
\pgfplotsset{plot coordinates/math parser=false}
\newlength\fwidth

\usepackage[numbered]{mcode}

\definecolor{myBlue}{rgb}{0.0,0.0,0.55}
\usepackage[pdftex,colorlinks=true,citecolor=myBlue,linkcolor=myBlue]{hyperref}

\usepackage[hyperpageref]{backref}

\newcommand{\LC}[1]{\textcolor{red}{#1}}

\usepackage{comment,enumerate,multicol,xspace}

  \newcounter{mnote}
  \newcommand{\mnote}[1]{\addtocounter{mnote}{1}
    \ensuremath{{}^{\bullet\arabic{mnote}}}
    \marginpar{\footnotesize\em\color{red}\ensuremath{\bullet\arabic{mnote}}#1}}
  \let\oldmarginpar\marginpar
    \renewcommand\marginpar[1]{\-\oldmarginpar[\raggedleft\footnotesize #1]%
    {\raggedright\footnotesize #1}}

\newcommand{\breakline}{
\begin{center}
------------------------------------------------------------------------------------------------------------
\end{center}
}

\newtheorem{theorem}{Theorem}[section]
\newtheorem{lemma}[theorem]{Lemma}
\newtheorem{corollary}[theorem]{Corollary}

\newtheorem{example}[theorem]{Example}

\newtheorem{remark}[theorem]{Remark}

\newcommand{\dx}{\,{\rm d}x}
\newcommand{\dd}{\,{\rm d}}

\newcommand{\bs}{\boldsymbol}

\DeclareMathOperator*{\img}{img}
\DeclareMathOperator*{\spa}{span}

\newcommand{\curl}{{\rm curl\,}}
\renewcommand{\div}{\operatorname{div}}
\newcommand{\grad}{{\rm grad\,}}
\DeclareMathOperator*{\rot}{rot}

\DeclareMathOperator{\dist}{dist}
\newcommand{\tr}{\operatorname{tr}}

\newcommand{\sym}{\operatorname{sym}}
\newcommand{\skw}{\operatorname{skw}}

\newcommand{\mskw}{\operatorname{mskw}}

\newcommand{\hess}{\operatorname{hess}}

\newcommand{\vertiii}[1]{{\left\vert\kern-0.25ex\left\vert\kern-0.25ex\left\vert #1 
    \right\vert\kern-0.25ex\right\vert\kern-0.25ex\right\vert}}

\usepackage{slashbox}
\usepackage{booktabs}
\usepackage{stmaryrd}
\usepackage{scalefnt}
\usepackage{graphicx} 
\newcommand{\Oplus}{\ensuremath{\vcenter{\hbox{\scalebox{1.5}{$\oplus$}}}}}

\newcommand{\air}{\operatorname{Air}}
\newcommand{\sskw}{\operatorname{sskw}}

\begin{document}
\title{Finite Element Complexes in Two Dimensions}

 \author{Long Chen}%
 \address{Department of Mathematics, University of California at Irvine, Irvine, CA 92697, USA}%
 \email{chenlong@math.uci.edu}%
 \author{Xuehai Huang}%
 \address{Corresponding author. School of Mathematics, Shanghai University of Finance and Economics, Shanghai 200433, China}%
 \email{huang.xuehai@sufe.edu.cn}%

 \thanks{The first author was supported by NSF DMS-1913080 and DMS-2012465. The second author was supported by the National Natural Science Foundation of China Project 12171300 and the Natural Science Foundation of Shanghai 21ZR1480500.}
\thanks{Accepted by SCIENTIA SINICA Mathematica (in Chinese).}

% \subjclass[2010]{
% %65N55;   %% Multigrid methods; domain decomposition
% %65F10;   %%  Iterative methods for linear systems
% 65N30;   %%  Finite elements, Rayleigh-Ritz and Galerkin methods, finite methods;
% 65N12;   %%  Stability and convergence of numerical methods;
% 65N22;   %%  Solution of discretized equations
% %65N15;   %%  Error bounds
% }

\begin{abstract}
In this study, two-dimensional finite element complexes with various levels of smoothness, including the de Rham complex, the curldiv complex, the elasticity complex, and the divdiv complex, are systematically constructed. Smooth scalar finite elements in two dimensions are developed based on a non-overlapping decomposition of the simplicial lattice and the Bernstein basis of the polynomial space, with the order of differentiability at vertices being greater than twice that at edges. Finite element de Rham complexes with different levels of smoothness are devised using smooth finite elements with smoothness parameters that satisfy certain relations. Finally, finite element elasticity complexes and finite element divdiv complexes are derived from finite element de Rham complexes by using the Bernstein-Gelfand-Gelfand (BGG) framework. This study is the first work to construct finite element complexes in a systematic way. Moreover, the novel tools developed in this work, such as the non-overlapping decomposition of the simplicial lattice and the discrete BGG construction, can be useful for further research in this field.
\end{abstract}

\maketitle

%\tableofcontents
\section{Introduction}
%\LC{Importance on Hilbert complexes.}
Hilbert complexes play a fundamental role in the theoretical analysis and the design of stable numerical methods for partial differential equations~\cite{ArnoldFalkWinther2006,Arnold;Falk;Winther:2010Finite,Arnold:2018Finite,ChenHuang2018}. Recently in~\cite{Arnold;Hu:2020Complexes}  Arnold and Hu have developed a systematical approach to derive new complexes from well-understood differential complexes such as the de Rham complex involving Sobolev spaces. In this work we shall construct two-dimensional finite element complexes with various smoothness in a systematic way, including finite element de Rham complexes, finite element elasticity complexes, and finite element divdiv complexes etc.

We first construct smooth finite elements in two dimensions by a geometric approach, in which the simplicial lattice $\mathbb T^{2}_k = \left \{ \alpha = (\alpha_0, \alpha_1, \alpha_2)\in\mathbb N^{0:2} \mid  \alpha_0+ \alpha_1+ \alpha_2 = k \right \}$ as the multi-index set with sum $k$ is employed. The smoothness (order of differentiability) at vertices and edges are specified by parameters $r^{\texttt{v}}$ and $r^e$, respectively. Let  $\mathcal T_h$ be a triangulation of a domain $\Omega \subset \mathbb R^2$ and denote by $\bs r = (r^{\texttt{v}}, r^e)$. When $r^{\texttt{v}}\geq 2r^e$ and $k\geq 2r^{\texttt{v}} + 1$, we construct $C^{r^e}$-continuous finite element spaces  $\mathbb V_k(\mathcal T_h; \bs r)$ using a non-overlapping decomposition (partition) of the simplicial lattice $\mathbb T_k^2$ and the Bernstein basis $\{\lambda^{\alpha}, \alpha\in \mathbb T^{2}_k\}$ of polynomial space $\mathbb P_k$, where $\lambda$ is the barycentric coordinate. Notice that the $C^{r^e}$-continuity implies   $\mathbb V_k(\mathcal T_h; \bs r)\subset H^{r^e+1}(\Omega)$. 

We then move to the finite element de Rham complexes with various smoothness which include discrete versions of the de Rham complex, for $r\geq 1$, 
\begin{equation}\label{eq:Stokescomplex}
\mathbb R \stackrel{\subset}{\longrightarrow} H^{r+1}(\Omega) \stackrel{\text { curl }}{\longrightarrow} \bs H^{r}(\Omega; \mathbb{R}^2) \stackrel{\text { div }}{\longrightarrow} H^{r-1}(\Omega) \longrightarrow 0, 
\end{equation}
and one with mixed regularities, for $r\geq 0, s\geq \max\{r-1,0\}$, 
\begin{equation}\label{eq:generalcase}
\mathbb R \stackrel{\subset}{\longrightarrow} H^{r+1}(\Omega) \stackrel{\text { curl }}{\longrightarrow} \bs H^{r, s} (\div, \Omega) \stackrel{\text { div }}{\longrightarrow} H^s(\Omega) \longrightarrow 0,
\end{equation}
where
$$
\boldsymbol{H}^{r, s}(\div,\Omega):=\{\boldsymbol{v}\in\boldsymbol{H}(\div,\Omega): \boldsymbol{v}\in\boldsymbol{H}^{r}(\Omega;\mathbb R^2), \div\boldsymbol{v}\in H^{s}(\Omega)\}.
$$
Obviously \eqref{eq:Stokescomplex} is a special case of \eqref{eq:generalcase} and also known as the Stokes complex. 

Given three integer vectors $\bs r_0 =(r_0^{\texttt{v}}, r_0^e)$, $\bs r_1 =(r_1^{\texttt{v}}, r_1^e)$, $\bs r_2 =(r_2^{\texttt{v}}, r_2^e)$ satisfying $r_1^{\texttt{v}}\geq 2 \, r_1^e + 1, r_2^{\texttt{v}}\geq 2 \, r_2^e$ and $\bs r_0 = \bs r_1 + 1, \bs r_1\geq -1, \boldsymbol{r}_2\geq \boldsymbol{r}_1\ominus 1
$, for $k$ sufficiently large,
% $k\geq \max\{2r_1^{\texttt{v}}+2, 2r_2^{\texttt{v}}+2, 1\}$, 
 we devise finite element de Rham complexes of various smoothness
\begin{equation}\label{eq:introfemcomplex1}
\mathbb R\xrightarrow{\subset} \mathbb V^{\curl}_{k+1}(\mathcal T_h; \boldsymbol{r}_0)\xrightarrow{\curl}\mathbb V^{\div}_{k}(\mathcal T_h; \boldsymbol{r}_1,\boldsymbol{r}_2) \xrightarrow{\div} \mathbb V^{L^2}_{k-1}(\mathcal T_h; \boldsymbol{r}_2)\xrightarrow{}0,
\end{equation}
which is a conforming discretization of the de Rham complex \eqref{eq:generalcase}. The finite element de Rham complex \eqref{eq:introfemcomplex1} with $\bs r_0 = \bs r_1 + 1$ and $\bs r_2 = \bs r_1 - 1$ has been developed recently in~\cite{huConstructionConformingFinite2021}.
We refer to
\cite{MardalTaiWinther2002,GuzmanNeilan2012} for some nonconforming Stokes complexes modified from conforming finite element de Rham complexes.

%$$
%\mathbb R\xrightarrow{\subset} H^{r_0^e+1}(\Omega)\xrightarrow{\curl} \boldsymbol{H}^{r_1^e+1, r_2^e+1}(\div,\Omega)\xrightarrow{\div} H^{r_2^e+1}(\Omega)\xrightarrow{}0,
%$$
%where $H^s(\Omega)$ is the standard Sobolev space and
%$$
%\boldsymbol{H}^{r_1^e+1, r_2^e+1}(\div,\Omega):=\{\boldsymbol{v}\in\boldsymbol{H}(\div,\Omega): \boldsymbol{v}\in\boldsymbol{H}^{r_1^e+1}(\Omega,\mathbb R^2), \div\boldsymbol{v}\in H^{r_2^e+1}(\Omega)\}.
%$$
By rotation of the vector field and differential operators, we also obtain the finite element de Rham complex involving $\grad, \rot$ operators:
\begin{equation}\label{eq:introfemcomplex2}
\mathbb R\xrightarrow{\subset} \mathbb V^{\grad}_{k+1}(\mathcal T_h; \boldsymbol{r}_0)\xrightarrow{\grad}\mathbb V^{\rot}_{k}(\mathcal T_h; \boldsymbol{r}_1,\boldsymbol{r}_2) \xrightarrow{\rot} \mathbb V^{L^2}_{k-1}(\mathcal T_h; \boldsymbol{r}_2)\xrightarrow{}0,
\end{equation}
in which the space $\mathbb V^{\rot}_{k}(\mathcal T_h; \boldsymbol{r}_1,\boldsymbol{r}_2)$ can find applications in the discretization of Maxwell equation or the fourth-order curl problems.

Several existing finite element de Rham complexes in two dimensions are special examples of~\eqref{eq:introfemcomplex1} or~\eqref{eq:introfemcomplex2}, and summarized in Table~\ref{table:femcomplexexamples}.
\begin{table}[htp]
	\centering
	\caption{Examples of finite element de Rham complexes \eqref{eq:introfemcomplex1}.}
\label{table:femcomplexexamples}
	\renewcommand{\arraystretch}{1.35}
	\begin{tabular}{@{} c c c c c @{}}
	\toprule
$k\geq $	 &	 $\bs r_0$ &  $\bs r_1$ & $\bs r_2$ & Results \\

%		&  \multicolumn{2}{c}{$\bs r_0$}&   \multicolumn{2}{c}{$\bs r_1$} &  \multicolumn{2}{c}{$\bs r_2$} & Results \\
%		\cline{2-3}		\cline{4-5}    \cline{6-7}
%		&   $\texttt{v}$ & $e$ & $\texttt{v}$ & $e$  & $\texttt{v}$ & $e$		
 \hline
$1$ & $(0, 0)$ & $(-1, -1)$ & $(-1, -1)$ & standard\\
$4$ & $(0, 0)$ & $(-1, -1)$ & $(0, 0)$ &~\cite[Section 5.2.1]{HuZhangZhang2020curlcurl}\\
$4$ & $(2, 1)$ & $(1, 0)$ & $(0, -1)$ &~\cite[Section 3]{FalkNeilan2013} \\
$4$ & $(2, 1)$ & $(1, 0)$ & $(0, 0)$ &~\cite[Section 4]{FalkNeilan2013} \\
$2$ & $(1, 0)$ & $(0, -1)$ & $(-1, -1)$ &~\cite[Section 2.2]{Christiansen;Hu;Hu:2018finite} 
\medskip \\

\bottomrule
	\end{tabular}
%	\label{table:derhamexamples}
\end{table}

%\begin{itemize}
%\item~\cite[Section 3]{FalkNeilan2013}: case $\bs r_1=(1,0)$, $\bs r_2=(0,-1)$,
%
%
%\item~\cite{Christiansen;Hu;Hu:2018finite}: case $\bs r_1=(0,-1)$, $\bs r_2=-1$,
%
%\item~\cite[Section 4]{FalkNeilan2013}: case $\bs r_1=(1,0)$, $\bs r_2=0$,
%
%\item~\cite[Section 5.2.1]{HuZhangZhang2020curlcurl}: the rotation of the finite element complex for $\bs r_1=-1$ and $\bs r_2=0$.
%
%\end{itemize} 

Based on finite element de Rham complexes, we use the Bernstein-Gelfand-Gelfand (BGG) framework~\cite{Arnold;Hu:2020Complexes} to construct more finite element complexes. For $\bs r_1\geq -1$ and $\bs r_2\geq \boldsymbol{r}_1\ominus 1$ satisfying $r_1^{\texttt{v}}\geq 2 \, r_1^e + 2, r_2^{\texttt{v}}\geq 2 \, r_2^e,$ and polynomial degree $k$ sufficiently large,
%$k\geq \max\{2r_1^{\texttt{v}} + 3, 2r_2^{\texttt{v}}+2\}$, 
we design the BGG diagram
\begin{equation*}%\label{eq:BGGelasticity}
\begin{tikzcd}
\mathbb R \arrow{r}{\subset}
&
\mathbb V_{k+2}^{\curl}(\bs r_1+2)
\arrow{r}{\curl}
 &
\mathbb V_{k+1}^{\div} (\bs r_1+1)
   \arrow{r}{\div}
 &
\mathbb V^{L^2}_k(\bs r_1)
 \arrow{r}{}
 & 0 \\
 \mathbb R^2 \arrow{r}{\subset}
&
\mathbb V_{k+1}^{\curl}(\bs r_1+1;\mathbb R^2)
 \arrow[ur,swap,"{\rm id}"] \arrow{r}{\curl}
 & 
\mathbb V_{k}^{\div}(\bs r_1, \bs r_2;\mathbb M) 
 \arrow[ur,swap,"{\rm -2\,sskw}"] \arrow{r}{\div}
 & 
\mathbb V^{L^2}_{k-1}(\bs r_2; \mathbb R^2) 
\arrow[r] 
 &\boldsymbol{0} 
\end{tikzcd}
\end{equation*}
which leads to the finite element elasticity complex
\begin{equation}\label{intro:elasticityfemcomplex}
{\mathbb P}_1\xrightarrow{\subset} \mathbb V_{k+2}^{\curl}(\bs r_1+2)\xrightarrow{\air}\mathbb V_{k}^{\div}(\bs r_1, \bs r_2; \mathbb S)\xrightarrow{\div}  \mathbb V^{L^2}_{k-1}(\bs r_2; \mathbb R^2)\xrightarrow{}\boldsymbol{0}.
\end{equation}
%where $\mathbb V_{k}^{\div}(\bs r_1, \bs r_2; \mathbb S):=\mathbb V_{k}^{\div}(\bs r_1, \bs r_2;\mathbb M)\cap\ker(\sskw)$. And we figure out the DoFs of space $\mathbb V_{k}^{\div}(\bs r_1, \bs r_2; \mathbb S)$ by symmetrizing the DoFs of space $\mathbb V_{k}^{\div}(\bs r_1, \bs r_2;\mathbb M)$.

For $\bs r_1\geq 0$, $\bs r_2\geq \max\{\bs r_1 - 2,-1\}$, and $k\geq \max\{2r_1^{\texttt{v}} + 3, 2r_2^{\texttt{v}}+3\}$, we build the BGG diagram
 \begin{equation*}%\label{eq:femdivdivbgg}
\begin{tikzcd}[column sep=0.5cm]
\mathbb R^2 \arrow{r}{\subset}
&
\mathbb V_{k+1}^{\curl}(\bs r_1+1; \mathbb R^2)
\arrow{r}{\curl}
 &
\mathbb V_{k}^{\div\div^+} (\bs r_1, \bs r_2; \mathbb M)
   \arrow{r}{\div}
 &
\mathbb V^{\div}_{k-1}(\bs r_1 - 1, \bs r_2)
 \arrow{r}{}
 & \boldsymbol{0} \\
 \mathbb R \arrow{r}{\subset}
&
\mathbb V_{k}^{\curl}( \bs r_1)
 \arrow[ur,swap,"{\rm mskw}"] \arrow{r}{\curl}
 & 
\mathbb V_{k-1}^{\div}(\bs r_1 - 1, \bs r_2) 
 \arrow[ur,swap,"{\rm id}"] \arrow{r}{\div}
 & 
\mathbb V^{L^2}_{k-2}(\bs r_2) 
\arrow[r] 
 &0 
\end{tikzcd}
\end{equation*}
which leads to the finite element divdiv complex
\begin{equation}\label{intro:continuousdivdivcomplex}
{\bf RT}\xrightarrow{\subset} 
\mathbb V_{k+1}^{\curl}(\bs r_1+1;\mathbb R^2)
\xrightarrow{{\sym\curl}}
\mathbb V_{k}^{\div\div^+} (\bs r_1, \bs r_2; \mathbb S)\xrightarrow{\div\div}  
\mathbb V^{L^2}_{k-2}(\bs r_2) 
\xrightarrow{}0,
\end{equation}
where $\mathbb V_{k}^{\div\div^+}(\bs r_1, \bs r_2; \mathbb S) \subset \boldsymbol{H}(\operatorname{divdiv}, \Omega ; \mathbb{S}) \cap \boldsymbol{H}(\operatorname{div}, \Omega ; \mathbb{S})$. We refer to Section~\ref{sec:femcomplexbgg} for details. By a refinement of the BGG diagram, the finite element divdiv complexes presented in~\cite{Hu;Ma;Zhang:2020family} and~\cite{ChenHuang2020} with $\bs r_1=(0,-1),\bs r_2=(-1,-1)$ are also covered.

Several existing finite element complexes in two dimensions can be viewed as special cases of~\eqref{intro:elasticityfemcomplex} or~\eqref{intro:continuousdivdivcomplex}, and are summarized in Table~\ref{table:complexampless}. However, discrete elasticity complexes and rot\,rot complexes based on the Clough-Tocher split in~\cite{ChristiansenHu2022} are constructed using piece-wise polynomials as shape functions, which are not covered by~\eqref{intro:elasticityfemcomplex} and \eqref{intro:continuousdivdivcomplex}.
\begin{table}[htp]
	\centering
	\caption{Examples of finite element elasticity and finite element divdiv complexes.}
	\renewcommand{\arraystretch}{1.35}
	\begin{tabular}{@{} c c c c c @{}}
	\toprule
	Type & $k\geq$ &	 $\bs r_1$ & $\bs r_2$ & Results \\

%		&  \multicolumn{2}{c}{$\bs r_0$}&   \multicolumn{2}{c}{$\bs r_1$} &  \multicolumn{2}{c}{$\bs r_2$} & Results \\
%		\cline{2-3}		\cline{4-5}    \cline{6-7}
%		&   $\texttt{v}$ & $e$ & $\texttt{v}$ & $e$  & $\texttt{v}$ & $e$		
 \hline
Elasticity complex \eqref{intro:elasticityfemcomplex} & $3$ & $(0,-1)$ & $(-1, -1)$ &~\cite[Section 6]{Christiansen;Hu;Hu:2018finite}\\
Hessian complex (rotation of \eqref{intro:elasticityfemcomplex})& $5$ & $(0, -1)$ & $(0, 0)$ &~\cite[Section 5.1]{Chen;Huang:2021Finite} \\
divdiv complex \eqref{intro:continuousdivdivcomplex} & $6$ & $(1, 0)$ & $(0, 0)$ &~\cite[Section 5.2]{Chen;Huang:2021Finite}
 \\
divdiv complex \eqref{eq:femdivdivcomplex} & $3$ & $(0, -1)$ & $(-1, -1)$ &~\cite[Section 2.3]{Hu;Ma;Zhang:2020family} \\
divdiv complex \eqref{eq:divdivfemcomplex} & $3$ & $(0, -1)$ & $(-1, -1)$  &~\cite[Section 3.3]{ChenHuang2020}
\medskip \\

\bottomrule
	\end{tabular}
	\label{table:complexampless}
\end{table}

The rest of this paper is organized as follows. The de Rham complex and BGG framework are reviewed in Section~\ref{sec:hilbertcomplex}. 
In Section~\ref{sec:geodecomp2d} the geometric decomposition of $C^m$-conforming finite elements in two dimensions is studied. Finite element de Rham complexes with various smoothness are constructed in Section~\ref{sec:femderhamcomplex}.
More finite element complexes based on the BGG approach are developed in Section~\ref{sec:femcomplexbgg}.

\section{Preliminaries on Hilbert complexes}\label{sec:hilbertcomplex}
%In this paper, we shall consider Hilbert complexes associated to a two-dimensional domain $\Omega$ and $\dd_i$ are some differential operators of scalar, vector, or tensor functions. 
%
\subsection{Notation}
For scalar function $v$, denote
$$
\mskw v:=\begin{pmatrix}
0 & -v \\
v & 0
\end{pmatrix},\quad \curl v:=\left(
\frac{\partial v}{\partial x_2},
-\frac{\partial v}{\partial x_1}
\right)^{\intercal} = (\grad v)^{\perp},
$$
where $(a,b)^{\perp} := (b, -a)$ is the $90^{\circ}$ rotation clock-wisely, $\hess:=\grad\grad$, and
$$
\air v:=\curl\curl v=\begin{pmatrix}
\medskip
\frac{\partial^2v}{\partial x_2^2} & -\frac{\partial^2v}{\partial x_1\partial x_2} \\
-\frac{\partial^2v}{\partial x_1\partial x_2} & \frac{\partial^2v}{\partial x_1^2}
\end{pmatrix}.
$$
Then $\div(\mskw v)=-\curl v$.
For vector function $\boldsymbol{v}=(v_1, v_2)$, denote
$$
\rot\boldsymbol{v}:=\frac{\partial v_2}{\partial x_1}-\frac{\partial v_1}{\partial x_2} = \div \bs v^{\perp}.
$$
For tensor function $\boldsymbol{\tau}=\begin{pmatrix}
\tau_{11} & \tau_{12} \\
\tau_{21} & \tau_{22}
\end{pmatrix}$, denote
\begin{align*}
&\sym\boldsymbol{\boldsymbol{\tau}}:=\frac{1}{2}(\boldsymbol{\tau}+\boldsymbol{\tau}^{\intercal}),\, \skw\boldsymbol{\boldsymbol{\tau}}:=\frac{1}{2}(\boldsymbol{\tau}-\boldsymbol{\tau}^{\intercal}),\,\\
&\sskw\boldsymbol{\boldsymbol{\tau}}:=\mskw^{-1}\circ\skw\boldsymbol{\boldsymbol{\tau}}=\frac{1}{2}(\tau_{21}-\tau_{12}).
\end{align*}
By direct calculation, we have 
\begin{equation}\label{eq:anticommutativeprop1}  
\div\boldsymbol{v}=2\,\sskw(\curl \boldsymbol{v}), \quad \rot\boldsymbol{v} = 2\,\sskw(\grad \boldsymbol{v}).
\end{equation}

\subsection{Hilbert complex and exact sequence}
A Hilbert complex is a sequence of Hilbert spaces $\{\mathcal V_i \}$ connected by a sequence of closed densely defined linear operators $\{\dd_i\}$ 
\begin{equation}\label{eq:Hilbertcomplex}
0 \stackrel{}{\hookrightarrow} \mathcal V_1 \stackrel{\dd_1}{\longrightarrow} \mathcal V_2 \stackrel{\dd_2}{\longrightarrow} \cdots \stackrel{\dd_{n-2}}{\longrightarrow}\mathcal V_{n-1}\stackrel{\dd_{n-1}}{\longrightarrow} \mathcal V_n \stackrel{\dd_{n}}{\longrightarrow} 0,
\end{equation}
satisfying the property $\img (\dd_i)\subseteq \ker(\dd_{i+1})$. 
% Therefore in the Hilbert complex \eqref{eq:Hilbertcomplex}, $\dd_1$ is injective. 
We will abbreviate Hilbert complexes as complexes. The complex is called an exact sequence if $\ker(\dd_{1})=0$ and $\img (\dd_i)=  \ker(\dd_{i+1})$ for $i=1, \ldots, n-1$. Therefore if  \eqref{eq:Hilbertcomplex} is exact, $\dd_1$ is injective and $\dd_{n-1}$ is surjective. To save notation, we usually skip the trivial space $0$ in the beginning of the complex and use the embedding $\mathcal V_1 \stackrel{\subset}{\longrightarrow} \mathcal V_2$ to indicate $\dd_1$ is injective. For more background on Hilbert complexes, we refer to~\cite{Arnold:2018Finite}. 

When the Hilbert spaces are finite-dimensional, to verify the exactness, we rely on the following result on the dimension count.
\begin{lemma}\label{lm:abstract}
Let 
\begin{equation}\label{eq:shortexactsequence}
\mathcal V_0 \stackrel{\subset}{\longrightarrow} \mathcal V_1 \stackrel{\dd_1}{\longrightarrow} \mathcal V_2 \stackrel{\dd_2}{\longrightarrow} \mathcal V_{3}\longrightarrow 0,
\end{equation}
be a complex, where $\mathcal V_i$ are finite-dimensional linear spaces for $i=0,\ldots,3$. Assume $\mathcal V_0=\mathcal V_1\cap\ker(\dd_1)$, and 
\begin{equation}\label{eq:abstractdimenidentity}
\dim \mathcal V_0 - \dim \mathcal V_1 + \dim \mathcal V_2 - \dim \mathcal V_3 = 0.
\end{equation}
If either $\dd_1\mathcal V_1 = \mathcal V_2\cap \ker(\dd_2)$ or $\dd_2 \mathcal V_2 = \mathcal V_3$, then complex~\eqref{eq:shortexactsequence} is exact. 
\end{lemma}
\begin{proof}
Given the identity~\eqref{eq:abstractdimenidentity} and the relation $\mathcal V_0=\mathcal V_1\cap\ker(\dd_1)$, we prove the equivalence of $\dd_1\mathcal V_1 = \mathcal V_2\cap \ker(\dd_2)$ and $\dd_2 \mathcal V_2 = \mathcal V_3$ by dimension count. By $\mathcal V_0=\mathcal V_1\cap\ker(\dd_1)$,
$$
\dim \dd_1\mathcal V_1=\dim\mathcal V_1 -\dim(\mathcal V_1\cap\ker(\dd_1))=\dim\mathcal V_1 -\dim\mathcal V_0.
$$
Then it follows from~\eqref{eq:abstractdimenidentity} that
\begin{align*}
\dim(\mathcal V_2\cap\ker(\dd_2))-\dim \dd_1\mathcal V_1 &=\dim\mathcal V_2-\dim\dd_2\mathcal V_2 - \dim\mathcal V_1+\dim\mathcal V_0 \\
&=\dim\mathcal V_3-\dim\dd_2\mathcal V_2,
\end{align*}
as required.
\end{proof}

\subsection{The de Rham complex}
For a domain $\Omega \subseteq \mathbb R^2$, the de Rham complex is
\begin{equation}\label{eq:deRham}
\mathbb R\xrightarrow{\subset} H^{1}( \Omega) \xrightarrow{\curl} \boldsymbol{H}(\operatorname{div}, \Omega) \xrightarrow{\operatorname{div}} L^{2}(\Omega) \rightarrow 0.
\end{equation}
When $\Omega$ is simply connected, the de Rham complex~\eqref{eq:deRham} is exact. By changing smoothness of the Sobolev spaces, we obtain the version \eqref{eq:generalcase}. 

Restricted to one triangle, a polynomial de Rham complex is, for integer $k\geq 1$,
\begin{equation}\label{eq:polyderham}
\mathbb R\xrightarrow{\subset} \mathbb P_{k+1}(T)\xrightarrow{\curl}\mathbb P_{k}(T;\mathbb R^2) \xrightarrow{\div} \mathbb P_{k-1}(T)\xrightarrow{}0,   
\end{equation}
where $\mathbb P_k(T)$  denotes the set of real valued polynomials defined on $T$ of degree less than or equal to $k$, and $\mathbb P_{k}(T;\mathbb X):=\mathbb P_{k}(T)\otimes\mathbb X$ for $\mathbb X$ being vector space $\mathbb R^2$, tensor space $\mathbb M$, or symmetric tensor space $\mathbb S$.

The following identity
\begin{equation}\label{eq:dimensionpolyderham}
1 - {k+3 \choose 2} + 2{k+2 \choose 2} - {k+1 \choose 2} = 0
\end{equation}
can be verified directly. The relation $\curl \mathbb P_{k+1}(T) = \mathbb P_{k}(T;\mathbb R^2)\cap \ker (\div)$ is due to the fact: if $\grad p\in \mathbb P_k(T;\mathbb R^2)$, then $p\in \mathbb P_{k+1}(T)$, and in two dimensions $\curl$ is a rotation of $\grad$. Therefore complex \eqref{eq:polyderham} is exact by Lemma \ref{lm:abstract}.

\subsection{Bernstein-Gelfand-Gelfand construction}
Eastwood's work~\cite{Eastwood2000} established the relationship between the elasticity complex and the de Rham complex via the Bernstein-Gelfand-Gelfand (BGG) construction~\cite{BernsteinGelfandGelfand1975}. Arnold, Falk, and Winther~\cite{ArnoldFalkWinther2006b} expanded upon this connection by replicating the same construction in the discrete setting, which they used to reconstruct the finite element elasticity complex from the finite element de Rham complexes, as previously introduced in~\cite{Arnold;Winther:2002finite}. While a systematic BGG construction has been developed more recently in~\cite{Arnold;Hu:2020Complexes}, our focus in this work is limited to two-dimensional complexes, so we will rely on specific examples rather than the abstract framework in~\cite{Arnold;Hu:2020Complexes}.

We stack two de Rham complexes into the BGG diagram
\begin{equation}\label{eq:BGGelasticitysobolev}
\begin{tikzcd}
\mathbb R \arrow{r}{\subset}
&
H^{2}(\Omega)
\arrow{r}{\curl}
 &
\bs H^1 (\Omega; \mathbb R^2)
   \arrow{r}{\div}
 &
L^2(\Omega)
 \arrow{r}{}
 & 0 \\
 \mathbb R^2  \arrow[ur,swap, near end, "{\rm \cdot(-\bs x)^{\perp}}"] \arrow{r}{\subset}
&
\bs H^1 (\Omega; \mathbb R^2)
 \arrow[ur,swap,"{\rm id}"] \arrow{r}{\curl}
 & 
\bs H(\div,\Omega;\mathbb M) 
 \arrow[ur,swap,"{\rm -2\,sskw}"] \arrow{r}{\div}
 & 
\bs L^2(\Omega;\mathbb R^2)
\arrow[r] 
 &\bs 0 
\end{tikzcd},
\end{equation}
which leads to the elasticity complex
\begin{equation}\label{eq:elascomplex}
\mathbb P_1\xrightarrow{\subset} H^{2}\left(\Omega\right) \xrightarrow{\air} \boldsymbol{H}(\operatorname{div}, \Omega ; \mathbb{S}) \xrightarrow{\operatorname{div}} \boldsymbol{L}^2(\Omega;\mathbb R^2) \rightarrow \boldsymbol{0}.
\end{equation}
By rotation, we also have the Hessian complex
\begin{equation*}%\label{eq:hesscomplex}
\mathbb P_1\xrightarrow{\subset} H^{2}\left(\Omega\right) \xrightarrow{{\rm hess}} \boldsymbol{H}(\operatorname{rot}, \Omega ; \mathbb{S}) \xrightarrow{\operatorname{rot}} \boldsymbol{L}^2(\Omega;\mathbb R^2) \rightarrow \boldsymbol{0}.
\end{equation*}

To provide a more effective explanation of how \eqref{eq:elascomplex} is derived from \eqref{eq:BGGelasticitysobolev}, we present a step-by-step breakdown of the process. 
The anti-commutativity $\div\boldsymbol{v}=2\,\sskw(\curl \boldsymbol{v})$ is exactly the first identity in~\eqref{eq:anticommutativeprop1}, by which we can change $\boldsymbol{H}(\operatorname{div}, \Omega ; \mathbb{M})$ to $\boldsymbol{H}(\operatorname{div}, \Omega ; \mathbb{S})$ as follows. For $\bs u\in\boldsymbol{L}^2(\Omega;\mathbb R^2)$, by the exactness of the bottom complex in \eqref{eq:BGGelasticitysobolev}, there exists $\bs\tau\in \bs H(\div,\Omega;\mathbb M)$ satisfying $\bs u=\div\bs\tau$. Then apply the top complex in \eqref{eq:BGGelasticitysobolev} to find $\bs v\in\bs H^1 (\Omega; \mathbb R^2)$ satisfying $-2\sskw\bs\tau=\div\bs v$. Set $\bs\sigma=\bs\tau+\curl\bs v\in \bs H(\div,\Omega;\mathbb M)$. Clearly $\div\bs\sigma=\div\bs\tau=\bs u$. By the anti-commutativity, we have $2\,\sskw\bs\sigma=2\,\sskw\bs\tau+2\,\sskw(\curl\bs v)=2\,\sskw\bs\tau+\div\bs v=0$, i.e. $\bs\sigma\in\boldsymbol{H}(\operatorname{div}, \Omega ; \mathbb{S})$. This explains the div stability $\div\boldsymbol{H}(\operatorname{div}, \Omega ; \mathbb{S})=\boldsymbol{L}^2(\Omega;\mathbb R^2)$. 
The relation of these functions is summarized below:
\begin{equation*}
\begin{tikzcd}[column sep=large]
&
\bs v
 \arrow[dl,swap,"{\rm id}"]   \arrow{r}{\div}
 &
-2\sskw(\bs \tau)
 \\
 \bs v
% \arrow[ur,swap,"2\vskw"] 
 \arrow{r}{\curl}
 & 
 \bs \tau
\arrow[ur,swap, near end, "{\rm -2\,sskw}"] \arrow{r}{\div}
 & 
\bs u
\end{tikzcd}.
\end{equation*}
The composition of two $\curl$ operators leads to $H^{2}\left(\Omega\right) \xrightarrow{\air} \boldsymbol{H}(\operatorname{div}, \Omega ; \mathbb{S})$. The null space $\ker(\air)$ consists of $\mathbb R+\mathbb R^2\cdot\bs x^{\perp}=\mathbb P_1$.

The BGG diagram 
\begin{equation}\label{eq:BGGdivdivsobolev}
\begin{tikzcd}
\mathbb R^2 \arrow{r}{\subset}
&
\bs H^1 (\Omega; \mathbb R^2) \arrow{r}{\curl}
& 
\boldsymbol{H}(\operatorname{divdiv}, \Omega ; \mathbb{M}) 
 \arrow{r}{\div}
& 
\bs H^{-1,0}(\div,\Omega)
\arrow[r] 
&\bs 0 
\\
 \mathbb R \arrow[ur,swap,"{\rm -\bs x}"] \arrow{r}{\subset}
&
L^{2}(\Omega)
\arrow[ur,swap,"{\rm mskw}"] \arrow{r}{\curl}
& 
\bs H^{-1,0}(\div,\Omega)
\arrow[ur,swap,"{\rm id}"] \arrow{r}{\div}
& 
L^{2}(\Omega)
\arrow[r] 
&0 
\end{tikzcd}
\end{equation}
will lead to the divdiv complex
\begin{equation*}%\label{eq:divdivcomplex}
{\bf RT}\xrightarrow{\subset} \boldsymbol{H}^{1}\left(\Omega ; \mathbb{R}^{2}\right) \xrightarrow{\sym\curl} \boldsymbol{H}(\operatorname{divdiv}, \Omega ; \mathbb{S}) \xrightarrow{\operatorname{divdiv}} L^2(\Omega) \rightarrow 0,
\end{equation*}
and, again by rotation, the strain complex
\begin{equation*}%\label{eq:rotrotcomplex}
{\bf RM}\xrightarrow{\subset} \boldsymbol{H}^{1}\left(\Omega ; \mathbb{R}^{2}\right) \xrightarrow{\sym\grad} \boldsymbol{H}(\operatorname{rotrot}, \Omega ; \mathbb{S}) \xrightarrow{\operatorname{rotrot}} L^2(\Omega) \rightarrow 0,
\end{equation*}
where ${\bf RT}:=\mathbb R^2+\bs x\mathbb R$ and ${\bf RM}:=\mathbb R^2+\bs x^{\perp}\mathbb R$.

The anti-commutativities in \eqref{eq:BGGdivdivsobolev} are $\curl(c\,\bs x)=\mskw c$ for $c\in\mathbb R$ and  $\div(\mskw v)=-\curl v$ for $v\in L^2(\Omega)$.
For $p \in L^2(\Omega)$, by the exactness of the bottom complex in \eqref{eq:BGGdivdivsobolev}, there exists $\bs u\in \bs H^{-1,0}(\div,\Omega)$ satisfying $p=\div\bs u$. Then apply the top complex in \eqref{eq:BGGdivdivsobolev} to find $\bs\tau\in\boldsymbol{H}(\operatorname{divdiv}, \Omega ; \mathbb{M})$ s.t. $\bs u=\div\bs\tau$. Set $\bs\sigma=\bs\tau+\mskw w\in \bs H(\div\div,\Omega;\mathbb M)$ with $w=-\sskw\bs\tau\in L^2(\Omega)$. 
By the anti-community, $\div\div(\mskw w)=-\div(\curl w)=0$.
Hence $\div\div\bs\sigma=\div\div\bs\tau=p$, and $\sskw\bs\sigma=\sskw\bs\tau+\sskw(\mskw w)=\sskw\bs\tau+w=0$, i.e. $\bs\sigma\in\boldsymbol{H}(\div\div, \Omega ; \mathbb{S})$. This explains $\div\div\boldsymbol{H}(\div\div, \Omega ; \mathbb{S})=L^2(\Omega)$.
The chase of the diagram is summarized below:
\begin{equation*}
\begin{tikzcd}[column sep=large]
&
\bs \tau
 \arrow[dl,swap,shift right=0.6ex,"{\rm -\sskw}"]   \arrow{r}{\div}
 &
\bs u
 \\
w
 \arrow[ur,swap,shift right=0.1ex, near end,"\mskw"] 
 \arrow{r}{\curl}
 & 
\bs u
\arrow[ur,swap, near end, "{\rm id}"] \arrow{r}{\div}
 & 
p
\end{tikzcd}.
\end{equation*}
The null space $\ker(\sym\curl)$ is given by $\mathbb R^2+\bs x\mathbb R={\bf RT}$.

We shall construct finite element counterparts of the BGG diagrams~\eqref{eq:BGGelasticitysobolev}-\eqref{eq:BGGdivdivsobolev}, and derive several finite element elasticity and divdiv complexes.
The first step is to design finite element de Rham complexes of different smoothness.

\section{Smooth Finite Elements in Two Dimensions}\label{sec:geodecomp2d}
In this section, we shall construct $C^m$-continuous finite elements on two-dimensional triangular grids, firstly constructed by Bramble and Zl\'amal~\cite{BrambleZlamal1970},  by a decomposition of the simplicial lattice. 
%We start from a Hermite finite element space which ensures the tangential derivatives across edges are continuous. By adding degrees of freedom on the normal derivative, we can impose the continuity of derivatives across triangles. 
%Note that the smoothness at vertices is at least $C^{2m}$ which is sufficient but may not be necessary.

We use a pair of integers $\bs r = (r^{\texttt{v}}, r^{e})$ for the smoothness at vertices and at edges, respectively. Value $-1$ means no continuity. To be $C^m$-continuous, $r^{e} = m$ is the minimum requirement for edges and $r^{\texttt{v}}\geq \max\{2r^{e},-1\}$  for vertices. The polynomial degree $k\geq \max\{2r^{\texttt{v}}+1,0\}$. For a vector $\bs r\in \mathbb R^d$ and a constant $c$, $\bs r \geq c$ means $r_i \geq c$ for all components $i=1, 2, \ldots, d$, and $\bs r+c:=(r_1+c, r_2+c, \ldots, r_d+c)$. Define $\boldsymbol r\ominus 1: = \max \{\boldsymbol r-1, -1\}$

\subsection{Simplicial lattice}
%We first introduce the simplicial lattice for 2D triangle. Generalization to arbitrary dimension can be found in our recent work~\cite{ChenHuang2021Cmgeodecomp}.

For two non-negative integers $l\leq m$, we will use the multi-index notation $\alpha \in \mathbb{N}^{l:m}$, meaning $\alpha=\left(\alpha_{l}, \ldots, \alpha_{m}\right)$ with integer $\alpha_{i} \geqslant 0$. The length of $\alpha$ is $m-l+1$. The sum (absolute value) of a multi-index is $|\alpha|:=\sum_{i=l}^m \alpha_{i}$ for $\alpha\in \mathbb{N}^{l:m}$ and the factorial is $\alpha! := \alpha_l!\cdots\alpha_m!$. Denote
$$
D^{\beta}:=\frac{\partial^{|\beta|}}
{\partial x_1^{\beta_1} x_2^{\beta_2}}, \quad \beta \in \mathbb{N}^{1:2}.
%    \quad \mbox{with } |\alpha|=\sum_{i=1}^n\alpha_i.
$$

A simplicial lattice of degree $k\geq 1$ in two dimensions is a multi-index set of length $3$ with fixed sum $k$, i.e.,
$$
\mathbb T^{2}_k := \left \{ \alpha = (\alpha_0, \alpha_1, \alpha_2)\in\mathbb N^{0:2} \mid |\alpha | = k \right \}.
$$
%Similarly define $\mathbb T^{1:n}_k =\left \{ \alpha = (\alpha_1, \alpha_2)\in\mathbb N^{1:2} \mid \alpha_1 + \ldots + \alpha_n = k \right \}$.
An element $\alpha\in \mathbb T^{2}_k$ is called a node of the lattice. 
%

%Two nodes $\alpha, \beta\in \mathbb T^{2}_k$ are adjacent if there exist $0\leq i_1<i_2\leq n$ such that $|\alpha_{i_1} - \beta_{i_1}|=|\alpha_{i_2} - \beta_{i_2}|=1$ and $|\alpha_i - \beta_i|=0$ for $i\neq i_1, i_2$. By assigning edges to all adjacent nodes, the simplicial lattice becomes an undirected graph. The distance of two nodes in the graph is the length of a minimal path connecting them, where the length of a path is defined as the number of edges in the path. One can easily verify that the graph distance $\dist_G(\alpha,\beta)=\frac{1}{2}\sum_{i=0}^n|\alpha_i-\beta_i|$ which is one half of the $L^1$-norm of $\alpha - \beta$ treating $\alpha,\beta \in \mathbb R^{n+1}$.

%\subsection{Geometric embedding of a simplicial lattice}
We can embed the simplicial lattice $\mathbb T^{2}_k$ into a triangle $T$ with vertices $\{\texttt{v}_0, \texttt{v}_1, \texttt{v}_2\}$. Given $\alpha\in \mathbb T^{2}_k$, the barycentric coordinate of $\alpha$ is given by
$\lambda(\alpha) = (\alpha_0, \alpha_1, \alpha_2 )/k$, and the geometric embedding is
$$
x: \mathbb T^{2}_k \to T, \quad x(\alpha) = \sum_{i=0}^2 \lambda_i(\alpha) \texttt{v}_i. 
$$
%We will always assume such a geometric embedding of the simplicial lattice exists and write as $\mathbb T^{2}_k(T)$. 

The left side of Fig.~\ref{fig:lattice} illustrates the embedding of a two-dimensional simplicial lattice $\mathbb T_8^2$ within a reference triangle $\hat{T}$ with vertices ${(0,0), (1,0), (0,1)}$, while the right side shows the embedding of the same lattice into an equilateral triangle.
%See Fig.~\ref{fig:lattice} for an illustration for a two-dimensional simplicial lattice embedded into an equilateral triangle (right) and a reference triangle $\hat T$ with vertices $\{ (0,0), (1,0), (0,1)\}$ (left).

% \LC{Add a reference triangle with lattice nodes on the left.}
\begin{figure}[htbp]
\begin{center}
\includegraphics[width=4.5cm]{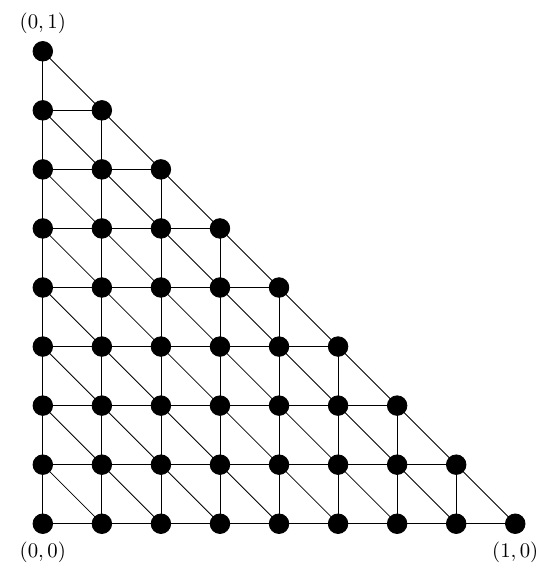}\qquad \quad
\includegraphics[width=4.5cm]{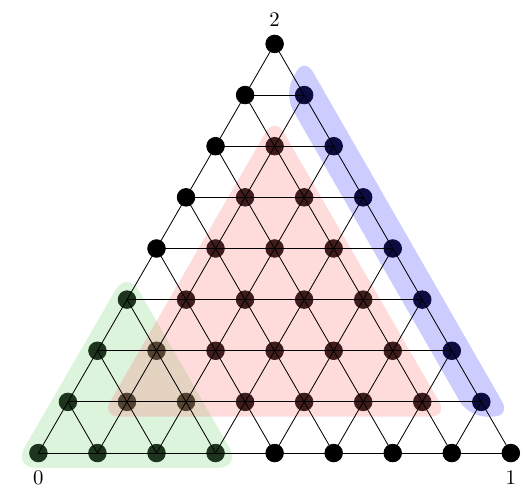}
\caption{Two embedding of the simplicial lattice $\mathbb T_8^2$ in two dimensions.}
\label{fig:lattice}
\end{center}
\end{figure}

A simplicial lattice $\mathbb T^{2}_k$ is, by definition, an algebraic set. Through the geometric embedding $\mathbb T^{2}_k(T)$, we can apply operators for the geometric simplex $T$. For example, for a subset $S\subseteq T$, we use $\mathbb T^{2}_k(S) = \{ \alpha \in \mathbb T^{2}_k, x(\alpha)\in S\}$ to denote the portion of lattice nodes whose geometric embedding is inside $S$. 

%The so-called reference simplex $\hat T$ is spanned by vertices $\texttt{v}_0 = \bs 0$ and $\texttt{v}_i = \bs e_i = (0, \ldots, 1, \ldots, 0)$, whose barycentric coordinate $\lambda_i = x_i$ for $i=1,2,\ldots n$ and $\lambda_0 = 1 - \sum_{i=1}^n x_i$. If we embed $\mathbb T^{2}_k$ to the scaled reference simplex $k \hat T$, then the coordinate of $(\alpha_0, \alpha_1, \alpha_2 )\in \mathbb T^{2}_k(k \hat T)$ is simply $(\alpha_1, \alpha_2 )$, where $\alpha_0$ is dropped as the vertex $\texttt{v}_0$ is mapped to the origin. Of course, we can set other vertex as the origin and obtain other embeddings.
%By such embedding, we can extend the length operator for multi-index to barycentric coordinate:
%$$
%| \lambda | = \lambda_0 + \lambda_1 + \ldots + \lambda_n
%$$

\subsection{Bernstein basis}
It holds that
$$| \mathbb T^{2}_k | = {k + 2\choose k} = \dim \mathbb P_k(T).$$
Let $T$ be a triangle with vertices $\{\texttt{v}_0, \texttt{v}_1, \texttt{v}_2\}$ and $\lambda_i, i=0,1,2,$ be the barycentric coordinate. 
%We define $\lambda^{\alpha}:=\lambda_{0}^{\alpha_{0}} \lambda_{1}^{\alpha_{1}} \lambda_{2}^{\alpha_{2}}$ for $\alpha\in \mathbb{N}^{0: 2}$.
The Bernstein basis of $\mathbb P_k(T)$ is
$$
\{ \lambda^{\alpha}: = \lambda_0^{\alpha_0}\lambda_1^{\alpha_1}\lambda_2^{\alpha_2} \mid \alpha \in \mathbb T^{2}_k\}.
$$
For a subset $S\subseteq \mathbb T^{2}_k$, we define
$$
\mathbb P_k(S) := \spa \{ \lambda^{\alpha}, \alpha \in S\subseteq \mathbb T^{2}_k \}.
$$
By establishing a one-to-one mapping between the lattice node $\alpha$ and the corresponding Bernstein polynomial $\lambda^{\alpha}$, we can analyze polynomial properties through the simplicial lattice. In fact, all lattice nodes serve as interpolation nodes for the $k$-th order Lagrange element.

%The superscript ${}^n$ will be replaced by the dimension of $S$ when $S$ is a lower dimensional simplex introduced in a moment. 
%In particular,  $\mathbb T^{2}_k(\stackrel{\circ}{T})$

\subsection{Sub-simplicial lattices and distance}
We adopt the notation of~\cite{ArnoldFalkWinther2009} and define $\Delta(T)$ as the set of all sub-simplices of $T$, and $\Delta_{\ell}(T)$ as the set of all sub-simplices of dimension $\ell$, where $0\leq \ell \leq 2$. A sub-simplex $f\in \Delta_{\ell}(T)$ is determined by choosing $\ell+1$ vertices  from the $3$ vertices of $T$.  We will overload the notation $f$ for both the geometric simplex and the algebraic set of indices. As an algebraic set, $f = \{f(0), \ldots, f(\ell)\}\subseteq \{0, 1, 2\}$ is a subset of indices, and also
$$
f ={\rm Convex}(\texttt{v}_{f(0)}, \ldots, \texttt{v}_{f(\ell)}) \in \Delta_{\ell}(T)
$$
is the $\ell$-dimensional simplex spanned by the vertices $\texttt{v}_{f(0)}, \ldots, \texttt{v}_{f( \ell)}$. We also use notation $e_{ij}$ for the edge formed by vertices $\texttt{v}_i$ and $\texttt{v}_j$ for $i,j=0,1,2,i\neq j$.

For $\ell =0,1$ and $f \in \Delta_{\ell}(T)$, we let $f^* \in \Delta_{2- \ell-1}(T)$ denote the sub-simplex of $T$ opposite to $f$. When treating $f$ as a subset of $\{0, 1, 2\}$, $f^*\subset \{0,1, 2\}$ so that $f\cup f^*= \{0, 1, 2\}$, i.e., $f^*$ is the complementary set of $f$.  Geometrically,
$$
f^* ={\rm Convex}(\texttt{v}_{f^*(1)}, \ldots, \texttt{v}_{f^*(2-\ell)}) \in \Delta_{1- \ell}(T)
$$ represents the $(1- \ell)$-dimensional simplex spanned by vertices not contained in $f$. When $f$ is a vertex $\texttt{v}_i$, we simply write $f^*$ as $i^*$. Note that $f$ can be identified as the zero level set of the barycentric coordinate associated with the index set $f^*$, i.e., $f=\{x\in T\mid \lambda_i(x) = 0, \text{ for all } i\in f^*\}$.
%
%We use capital $F$ for an $(n-1)$-dimensional face of $T$ and label $F_{i}$ as the $(n-1)$-dimensional face opposite to $\texttt{v}_i$, i.e., $F_i = \{ i \}^*$ as the set and as a simplex $F_i = {\rm Convex}(\texttt{v}_{0}, \ldots, \widehat{\texttt{v}}_i, \ldots, \texttt{v}_{n})$ where $\widehat{\texttt{v}}_i$ means $\texttt{v}_i$ is removed. The sub-simplicial lattice $\mathbb T_k^{n-1}(F_0) =\{ \alpha\in \mathbb N^{1:2}, |\alpha | = k\}$ can be related to derivative $D^{\alpha}u, \alpha\in \mathbb N^{1:2}, |\alpha | = k$. 
%
%
%If $f=f_{\sigma}$ with $\sigma \in \Sigma(0:  \ell, 0: n)$, then $f^{*}=f_{\sigma^{*}}$. Again $f^*$ can be refer to $\sigma^*$.
%

Given a sub-simplex $f\in \Delta_{\ell}(T)$, through the geometric embedding $f \hookrightarrow T$, we define the prolongation/extension operator $E_f: \mathbb T^{\ell}_k(f) \to \mathbb T^{2}_k(T)$ as follows:
$$
E_f(\alpha)_{f(i)} = \alpha_{i}, i=0,\ldots, \ell, \quad \text{ and } E_f(\alpha)_j = 0, j\not\in f.
$$
For example, for 
$\alpha = ( \alpha_{0},\alpha_{1})\in \mathbb T^{1}_k(f)$, when $f = \{ 1, 2\}$, the extension $ E_f(\alpha)= (0, \alpha_{0}, \alpha_{1}) \in \mathbb T^{2}_k(T)$, and when $f = \{ 0, 2\}$, the extension $ E_f(\alpha)= (\alpha_{0}, 0, \alpha_{1}) \in \mathbb T^{2}_k(T)$. The geometric embedding $x(E_f(\alpha))\in f$ justifies the notation $\mathbb T^{\ell}_k(f)$.
With a slight abuse of notation, for a node $\alpha_f\in \mathbb T^{\ell}_k(f)$, we still use the same notation $\alpha_f\in \mathbb T^{2}_k(T)$ to denote $E_f(\alpha_f)$. Then we have the following direct decomposition
\begin{equation}\label{eq:decalpha}
\alpha = E_f(\alpha_f) + E_{f^*}(\alpha_{f^*}) = \alpha_f + \alpha_{f^*}, \text{ and } |\alpha | = |\alpha_f | + | \alpha_{f^*}|.
\end{equation}

Based on~\eqref{eq:decalpha}, we can write a Bernstein polynomial as
\begin{equation*}
\lambda^{\alpha} = \lambda_{f}^{\alpha_f}\lambda_{f^*}^{\alpha_{f^*}},
\end{equation*}
where $\lambda_{f}=\lambda_{f(0)} \ldots \lambda_{f(\ell)}\in \mathbb P_{\ell +1}(f)$ is the bubble function on $f$ and also denoted by $b_f$.

%Define $\mathbb T^{\ell}_{k, c}(f) := \{ \alpha_{f} \in \mathbb T^{\ell}_k(f),  \alpha_{f}\geq c \}$. Then it is easy to see $\mathbb T^{\ell}_{k, 1}(f) = \mathbb T^{\ell}_{k}(\stackrel{\circ}{f})$, where the latter notation indicates the lattices nodes are contained in the interior of $f$; see the red triangle in Fig.~\ref{fig:lattice}. The interior lattice $\mathbb T^{\ell}_{k}(\stackrel{\circ}{f})$
%%\mnote{ Using the same kind of notation, $\mathbb T^{\ell}_{k}(\stackrel{\circ}{f})$ and $\mathbb T^{\ell}_{k-(\ell +1)}(f)$ are confused. For example, $\mathbb T^{\ell}_{k}(\stackrel{\circ}{f})\approx \mathbb T^{\ell}_{k-(\ell +1)}(f)\approx \mathbb T^{\ell}_{k-2(\ell +1)}(\stackrel{\circ}{f})\approx \mathbb T^{\ell}_{k-3(\ell +1)}(\stackrel{\circ}{f})$} 
%is isomorphic to a simplicial lattice with a smaller degree $k-(\ell +1)$, denoted by $\mathbb T^{\ell}_{k-(\ell +1)}(f)$. The one-to-one mapping is 
%%
%%xxxx
%%
%%We introduce a smaller simplicial lattice associated to $\stackrel{\circ}{f}$ as 
%%
%%The nodes contained in the interior $\stackrel{\circ}{f}$ of $f$ form a simplicial lattice and denoted by $\mathbb T^{\ell}_{k-(\ell +1)}(f)$. It can be embedded into  $\mathbb T^{\ell}_k(f)$ as follows. 
%%
%%
%%Then
%$$
%\mathbb T^{\ell}_{k-(\ell +1)}(f) \to \mathbb T^{\ell}_{k, 1}(f): \alpha_{f} \to \alpha_{f} + 1.
%$$
%Denote by $b_f := \lambda_f$. 
The bubble polynomial of $f$ is 
$$
b_f \mathbb P_{k-(\ell +1)}(f) := \spa\{ b_f\lambda_{f}^{\alpha_{f}}: \alpha_{f} \in  \mathbb T^{\ell}_{k-(\ell +1)}(f)\}.
$$
Geometrically as the bubble polynomial space vanished on the boundary, it is generated by the interior lattice nodes only. In Fig.~\ref{fig:lattice}, $\mathbb T^{2}_{k}(\stackrel{\circ}{T})$ consists of the nodes inside the red triangle, and $\mathbb T^{1}_{k}(\stackrel{\circ}{f})$ for $f = \{1,2\}$ is in the blue trapezoid region.

%In summary, by treating $f$ as a set of indices, we can apply the operators $\cup, \cap, {}^*, \backslash$ on sets. While treating $f$ as a geometric simplex, $\partial f, \stackrel{\circ}{f}$ etc can be applied.
%
%
%
% the contained nodes in $\mathbb T^{2}_k(T)$ can be written as  By removing all zeros, we can also treat $\alpha_{\sigma} \in  \mathbb N^{0:\ell}_k(f_{\sigma})$.

%The interior lattice of $f$ is \mnote{ Should $\mathcal G_{k-(\ell +1)}(\stackrel{\circ}{f})$ be $\mathcal G_{k-(\ell +1)}(f)$?}
%$$
%\mathcal G_r(\stackrel{\circ}{f}) = \{ \alpha_{\sigma} \in  \mathbb N^{0:\ell}_k: \alpha_{\sigma}\geq 1\} \cong \mathcal G_{k-(\ell +1)}(\stackrel{\circ}{f}) = ,
%$$
%and the boundary lattice is
%$$
%\mathcal G_r(\partial f) = \{ \alpha_{\sigma} \in  \mathbb N^{0:\ell}_k: \textrm{ there exists } 0\leq i\leq \ell \textrm{ such that } \alpha_{\sigma(i)} = 0\}.
%$$

%\subsection{Distance to a sub-simplex}
\subsection{Derivative and distance}
Given $f\in \Delta_{\ell}(T)$, we define the distance of a node $\alpha\in \mathbb T_k^2$ to $f$ as
\begin{equation*}%\label{eq:dist}
\dist(\alpha, f) :=| \alpha_{f^*} | = \sum_{i\in f^*} \alpha_i.
\end{equation*}

We define the lattice tube of $f$ with radius $r$ as
$$
D(f, r) := \{ \alpha \in \mathbb T^{2}_k, \dist(\alpha,f) \leq r\},
$$
which contains lattice nodes at most $r$ distance away from $f$. Define
$$
L(f,s) := \{ \alpha \in \mathbb T^{2}_k, \dist(\alpha,f) = s\}.
$$
%which is defined as a line before but here is a subset of lattice nodes on this line. 
Then by definition, 
$$
D(f, r) = \cup_{s=0}^r L(f,s), \quad L(f,s) = L(f^*, k - s). 
$$
We have the following characterization of lattice nodes in $D(f, r)$.
\begin{lemma} \label{lm:dist}
%It holds that
%the nodes with equality $ \{ \alpha \in \mathbb T^{2}_k, |\alpha_{f^*}| = s \}$ is on a line with distance $k$ to $f_{\sigma}$.
%Consequently
For lattice node $\alpha \in \mathbb T^{2}_k$,
\begin{align*}
%\label{eq:Dfr} 
\alpha \in D(f, r)  &\iff  |\alpha_{f^*}| \leq r \iff | \alpha_{f} | \geq k - r,\\
%\label{eq:notDfr} 
\alpha \notin D(f, r) & \iff  |\alpha_{f^*}| > r \iff | \alpha_{f} | \leq k - r - 1.
\end{align*}
\end{lemma}
\begin{proof}
%In particular $f_{\sigma}$ is characterized by $\{ \lambda_{\sigma^*(1)} = \lambda_{\sigma^*(2)} =\ldots \lambda_{\sigma^*(n-\ell)}= 0 \}$.
By definition of $\dist(\alpha, f)$ and the fact $| \alpha_{f} |  + | \alpha_{f^*} | = k$.
\end{proof}

For each vertex $\texttt{v}_i\in \Delta_0(T)$ and an integer $0\leq r\leq k$, the tube
$$
D(\texttt{v}_i, r) =  \{ \alpha \in \mathbb T^{2}_k(T), |\alpha_{i^*}| \leq r \},
$$
 is isomorphic to a simplicial lattice $\mathbb T_{r}^2$ of degree $r$. In Fig.~\ref{fig:lattice}, $D(\texttt{v}_0, 3)$ consists of lattice nodes in the green triangle which itself can be treated as a smaller simplicial lattice $\mathbb T_{3}^2$. For an edge $e\in \Delta_{1}(T)$, $D(e,r)$ is a trapezoid of height $r$ with base $e$. 
%In general for $f\in \Delta_{\ell}(T)$, the hyper plane  $L(f, r)$ will cut the simplex $T$ into two parts, and $D(f,r)$ is the part containing $f$.

Recall that in~\cite{ArnoldFalkWinther2009} a smooth function $u$ is said to vanish to order $r$ on $f$ if $D^{\beta} u|_f = 0$ for all $\beta \in \mathbb N^{1:2}$ satisfying $|\beta | < r$. The following result shows that the vanishing order $r$ of a Bernstein polynomial $\lambda^{\alpha}$ on a sub-simplex $f$ is the distance $\dist(\alpha, f)$.
%The distance of a node $\alpha$ to a sub-simplex $f$ can be used to control the derivative of the corresponding Bernstein polynomial.
\begin{lemma}\label{lm:derivative}
Let $f\in \Delta_{\ell}(T)$ be a sub-simplex of $T$. For $\alpha\in \mathbb T^{2}_k, \beta \in \mathbb N^{1:2}$, and $|\alpha_{f^*}| > |\beta|$, i.e., $\dist(\alpha, f) > |\beta|$, then
\begin{equation*}%\label{eq:distorder}
D^{\beta} \lambda^{\alpha}|_{f} = 0, \quad \text{ when }\dist(\alpha, f) > |\beta|.
\end{equation*}
\end{lemma}
\begin{proof}
For $\alpha\in \mathbb T^{2}_k$, we write $\lambda^{\alpha} = \lambda_{f}^{\alpha_f}\lambda_{f^*}^{\alpha_{f^*}}$. When $|\alpha_{f^*}| > |\beta|$, the derivative $D^{\beta} \lambda^{\alpha}$ will contain a factor $\lambda_{f^*}^{\gamma}$ with $\gamma\in \mathbb N^{1:(2-\ell)},$ and $|\gamma| = |\alpha_{f^*}| - |\beta| > 0$. Therefore $D^{\beta} \lambda^{\alpha}|_{f} = 0$ as $\lambda_{i}|_f = 0$ for $i\in f^*$.
%Otherwise $|\alpha_{f^*}| \leq |\beta|$, there will be a term in $D^{\beta} \lambda^{\alpha}$ which does not contain $\lambda_{f^*}$
\end{proof}

\subsection{Derivatives at vertices}
Consider a function $u\in C^{m}(\Omega)$. The set of derivatives of order up to $m$ can be written as
$$
\{ D^{\beta} u, \beta \in \mathbb N^{1:2}, |\beta |\leq m\}.
$$
Notice that the multi-index $\beta \in \mathbb N^{1:2}$ is not in $\mathbb N^{0:2}$. We can add a component with value $m - |\beta|$ to form a simplicial lattice $\mathbb T^{2}_m$ of degree $m$, which can be used to determine the derivatives at that vertex.

%On the other hand, given a lattice node $\alpha \in D(\texttt{v}_i,m)$, we can define $\beta = \alpha_{i^*}$

\begin{lemma}\label{lem:vertexunisolvence}
Let $i\in \{0,1,2\}$. The polynomial space
\begin{equation*}%\label{eq:DvPk}
\mathbb P_k(D(\texttt{v}_i, m)): = \spa \left \{ \lambda^{\alpha}, \alpha  \in \mathbb T^{2}_k, \dist(\alpha, \texttt{v}_i) =  |\alpha_{i^*} | \leq m \right \},
\end{equation*}
is uniquely determined by the DoFs
\begin{equation}\label{eq:DvDoF}
\{ D^{\beta} u (\texttt{v}_i), \beta \in \mathbb N^{1:2}, |\beta | \leq  m\}.
\end{equation}
\end{lemma}
\begin{proof}
Without loss of generality, consider $\texttt{v}_0$. Define map $\alpha = (\alpha_0, \alpha_1, \alpha_2) \to \beta = (\alpha_1, \alpha_2)$ which induces a one-to-one map from $D(\texttt{v}_0, m) = \left \{ \alpha  \in \mathbb T^{2}_k, \alpha_1 + \alpha_2 \leq m \right \}$ to $\{\beta \in \mathbb N^{1:2}, |\beta | \leq  m\}$. So the dimension of $\mathbb P_k(D(\texttt{v}_i, m))$ matches the number of DoFs~\eqref{eq:DvDoF}. It suffices to show that for $u\in \mathbb P_k(D(\texttt{v}_0, m))$ if DoFs~\eqref{eq:DvDoF} vanish, then $u = 0$.

Recall the multivariate calculus result
\begin{equation}\label{eq:calculus}
D^{\beta}(x_1^{\alpha_1}x_2^{\alpha_2}) = \beta! \delta (\alpha, \beta) \textrm{ for }\alpha, \beta\in \mathbb N^{1:2}, |\alpha |=|\beta |=r\geq 0,
\end{equation}
where $\delta(\alpha,\beta)$ is the Kronecker delta function.
When the triangle $T$ is the reference triangle, $\texttt{v}_0$ is the origin and $\lambda_1 = x_1, \lambda_2 = x_2$. So we conclude that the homogenous polynomial space
$ \spa \left \{ x_1^{\alpha_1}x_2^{\alpha_2}, \alpha  \in \mathbb N^{1:2}, |\alpha |=r \right \}$ is determined by DoFs $\{ D^{\beta} u (\texttt{v}_0), \beta \in \mathbb N^{1:2}, |\beta | = r\}.$ Running $r=0,1,\ldots, m$, we then finish the proof when the triangle is the reference triangle. 

%Obviously the dimensions match. Indeed, a one-to-one mapping is from $\alpha_{i^*}$ to $\beta$. 

%Without loss of generality, we assume $i=0$. 
For a general triangle, instead of changing to the reference triangle, we shall use the barycentric coordinate. 
Clearly $\{\nabla \lambda_1, \nabla \lambda_2\}$ forms a basis of $\mathbb R^2$.
Choose another basis $\{l^1, l^{2}\}$ of $\mathbb R^2$, being dual to $\{\nabla \lambda_1, \nabla \lambda_2\}$, i.e., $\nabla \lambda_{i}\cdot l^j = \delta_{i,j}$ for $i,j=1,2$. Indeed $l^i$ is the edge vector $e_{0i}$ as $\nabla \lambda_i$ is orthogonal to $e_{0i}$ for $i=1,2$. We can express the derivatives in this non-orthogonal basis  and denote by $D^{\beta}_nu:=\frac{\partial^{|\beta|} u}{\partial (l^1)^{\beta_1}\partial(l^2)^{\beta_2}}$ with $\frac{\partial}{\partial l^i}=l^i\cdot\nabla$.
By the duality $\nabla \lambda_{i}\cdot l^j = \delta_{i,j}$, $i,j=1,2$, we have the generalization of \eqref{eq:calculus}
\begin{equation}\label{eq:Dn}
D^{\beta}_n (\lambda_{1}^{\alpha_1}\lambda_{2}^{\alpha_2}) = \beta!\delta( \alpha, \beta) \quad \textrm{ for }\alpha, \beta\in \mathbb N^{1:2}, |\alpha |=|\beta |=r.
\end{equation}
%When $T$ is the reference triangle, $\lambda_{0^*}^{\alpha} = x_1^{\alpha_1}x_2^{\alpha_2}$, $D_n^{\beta} = D^{\beta}$, and~\eqref{eq:Dn} is the standard calculus result.  

By the chain rule, it is easy to show that the vanishing $\{D^{\beta}_nu, \beta \in \mathbb N^{1:2}, |\beta | \leq m\}$ is equivalent to the vanishing $\{D^{\beta} u, \beta \in \mathbb N^{1:2}, |\beta | \leq m\}$. So we will work with $D^{\beta}_n$. 

A Bernstein basis of $\mathbb P_k(D(\texttt{v}_0, m))$ is given by $\{ \lambda_{0}^{k - |\alpha|}\lambda_{1}^{\alpha_1}\lambda_{2}^{\alpha_2}, \alpha\in \mathbb N^{1:2}, |\alpha|\leq m \}$. 
%First we show that the polynomial space
%$
%\spa \left \{ \lambda^{\alpha}, \alpha  \in \mathbb T^{2}_k, |\alpha_{0^*} |=r\right \}
%$
%is uniquely determined by the DoFs
%$
%\{D^{\beta}_n u (\texttt{v}_0), \beta \in \mathbb T_r^{1:n}\}.
%$ 
%Assume $u=\sum_{\alpha \in \mathbb T_r^{1:n}}c_{\alpha}\lambda_{0}^{k - r}\lambda_{0^*}^{\alpha}$ with $c_{\alpha}\in
%\mathbb R$ and $D^{\beta}_n u (\texttt{v}_0)=0$ for all $\beta \in \mathbb T_r^{1:n}$.
%% By the duality $\nabla \lambda_{i}\cdot n^j = \delta_{i,j}$, $i,j=1,2$,
%Apply~\eqref{eq:Dn} to get
%$$
%D^{\beta}_n u (\texttt{v}_0) = \beta!  c_{\beta}  \lambda_0^{k-r}(\texttt{v}_0)=\beta!  c_{\beta} \quad\textrm{ for } \beta \in \mathbb T_r^{1:n},
%$$
%which implies $c_{\beta}=0$, and then $u=0$.
%
Assume $u=\sum_{\alpha \in \mathbb N^{1:2}\atop |\alpha| \leq  m}c_{\alpha}\lambda_{0}^{k - |\alpha|}\lambda_{1}^{\alpha_1}\lambda_{2}^{\alpha_2}$ with $c_{\alpha}\in
\mathbb R$ and $D^{\beta} u (\texttt{v}_0)=0$ for all $\beta \in \mathbb N^{1:2}$ satisfying $ |\beta | \leq  m$. We shall prove $c_{\alpha} = 0$ by induction. 

For $|\alpha | = 0$, as $c_{(0,0)} = u(\texttt{v}_0)= 0$, we conclude $c_{(0,0)}=0$. Assume $c_{\alpha}=0$ for all $\alpha \in \mathbb N^{1:2}$ satisfying $|\alpha|\leq r-1$, i.e., $u=\sum_{\alpha \in \mathbb N^{1:2}\atop r\leq|\alpha| \leq  m}c_{\alpha}\lambda_{0}^{k - |\alpha|}\lambda_{1}^{\alpha_1}\lambda_{2}^{\alpha_2}$. By Lemma~\ref{lm:derivative}, the derivative $D^{\beta}(\lambda_{0}^{k - |\alpha|}\lambda_{1}^{\alpha_1}\lambda_{2}^{\alpha_2})$ vanishes at $\texttt{v}_0$ for all $\beta\in\mathbb N^{1:2}$ satisfying $|\beta|<|\alpha|$. Hence, for $|\beta| = r$, using~\eqref{eq:Dn},
$$
D_n^{\beta}u(\texttt{v}_0)=D_n^{\beta}\Big(\sum_{\alpha \in \mathbb N^{1:2}, |\alpha|=r}c_{\alpha}\lambda_{0}^{k - r}\lambda_{1}^{\alpha_1}\lambda_{2}^{\alpha_2}\Big )(\texttt{v}_0)=\beta!  c_{\beta} = 0,
$$
which implies $c_{\beta}=0$ for all $\beta \in \mathbb N^{1:2}$, $|\beta | = r$.
Running $r=1,2, \ldots, m$, we conclude $u = 0$.
\end{proof}

%Here is a summary of the proof. The DoF-Function matrix
%$(D_n^{\beta}(\lambda_{0}^{k - |\alpha|}\lambda_{0^*}^{\alpha})(\texttt{v}_0))_{\beta, \alpha}$ is block lower triangular where the lattice nodes are sorted by their length. Then if each block matrix on the diagonal ($|\alpha| = |\beta| = r$) is invertible, the whole matrix is invertible which is equivalent to the unisolvence. For the block diagonal matrix, we switch to the derivative $D^{\beta}_n$ which is dual to the Bernstein form of polynomials. 

\subsection{Normal derivatives on edges}
Given an edge $e$, we identify lattice nodes to determine the normal derivative up to order $m$
$$
\left \{ \frac{\partial^{\beta} u}{\partial n_e^{\beta}} \mid _e , 0\leq \beta \leq m\right \}.
$$
By Lemma~\ref{lm:derivative}, if the lattice node is $r^{e}+1$ away from the edge, then the corresponding Bernstein polynomial will have vanishing normal derivatives up to order $r^{e}$.
%\begin{lemma}
%For $u = \lambda^{\alpha}, \alpha \in \mathbb N^{0:2}_k$, if $\alpha \notin D(e, k)$, then
%$$
%\frac{\partial^{k} u}{\partial n_e^{k}} \mid _e = 0.
%$$
%\end{lemma}
%\begin{proof}
%The condition $\alpha \notin D(e, k)$ means $\dist(\alpha, e)>k$. Without loss of generality, we take $e = e_{0,1}$. Then $u = \lambda_0^{\alpha_0}\lambda_1^{\alpha_1}\lambda_2^{\alpha_2}$ and $\alpha_2 > s$. Notice that edge $e_{0,1}$ is on the zero level set  of $\lambda_2$, i.e., $\lambda_2|_{e_{0,1}} = 0$. The derivative $D^k u$ will contain a factor $\lambda_2^{\alpha_2 - k}$ and thus $\alpha_2 > s$ implies $D^ku|_e = 0$.
%\end{proof}

We have used lattice nodes $D(\Delta_0(e), r^{\texttt{v}}):= \cup_{\texttt{v}\in \Delta_0(e)} D( \texttt{v}, r^{\texttt{v}})$ to determine the derivatives at vertices.
We will use $D(e,r^{e})\backslash  D(\Delta_0(e), r^{\texttt{v}})$ for the normal derivative.

\begin{lemma}\label{lem:edgeunisolvence2d}
Let $r^{\texttt{v}}\geq r^{e} \geq 0$ and $k \geq 2 r^{\texttt{v}}+1$. Let $e\in \Delta_1(T)$ be an edge of a triangle $T$.
The polynomial function space $\mathbb P_{k}( D(e, r^{e})\backslash D(\Delta_0(e), r^{\texttt{v}}))$ is determined by DoFs
\begin{equation*}%\label{eq:normalDoF2D}
\int_e \frac{\partial^{\beta} u}{\partial n_e^{\beta}}  \, \lambda_e^{\alpha} \dd s \quad \alpha \in \mathbb T^{1}_{k - 2(r^{\texttt{v}}+1) + \beta}, \beta = 0,1,\ldots, r^{e}.
\end{equation*}
\end{lemma}
\begin{proof}
Without loss of generality, take $e = e_{0,1}$. By definition $ D(e,r^{e}) = \cup_{i=0}^{r^{e}} L(e,i)$, where recall that
$$
L(e,i) = \{ \alpha \in \mathbb T^{2}_k, \dist(\alpha,e) = i\} = \{ \alpha \in \mathbb T^{2}_k, \alpha_0+\alpha_1 = k - i\}
$$ 
consists of lattice nodes parallel to $e$ and with distance $i$. Define the map $(\alpha_0, \alpha_1, \alpha_2) \to (\alpha_0, \alpha_1)$ which is one-to-one between $L(e,i)$ and $\mathbb T^{1}_{k-i}(e)$.

Now we use the requirement $\alpha \notin D(\Delta_0(e), r^{\texttt{v}})$ to figure out the bound of the components. Using Lemma~\ref{lm:dist}, we derive from $\dist (\alpha, \texttt{v}_0) > r^{\texttt{v}}$ that $ \alpha_0 < k - r^{\texttt{v}}$. Together with $\alpha_0+\alpha_1 = k - i$, we get the lower bound $\alpha_1\geq r^{\texttt{v}} - i+1$. Similarly $\alpha_0\geq r^{\texttt{v}} - i+1$.
Therefore
$$
L(e,i) \backslash D(\Delta_0(e), r^{\texttt{v}}) =\{ (\alpha_0, \alpha_1, i), \alpha_0+\alpha_1 = k - i, \min\{\alpha_0, \alpha_1\} \geq  r^{\texttt{v}} - i+1\}.
$$
Define the one-to-one mapping 
\begin{align*}
\mathbb T^{1}_{k - 2(r^{\texttt{v}}+1) + i}&\to L(e,i) \backslash D(\Delta_0(e), r^{\texttt{v}}), \\
(\alpha_0, \alpha_1) &\mapsto (\alpha_0 + (r^{\texttt{v}} - i+1), \alpha_1+ (r^{\texttt{v}} - i+1), i).
\end{align*}
%Vice verse, given $\alpha \in \mathbb T^{1}_{k - 2(r^{\texttt{v}}+1) + i}(e)$, the node $(\alpha_0 + (r^{\texttt{v}} - i+1), \alpha_1 + (r^{\texttt{v}} - i+1), i)\in L(e,i) \backslash D(\Delta_0(e), r^{\texttt{v}})$. 
%
With the help of this one-to-one mapping, we shall prove the polynomial function space $\mathbb P_{k}(L(e,i) \backslash D(\Delta_0(e), r^{\texttt{v}}))$ is determined by DoFs
\begin{equation}\label{eq:normalDoF2Di}
\int_e \frac{\partial^{i} u}{\partial n_e^{i}}  \, \lambda_e^{\alpha} \dd s \quad \alpha \in \mathbb T^{1}_{k - 2(r^{\texttt{v}}+1) + i}.
\end{equation}
Take a $u= \sum_{\alpha \in \mathbb T^{1}_{k - 2(r^{\texttt{v}}+1) + i}}c_{\alpha}\lambda_e^{\alpha + r^{\texttt{v}} - i+1}\lambda_2^{i}  \in \mathbb P_{k}(L(e,i) \backslash D(\Delta_0(e), r^{\texttt{v}}))
$ with coefficients $c_{\alpha}\in\mathbb R$. By the chain rule and the fact $\lambda_2|_e=0$, in the non-zero terms of $\frac{\partial^{i} u}{\partial n_e^{i}}\mid_e$, the derivative in $\frac{\partial^{i} u}{\partial n_e^{i}}|_e$ will all apply to $\lambda_2^i$, so
$$
\frac{\partial^{i} u}{\partial n_e^{i}}\mid_e=i!(n_e\cdot\nabla\lambda_2)^i\lambda_e^{r^{\texttt{v}} - i+1}\sum_{\alpha \in \mathbb T^{1}_{k - 2(r^{\texttt{v}}+1) + i}}c_{\alpha}\lambda_e^{\alpha}|_e.
$$
Noting that $n_e\cdot\nabla\lambda_2$ is a constant and the bubble polynomial $\lambda_e^{r^{\texttt{v}} - i+1}$ is always positive in the interior of $e$, the vanishing DoF~\eqref{eq:normalDoF2Di} means $c_{\alpha} = 0$ for all $\alpha\in \mathbb T^{1}_{k - 2(r^{\texttt{v}}+1) + i}$. 

It follows from Lemma~\ref{lm:derivative} that $\frac{\partial^{\beta}}{\partial n_e^{\beta}}(\lambda_{e}^{\alpha}\lambda_{2}^{i})|_e=0$ for $\alpha\in\mathbb T^{1}_{k-i}(e)$ and $0\leq \beta<i\leq r^e$. That is the matrix 
$$
\left (\frac{\partial^{\beta}}{\partial n_e^{\beta}}(\lambda_{e}^{\alpha}\lambda_{2}^{i})|_e\right )_{0\leq \beta \leq r^e, 0\leq i \leq r^e,\alpha \in \mathbb T^1_{k - 2(r^{\texttt{v}}+1) + i}(e)}
$$
is lower block triangular as follows.
$$
\renewcommand{\arraystretch}{1.35}
\begin{array}{cc}
\begin{array}{c}
 \; \beta \backslash \;   i\\
 \alpha 
\end{array}
 &  
%\begin{array}{ccccc}
%0 & 1 & \ldots	& r^e-1 & r^e
%\end{array}
%\smallskip
%\\ 
\begin{array}{ccccc}
0 & 1 & \ldots	& r^e-1 & r^e\\
\mathbb{T}_{k-2(r^{\texttt{v}}+1)}^1 & \mathbb{T}_{k-2(r^{\texttt{v}}+1)+1}^1 & \cdots	& \mathbb{T}_{k-2(r^{\texttt{v}}+1)+r^e-1}^1 & \mathbb{T}_{k-2(r^{\texttt{v}}+1)+r^e}^1
\end{array}
\medskip
\\
\begin{array}{c}
0 \\ 1 \\ \vdots \\ r^e-1 \\ r^e
\end{array} 
& \left(
\begin{array}{>{\hfil$}m{1.5cm}<{$\hfil}|>{\hfil$}m{1.9cm}<{$\hfil}|>{\hfil$}m{0.5cm}<{$\hfil}|>{\hfil$}m{2.5cm}<{$\hfil}|>{\hfil$}m{1.9cm}<{$\hfil}}
\square & 0 & \cdots	& 0 & 0 \\
\hline
\square & \square & \cdots	& 0 & 0 \\
\hline
\vdots & \vdots & \ddots	& \vdots & \vdots \\
\hline
\square & \square & \cdots	& \square & 0 \\
\hline
\square & \square & \cdots	& \square& \square 
\end{array}
\right)
\end{array}
$$

Since we have proved each block matrix is invertible, then the whole lower block triangular matrix is invertible which is equivalent to the unisolvence.  
%
%
%Applying the same argument in the proof of Lemma~\ref{lem:vertexunisolvence},
%it follows from Lemma~\ref{lm:derivative} that matrix 
%$(\frac{\partial^{\beta}}{\partial n_e^{\beta}}(\lambda_{e}^{\alpha}\lambda_{2}^{i})|_e)_{\beta, i}$ is. \LC{ more explanation. } 
%
%\XH{ This means  is lower block triangular, which is invertible.}
%Consequently $u = 0$. 
%\LC{Using the previous argument for Hermite. Lower triangular structure indexed by $\beta$ and $\alpha_{\texttt{v}}$. Write $u = \lambda_e^{\alpha_e} \lambda_{\texttt{v}}^{\alpha_{\texttt{v}}}$ where $\texttt{v} = e^*$. Then only consider the block diagonal $\beta = \alpha_{\texttt{v}}$ for which the normal derivative will contain $\lambda_e^{\alpha_e}$ only. Then use the range.}
%Let $w = \frac{\partial^i u}{\partial n^i}\in \mathbb P_{k - i}(T)$ for $i=0,1\ldots, r^{e}$. On edge $e$, by the Hermite interpolation~\eqref{eq:1DHerm1}-\eqref{eq:1DHerm2}, $w|_e$ is uniquely determined by
%\begin{align*}
%\partial_e^{(j)} w(\texttt{v}) & \quad  j = 0,1,\ldots, r^{\texttt{v}} - i, \texttt{v}\in \Delta_0(e),\\
%\int_e w \, \lambda_e^{\alpha_e}\dd s & \quad \alpha_e \in \mathbb N^{0:1}_{k - 2(r^{\texttt{v}}+1) + i} (\stackrel{\circ}{e}).
%\end{align*}
%For $u\in \spa\{ \lambda^{\alpha}, \alpha \in D(e,k)\backslash D(\Delta_0(e), r^{\texttt{v}}) \}$, as $\dist(\alpha, \Delta_0(e)) > r^{\texttt{v}}$, $D^{\beta}u(\texttt{v}) = 0$ for all $\beta\leq r^{\texttt{v}}$ and $\texttt{v}\in \Delta_0(e)$. Therefore the vanishing DoF implies $w|_e = 0$.
\end{proof}

%Geometrically we push all lattice nodes in $D(e,r^{e})\backslash D(\Delta_0(e), r^{\texttt{v}})$ to the edge to determine normal derivatives on $e$ up to order $r^{e}$.

\subsection{Geometric decompositions of the simplicial lattice}
%The requirement $r^{e}\leq r^{\texttt{v}}$ in Lemma~\ref{lem:edgeunisolvence2d} is due to the fact that the smoothness of the normal derivative is less than or equal to that of the vertices. 
Inside a triangle, a vertex will be shared by two edges and to have enough lattice nodes for each edge,  $r^{\texttt{v}}\geq 2r^{e}$ is required; see Fig.~\ref{fig:2Ddec}(b).

%\begin{figure}[htbp]
%\begin{center}
%%\includegraphics[width=4in]{figures/2Ddecomposition.png}
%\includegraphics[width=2.5in]{figures/2DC1element.png}
%\caption{A .}
%\label{fig:2Ddec}
%\end{center}
%\end{figure}

\begin{lemma}\label{lem:geodecomp2d}
Let $r^{e} = m\geq -1$, $r^{\texttt{v}}\geq \max\{2r^{e},-1\}$, and nonnegative integer $k\geq 2r^{\texttt{v}}+1\geq 4m+1$. Let $T$ be a triangle. Then it holds that
\begin{align}\label{eq:smoothdec2d}
  \mathbb T^{2}_k(T) = S_0(T) \Oplus S_1(T) \Oplus S_2(T),
\end{align}
where
\begin{align*}
S_0(T) &:=  D( \Delta_0(T), r^{\texttt{v}}), \\
S_1(T) &:= \Oplus_{e\in \Delta_{1}(T)}\left ( D(e,r^{e}) \backslash S_0(T)\right ),\\
S_2(T) &:= \mathbb T^{2}_k(T) \backslash (S_0(T) \oplus S_1(T)),
\end{align*}
with cardinality 
\begin{align*}
|S_0(T)| & = 3 { r^{\texttt{v}}+ 2 \choose 2}, \\
|S_1(T)| & = 3 \sum_{i=0}^{r^{e}}(k-1 - 2r^{\texttt{v}}+i), \\
|S_2(T)| & = { k+2 \choose 2} - |S_0(T)| - |S_1(T)|.
\end{align*}
This leads to the decomposition of the polynomial space
\begin{align}\label{eq:PkCmdec}
\mathbb P_k(T)&= \mathbb P_k(S_0(T)) \Oplus \mathbb P_k(S_1(T)) \Oplus \mathbb P_k(S_2(T)).
%\Oplus_{\ell = 0}^2 \spa \{ \lambda^{\alpha}, \alpha\in S_{\ell}(T)\}\\
%&= \Oplus_{i=0}^2\spa\{ \lambda^{\alpha}, \alpha\in \mathbb N^{0:2}_k, |\alpha_{i^*}|\leq s_{0}\}  \\
%&\oplus  \\
%&\oplus b_T^{m+1}\mathbb P_{k-3(m+1)}^{0,0}(T),
\end{align}
%Furthermore, we have characterization 
%\begin{equation}\label{eq:PkS1}
%\mathbb P_k(S_1(T)) = \Oplus_{e\in\Delta_{1}(T)}\Oplus_{i=0}^mb_T^ib_e^{r^{\texttt{v}}+1-2i}\mathbb P_{k-2(r^{\texttt{v}}+1)+i}(e),
%\end{equation}
%and
%\begin{equation}\label{eq:PkS2}
%\mathbb P_k(S_2(T))  =  b_T^{r^e+1}\ \stackrel{\circ}{\mathbb P}_{k-3(r^{e}+1)}(\bs r),
%\end{equation}
%where 
%\begin{equation}\label{eq:Pcirc}
%\stackrel{\circ}{\mathbb P}_{k-3(r^{e}+1)}(\bs r) := {\rm span} \{ \lambda^{\alpha}, \alpha \in \mathbb T^{2}_{k-3(r^{e}+1)}, \alpha \leq k - r^{\texttt{v}} - r^{e} - 2\}.
%\end{equation}
\end{lemma}
\begin{proof}
As $k\geq 2r^{\texttt{v}}+1$, the sets $\{D(\texttt{v}, r^{\texttt{v}}), \texttt{v}\in \Delta_0(T)\}$ are disjoint. 
%As $\dist(\texttt{v},e) = k \geq 2r^{\texttt{v}}+1 > r^{e}$ for $\texttt{v}\not\in \Delta_0(e)$, $D(e,r^{e}) \backslash D( \Delta_0(T), r^{\texttt{v}}) = D(e, r^{e}) \backslash D( \Delta_0(e), r^{\texttt{v}})$. 

We then show that the sets $\{D(e,r^{e}) \backslash D( \Delta_0(e), r^{\texttt{v}}), e\in \Delta_1(T)\}$ are disjoints.
%, where $S_0(e) =  D( \Delta_0(e), r^{\texttt{v}})$. Take two edges $e_{01}$ and $e_{02}$. 
A node $\alpha \in D(e_{01}, r^{e})$ implies $\alpha_2\leq r^{e}$ and $\alpha \in D(e_{02}, r^{e})$ implies $\alpha_1\leq r^{e}$. Therefore $|\alpha_{0^*}| = \alpha_1+ \alpha_2 \leq 2r^{e}\leq r^{\texttt{v}}$, i.e., $D(e_{01}, r^{e})\cap D(e_{02}, r^{e})\subseteq D(\texttt{v}_0, r^{\texttt{v}})$. Repeat the argument for each pair of edges to conclude $\{D(e,r^{e}) \backslash D( \Delta_0(e), r^{\texttt{v}}), e\in \Delta_1(T)\}$ are disjoint. 

For a given edge $e$, the vertex $e^*$ is opposite to $e$ and $L(e^*, r^{\texttt{v}})= L(e, k - r^\texttt{v})$. As $k - r^\texttt{v}\geq r^\texttt{v}+1> r^e$, we conclude $D(e, r^{e}) \cap D(e^*, r^{\texttt{v}}) = \varnothing$ and consequently $D(e,r^{e}) \backslash D( \Delta_0(T), r^{\texttt{v}}) = D(e, r^{e}) \backslash D( \Delta_0(e), r^{\texttt{v}})$.

Then decompositions~\eqref{eq:smoothdec2d} and~\eqref{eq:PkCmdec} follow.
%
%As a consequence $$\mathbb P_k(T)= \Oplus_{\ell = 0}^2 \spa \{ \lambda^{\alpha}, \alpha\in S_{\ell}(T)\}.$$ 
%
%To show \eqref{eq:PkS1}, we notice that, for a lattice node $\alpha \in D(e_{01},r^{e}) \backslash D( \Delta_0(e_{01}), r^{\texttt{v}})$, it satisfies the constraint
%$$
%\alpha_2 \leq r^{e}, \; \alpha_0+\alpha_2 > r^{\texttt{v}}, \; \alpha_1+\alpha_2 > r^{\texttt{v}}, \;\alpha_ 0 + \alpha_1 + \alpha_2 = k.
%$$
%We let $\alpha_2 = i$ for $i=0,1,\ldots, r^{e} = m$. Then 
%%\mnote{ check the index range. seems not correct. e.g. $i=0, k- 2(r^{\texttt{v}}+1)$ could be negative. 
%%
%%It's OK. Because $k- 2(r^{\texttt{v}}+1)\geq-1$, $k- 2(r^{\texttt{v}}+1)=-1$ means there is no moment dof.
%%}
%$$
%\lambda_0^{\alpha_0}\lambda_1^{\alpha_1}\lambda_2^{\alpha_2} = \lambda^{i}(\lambda_0\lambda_1)^{r^{\texttt{v}} - 2i +1} \lambda_{0}^{\alpha_ 0 -  (r^{\texttt{v}} + 1)+i}\lambda_{1}^{\alpha_1 - (r^{\texttt{v}} + 1)+i} = b_T^i b_{e}^{r^{\texttt{v}} - 2i+1}\lambda_e^{\alpha_e}
%$$
%with $\alpha_e\in \mathbb T^{0:1}_{k-2(r^{\texttt{v}}+1)+i}$, which gives~\eqref{eq:PkS1}. 
%For a node $\alpha\in S_2(T)$, it satisfies the constraint
%$$
%\alpha_ 0 + \alpha_1 + \alpha_2 = k,\; \alpha > r^{e}, \; \alpha < k-r^{\texttt{v}}.
%$$
%Then set $\tilde \alpha = \alpha - (r^{e} +1)$. We can write
%$$
%\lambda^{\alpha} = \lambda^{r^{e}+1} \lambda^{\tilde \alpha}, \; |\tilde \alpha | = k - 3(r^{e}+1), \; \tilde \alpha\leq k-r^{\texttt{v}} - r^{e} - 2,
%$$
%which leads to~\eqref{eq:PkS2}.
%
%The cardinalities $|S_0|$ and $|S_2|$ are straight forward and $|S_1|$ is given by \eqref{eq:PkS1}.
\end{proof}

\begin{figure}[htbp]
\subfigure[The geometric decomposition of a Hermite element: $m=0, r^{e} = 0, r^{\texttt{v}} = 1, k = 8$.]{
\begin{minipage}[t]{0.425\linewidth}
\centering
\includegraphics*[width=4.75cm]{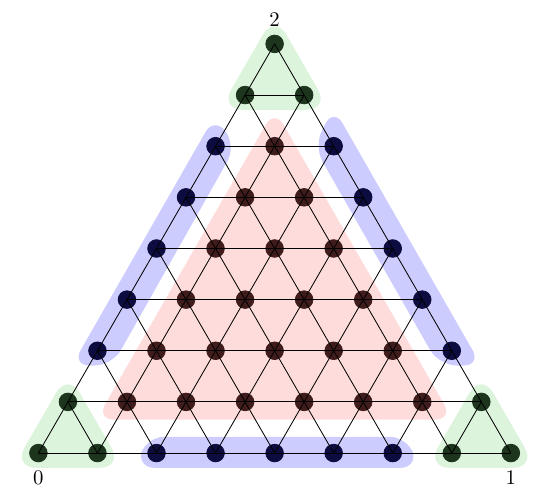}
\end{minipage}}%% 
\quad \quad
\subfigure[The geometric decomposition of a $C^1$ element: $m=1, r^{e} = 1, r^{\texttt{v}} = 2, k = 8$.]
{\begin{minipage}[t]{0.425\linewidth}
\centering
\includegraphics*[width=4.75cm]{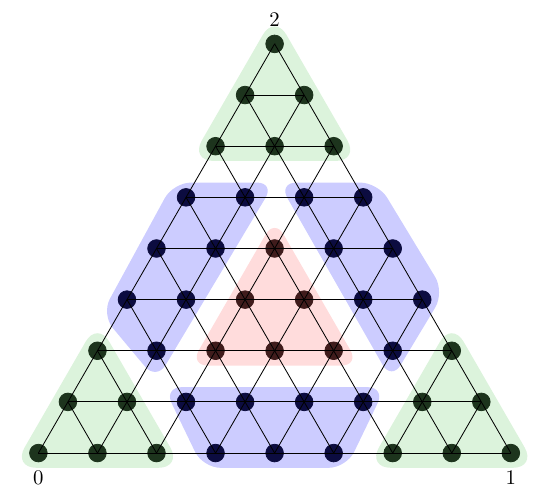}
\end{minipage}}
\caption{Comparison of the geometric decompositions of a two-dimensional Hermite element and a $C^1$-conforming element.}
\label{fig:2Ddec}
\end{figure}

Denote by 
$$
\mathbb B_k(\bs r) = \mathbb B_k(\bs r; T) := \mathbb P_k(S_2(T)) = {\rm span} \{ \lambda^{\alpha}, \alpha \in S_2(T) \},
$$
and call it the polynomial bubble space, which will play an important role in our construction of finite element de Rham complexes. Polynomials in $\mathbb B_k(\bs r)$ will have vanishing derivatives up to order $r^e$, and more precisely
\begin{align*}
\mathbb B_{k}(\boldsymbol r,T) =\{u\in\mathbb P_k(T):&\, \nabla^ju \textrm{ vanishes at all vertices of $T$ for $j=0,\ldots, r^{\texttt{v}}$}, \\
& \textrm{ and $\nabla^ju$ vanishes on all edges of $T$ for $j=0,\ldots, r^{e}$}\}.
\end{align*}
%and
%%and 
%Therefore 
%and 
%$$
%\dim \mathbb B_k(\bs r) = \dim \stackrel{\circ}{\mathbb P}_{k-3(r^{e}+1)}(\bs r) = |S_2(T)|.
%$$
%where
%\begin{align*}
%\mathbb P_{k-3(m+1)}^{0,0}(T):=\spa\big\{\lambda^{\alpha} \mid \alpha\in\mathbb T^{2}_{k-3(m+1)}, \alpha \leq k-r^{\texttt{v}}-m-2\big\}.
%\end{align*}
%When $r^{\texttt{v}} = 2r^{e}$ or $r^{\texttt{v}} = 2r^{e}+1$, $\stackrel{\circ}{\mathbb P}_{k-3(r^{e}+1)}(\bs r) = \mathbb P_{k-3(r^e+1)}(T)$ as the constraint automatically holds. 

\begin{lemma}
Let $r^{e} \geq -1$ and $r^{\texttt{v}}\geq \max\{2r^{e},-1\}$. Then
\begin{enumerate}
\item $\dim \mathbb B_k(\bs r)\geq 1$, when $k\geq \max\{2r^{\texttt{v}}+1, 3r^e+3, (r^e+3)[r^{\texttt{v}}=0]\}$;
\item $\dim \mathbb B_k(\bs r)\geq 3$, when $k\geq \max\{2r^{\texttt{v}}+1, 3r^e+4, 4[r^e=-1, r^{\texttt{v}}=1], 2[r^e=-1, r^{\texttt{v}}=0]\}$.
\
\end{enumerate}
\end{lemma}
\begin{proof}
The first statement has been proved in \cite{Chen;Huang:2022FEMcomplex3D}. We can prove the second statement by verifying the following inequality directly
$$
\dim \mathbb B_k(\bs r)={k-3r^e-1 \choose 2}-3{r^{\texttt{v}}-2r^e \choose 2}\geq3.
$$
\end{proof}

%\begin{remark}\rm
%Such simplification is not needed in implementation. Distance to a vertex or an edge is computable and a logic array can be used to represent $S_{\ell}(T)$.
%\end{remark}
\subsection{Smooth finite elements in two dimensions}
We are in the position to present $C^m$-finite elements on a triangulation. 
\begin{theorem}\label{thm:Cr2dfemunisolvence}
 Let $r^{e}=m \geq -1$, $r^{\texttt{v}}\geq\max\{2r^{e},-1\}$, and nonnegative integer $k\geq 2r^{\texttt{v}}+1$. Let $T$ be a triangle. The shape function space $\mathbb P_{k}(T)$ is determined by the DoFs
\begin{subequations}\label{eq:C12d}
\begin{align}
\label{eq:C12d0}
D^{\alpha} u (\texttt{v}) & \quad \alpha \in \mathbb N^{1:2}, |\alpha | \leq  r^{\texttt{v}}, \texttt{v}\in \Delta_0(T),\\
\label{eq:C12d1}
\int_e \frac{\partial^{\beta} u}{\partial n_e^{\beta}} \, q \dd s & \quad q \in \mathbb P_{k - 2(r^{\texttt{v}}+1) + \beta}(e), \beta = 0,\ldots, r^{e}, e\in \Delta_1(T),\\
\label{eq:C12d2}
\int_T u \, q \dx & \quad q \in \mathbb B_k(\bs r).
\end{align}
\end{subequations}
\end{theorem}
\begin{proof}
By the decomposition~\eqref{eq:PkCmdec} of $\mathbb P_k(T)$, the dimension of $\mathbb P_k(T)$ matches the number of DoFs. Let $u\in\mathbb P_{k}(T)$ satisfy all the DoFs~\eqref{eq:C12d0}-\eqref{eq:C12d2} vanish.
Thanks to Lemma~\ref{lem:vertexunisolvence}, Lemma~\ref{lem:edgeunisolvence2d} and Lemma~\ref{lem:geodecomp2d}, it follows from the vanishing DoFs~\eqref{eq:C12d0} and~\eqref{eq:C12d1} that $u\in \mathbb P_{k}(S_2(T))$. Then $u=0$ holds from the vanishing DoF~\eqref{eq:C12d2}.
% With DoFs~\eqref{eq:C12d0} and~\eqref{eq:C12d1} vanishes, only $\lambda^{\alpha}, \alpha \in \mathbb N^{0:2}_k,$ with property $\dist(\alpha, \Delta_0(T)) > r^{\texttt{v}}$ and $\dist(\alpha, \Delta_1(T)) > r^{e}=m$ are left to be determined.
% The condition $\dist(\alpha, \Delta_1(T)) > m$ is equivalent to $\alpha \geq m + 1$ and the condition  $\dist(\alpha, \Delta_0(T)) > r^{\texttt{v}}$ is equivalent to $\alpha \leq r - r^{\texttt{v}} - 1$. Let $\tilde \alpha = \alpha - (m+1)$. Then
% $\lambda^{\alpha} = b_T^{m+1} \lambda^{\tilde \alpha}$. As $b_T$ is always positive, it can be removed from DoFs and~\eqref{eq:C12d2} follows.
\end{proof}

When $r^{e}=m = 1$ and $r^{\texttt{v}}=2$, this is known as Argyris element~\cite{ArgyrisFriedScharpf1968,MorganScott1975}. 
When $r^{e}=m$, $r^{\texttt{v}}=2m$ and $k = 4m+1$, $C^m$-continuous finite elements have been constructed in~\cite{BrambleZlamal1970,Zenisek1970}, see also~\cite[Section 8.1]{Lai;Schumaker:2007Spline} and the references therein, whose DoFs are different from~\eqref{eq:C12d1}-\eqref{eq:C12d2}. Here DoFs~\eqref{eq:C12d0}-\eqref{eq:C12d2} are firstly constructed in~\cite{huConstructionConformingFinite2021}. The DoFs in \cite{Lai;Schumaker:2007Spline}, also called nodal minimal determining sets in the spline literature, are the point evaluation of functions and their derivatives at some nodes. While DoFs \eqref{eq:C12d1}-\eqref{eq:C12d2} are in the integral form, which is beneficial to the unisolvence of the DoFs and the construction of the finite element de Rham complexes. Smooth finite elements with the DoFs in the integral form on simplexes in arbitrary dimension were firstly constructed in \cite{huConstructionConformingFinite2021}.

% are adopted as
%\begin{align*}
%\int_e \frac{\partial^{\beta} u}{\partial n_e^{\beta}}  \, \lambda_e^{\alpha} \dd s & \quad \alpha \in \mathbb T^{1}_{k -\beta}, \alpha\geq r^{\texttt{v}}-\beta+1, \beta = 0,1,\ldots, r^{e},\\
%\int_T u \lambda^{\alpha} \dx & \quad \alpha \in S_2(T),
%\end{align*}
%which are slightly different from~\eqref{eq:C12d1}-\eqref{eq:C12d2}. 
%As $b_T$ is always positive in the interior of $T$, in DoF~\eqref{eq:C12d2}, it can be further changed to $q\in \, \stackrel{\circ}{\mathbb P}_{k-3(r^{e}+1)}(\bs r).$

With mesh $\mathcal T_h$, define the global $C^m$-continuous finite element space 
\begin{align*} 
V(\mathcal T_h) &= \{u\in L^2(\Omega): u|_T\in\mathbb P_k(T)\textrm{ for all } T\in\mathcal T_h, \\
&\qquad\textrm{ and all the DoFs~\eqref{eq:C12d0} and~\eqref{eq:C12d1} are single-valued}\}. 
\end{align*}
Since $r^{\texttt{v}}\geq r^e$, the single-valued DoFs~\eqref{eq:C12d0} and~\eqref{eq:C12d1} will imply $u\in C^m(\Omega)$.

The finite element space $V(\mathcal T_h)$ admits the following geometric decomposition 
\begin{align*} %\label{eq:Vh2d}
V(\mathcal T_h) &= \Oplus_{\texttt{v}\in \Delta_{0}(\mathcal T_h)} \mathbb P_k(D(\texttt{v}, r^{\texttt{v}})) \,\oplus\, \Oplus_{e\in \Delta_{1}(\mathcal T_h)}\mathbb P_k\left ( D(e,r^{e}) \backslash D( \Delta_0(e), r^{\texttt{v}}) \right ) \\
&\quad\,\oplus\,  \Oplus_{T\in \mathcal T_h}\mathbb B_k(\bs r,T). %\notag
\end{align*}
%where $S_0(e) =  $.
The dimension of $V(\mathcal T_h)$ is
\begin{align*}
\dim V(\mathcal T_h) &= |\Delta_0(\mathcal T_h)| {r^{\texttt{v}}+2 \choose 2} + |\Delta_{1}(\mathcal T_h)|\left[{k-2r^{\texttt{v}}+r^e\choose 2}-{k-2r^{\texttt{v}}-1\choose 2}\right] \\
&\quad+|\Delta_2(\mathcal T_h)|\left[{k-3r^e-1 \choose 2}-3{r^{\texttt{v}}-2r^e \choose 2}\right].
\end{align*}
%Denoted by $\tilde k = k-3(r^e+1)$ and $\tilde r = r^{\texttt{v}} - 2r^{e} - 2$. By Lemma~\ref{lm:dist}, we can view the constraint in~\eqref{eq:Pcirc} as $\alpha \notin D(\texttt{v}_i, \tilde r)$ and thus obtain another formula on the dimension
%$$
%\dim \mathbb B_k(\bs r) = {\tilde k+2 \choose 2} - 3{\tilde r + 2 \choose 2} = {k-3r^e-1 \choose 2}-3{r^{\texttt{v}}-2r^e \choose 2}.
%$$
%
In particular, denote by $V^{\rm BZ}(\mathcal T_h)$ the minimum degree case: $r^{e}=m, r^{\texttt{v}}= 2m, k = 4m+1$ with $m\geq0$, which is firstly constructed in \cite{BrambleZlamal1970},
 and the dimension is
$$
\dim V^{\rm BZ}(\mathcal T_h) = |\Delta_0(\mathcal T_h)| {2m + 2 \choose 2} + |\Delta_{1}(\mathcal T_h)| {m+1 \choose 2} + |\Delta_2(\mathcal T_h)| {m \choose 2}.
$$
When $m=0,1$, there is no interior moments as $k = 4m+1$ is not large enough.

\section{Finite Element de Rham Complexes}\label{sec:femderhamcomplex}
In this section we shall construct finite element spaces with appropriate DoFs which make the global finite element complexes~\eqref{eq:femderhamcomplex0} exact
\begin{equation}\label{eq:femderhamcomplex0}
\mathbb R\xrightarrow{\subset} \mathbb V^{\curl}_{k+1}(\mathcal T_h; \boldsymbol{r}_0)\xrightarrow{\curl}\mathbb V^{\div}_{k}(\mathcal T_h; \boldsymbol{r}_1,\boldsymbol{r}_2) \xrightarrow{\div} \mathbb V^{L^2}_{k-1}(\mathcal T_h; \boldsymbol{r}_2)\xrightarrow{}0.
\end{equation} 
Space $\mathbb V_{k}(\mathcal T_h; \boldsymbol{r}_2)$ is denoted as $\mathbb V_{k}^{L^2}(\mathcal T_h; \boldsymbol{r}_2)$ to emphasize it is considered as a subspace of $L^2(\Omega)$ although it might be continuous when $\bs r_2\geq 0$.

Unlike the classical FEEC~\cite{ArnoldFalkWinther2006}, additional smoothness on lower sub-simplexes (vertices and edges for a two-dimensional triangulation) will be imposed, which are described by three vectors $\bs r_0, \bs r_1,$ and $\bs r_2$ with the subscript referring to the $i$-form for $i=0,1,2$. Each $\bs r_i = (r_i^{\texttt{v}}, r_i^e)$ consists of two parameters for the smoothness at vertices and edges, respectively, and $r_i^{\texttt{v}}\geq 2\ r_i^e$ for $i=0,1,2$.
%We allow $r_i^e=-1$
%
%\LC{Use notation $\bs r_1 = (r_1^{\texttt{v}}, r_1^e)$}. The superscript means $1$-form and subscript indicates the smoothness associated to vertices and edges.
% \LC{ Mention Hu group's work and the difference.}
The finite element de Rham complexes constructed in~\cite{huConstructionConformingFinite2021} are exactly complex~\eqref{eq:femderhamcomplex0} with $\bs r_0 = \bs r_1 + 1$ and $\bs r_2 = \bs r_1 - 1$.
We shall consider the general case $\bs r_0 = \bs r_1 + 1, \bs r_1\geq -1,$ and $ \boldsymbol{r}_2\geq \boldsymbol{r}_1\ominus 1 =\max\{ \bs r_1-1, -1 \}$.
% where $k\geq \max\{2r_1^{\texttt{v}}+2, 2r_2^{\texttt{v}}+2, 1\}$, and the three integer vectors $\bs r_0 =(r_0^{\texttt{v}}, r_0^e)$, $\bs r_1 =(r_1^{\texttt{v}}, r_1^e)$, $\bs r_2 =(r_2^{\texttt{v}}, r_2^e)$ satisfying $r_1^{\texttt{v}}\geq 2 \, r_1^e + 1, r_2^{\texttt{v}}\geq 2 \, r_2^e$ and 

\subsection{Continuous vector finite element space and decay smoothness}
%Relations of $\bs r_0, \bs r_1,$ and $\bs r_2$ are needed to form the discrete de Rham complex. 
We first consider a simple case in which the smoothness parameters are decreased by $1$:
\begin{equation*}%\label{eq:rsimple}
\bs r_1\geq 0, \quad \bs r_0 = \bs r_1 + 1, \quad \bs r_2 = \bs r_1 -1.
\end{equation*}
% As the two-dimensional $\curl$ is a rotation of $\grad$, we must have $\bs r_1 = \bs r_0 - 1$ but the relation $ \bs r_2 = \bs r_1 -1$ is imposed for simplicity. 
As $\bs r_1\geq 0$, $\mathbb V^{\div}_{k}(\mathcal T_h; \boldsymbol{r}_1)=\mathbb V^{2}_{k}(\mathcal T_h; \boldsymbol{r}_1)$, and $\mathbb V^{\curl}_{k+1}(\mathcal T_h; \boldsymbol{r}_0)=\mathbb V_{k+1}(\mathcal T_h; \boldsymbol{r}_1+1)$ is at least in $C^1$. In this case, \eqref{eq:femderhamcomplex0} is also called a discrete Stokes complex. The vector element $\mathbb V^{\div}_{k}(\mathcal T_h; \boldsymbol{r}_1)$ is $H^1$-conforming and can be used as the velocity space in discretization of Stokes equation. 

\begin{lemma} \label{lm:dimension}
Let $r_1^e\geq 0, r_1^{\texttt{v}}\geq 2 \, r_1^e + 1, k\geq 2 r_1^{\texttt{v}} + 2
$, and let $ \bs r_0 = \bs r_1+1, \bs r_2 = \bs r_1-1$. 
Write
\begin{align*}
\dim \mathbb V^{\curl}_{k+1}(\mathcal T_h; \boldsymbol{r}_0) &= C_{00}|\Delta_0(\mathcal T_h)| + C_{01}|\Delta_1(\mathcal T_h)| + C_{02}|\Delta_2(\mathcal T_h)|,\\
\dim \mathbb V^{\div}_{k}(\mathcal T_h; \boldsymbol{r}_1)&= C_{10}|\Delta_0(\mathcal T_h)| + C_{11}|\Delta_1(\mathcal T_h)| + C_{12}|\Delta_2(\mathcal T_h)|,\\
\dim \mathbb V^{L^2}_{k-1}(\mathcal T_h; \boldsymbol{r}_2)&= C_{20}|\Delta_0(\mathcal T_h)| + C_{21}|\Delta_1(\mathcal T_h)| + C_{22}|\Delta_2(\mathcal T_h)|.
\end{align*}
The coefficients $C_{ij}$ are presented in the following table
\begin{table}[htp]
	\centering
%	\caption{Dimensions of finite element spaces.}
	\renewcommand{\arraystretch}{1.8}
	\begin{tabular}{@{} c c c c @{}}
	\toprule
& $|\Delta_0(\mathcal T_h)|$ & $|\Delta_1(\mathcal T_h)|$ & $|\Delta_2(\mathcal T_h)|$\\
\hline
$\dim \mathbb V^{\curl}_{k+1}(\mathcal T_h; \boldsymbol{r}_0)$ & ${r_0^{\texttt{v}}+2 \choose 2}$ & $\sum_{i=0}^{r_0^{e}}(k - 2r_0^{\texttt{v}}+i)$ & ${k+3 \choose 2} - 3(C_{00} + C_{01})$  \\
 $\dim \mathbb V^{\div}_{k}(\mathcal T_h; \boldsymbol{r}_1)$ & $2{r_1^{\texttt{v}}+2 \choose 2}$ & $2\sum_{i=0}^{r_1^{e}}(k-1 - 2r_1^{\texttt{v}}+i)$ & $ 2{k+2 \choose 2} - 3(C_{10} + C_{11})$  \\
$\dim \mathbb V^{L^2}_{k-1}(\mathcal T_h; \boldsymbol{r}_2)$ & ${r_2^{\texttt{v}}+2 \choose 2}$ & $\sum_{i=0}^{r_2^{e}}(k-2 - 2r_2^{\texttt{v}}+i)$ & ${k+1 \choose 2} - 3(C_{20} + C_{21})$  \\
$C_{0i} - C_{1i} + C_{2i}$ & $1$ & $-1$ & $1$ 
\medskip \\
\bottomrule
\end{tabular}
\label{table:dimfemdeRham2d}
\end{table}%
\end{lemma}
\begin{proof}
The dimension related to  $|\Delta_0(\mathcal T_h)|$ and $C_{00} - C_{10} + C_{20} = 1$, can be verified directly. For the column of  $|\Delta_1(\mathcal T_h)|$, by removing the same $k-1 - 2r_1^{\texttt{v}}$, we compute
$$
C_{01} - C_{11} + C_{21} = \sum_{i=0}^{r_0^e}(i-1)-2\sum_{i=0}^{r_1^e}i + \sum_{i=0}^{r_2^e}(i+1)=-1.
$$
% i.e., $C_{01} - C_{11} + C_{21} = -1$, it can be easily proved by the induction on $r_1^e$. 
With these two identities, the third column is an easy consequence of~\eqref{eq:dimensionpolyderham}.   
\end{proof}

As a corollary, we obtain the following polynomial bubble complex. 
\begin{corollary}
Let $r_1^e\geq 0, r_1^{\texttt{v}}\geq 2 \, r_1^e + 1, k\geq 2 r_1^{\texttt{v}} + 2
$, and let $ \bs r_0 = \bs r_1+1, \bs r_2 = \bs r_1-1$. The polynomial bubble complex
\begin{equation}\label{eq:femderhamcomplex}
% \resizebox{1.0\hsize}{!}{$
0\xrightarrow{\subset} \mathbb B_{k+1}(\boldsymbol{r}_0)\xrightarrow{\curl}\mathbb B_{k}^2(\boldsymbol{r}_1) \xrightarrow{\div} \mathbb B_{k-1}(\boldsymbol{r}_2)\xrightarrow{Q_0}\mathbb R\to 0
% $}
\end{equation}
 is exact, where $Q_0$ is the $L^2$-projection onto the constant space.
\end{corollary}
\begin{proof}
Clearly we have $\curl\mathbb B_{k+1}(\boldsymbol{r}_0)\subseteq\mathbb B_k^2(\boldsymbol{r}_1)\cap\ker(\div)$. For $\bs v\in\mathbb B_k^2(\boldsymbol{r}_1)\cap\ker(\div)$, apply complex \eqref{eq:polyderham} to get $\bs v=\curl q$ with $q\in\mathbb P_{k+1}(T)$. As $\curl q\in\mathbb B_k^2(\boldsymbol{r}_1)$, we have $(\curl q)|_{\partial T}=0$, which means $q|_{\partial T}$ is constant.
Hence by subtracting a constant, we can choose $q\in\mathbb P_{k+1}(T)$ to satisfy $q|_{\partial T}=0$, as a result $q\in\mathbb B_{k+1}(\boldsymbol{r}_0)$. This proves $\curl\mathbb B_{k+1}(\boldsymbol{r}_0)=\mathbb B_k^2(\boldsymbol{r}_1)\cap\ker(\div)$.

Thanks to the last column of the table in Lemma~\ref{lm:dimension},
\begin{align*}%\label{eq:bubblesdim2d0}
\dim\mathbb B_{k+1}(\boldsymbol{r}_0) - \dim\mathbb B_k^2(\boldsymbol{r}_1) +\dim\mathbb B_{k-1}(\boldsymbol{r}_2)-1 = 0, 
%\\
%&=\dim\curl\mathbb B_{k+1}(\boldsymbol{r}_0)+\dim\mathbb B_{k-1}(\boldsymbol{r}_2)-1 \\
%&=\dim\big(\mathbb B_k^2(\boldsymbol{r}_1)\cap\ker(\div)\big)+\dim\mathbb B_{k-1}(\boldsymbol{r}_2)-1.
\end{align*}
%Hence $\dim\div\mathbb B_k^2(\boldsymbol{r}_1)=\dim\mathbb B_{k-1}(\boldsymbol{r}_2)-1$, 
which together with Lemma~\ref{lm:abstract} concludes the exactness of bubble complex~\eqref{eq:femderhamcomplex}.
\end{proof}

% As a result of the bubble complex~\eqref{eq:femderhamcomplex}, when defining the space $\mathbb V^{\div}_{k}(\mathcal T_h; \boldsymbol{r}_1)$, by Lemma 2.1 in~\cite{Chen;Huang:2021divFinite} (see also Lemma 4.10 in~\cite{Chen;Huang:2021Finite}), DoF~\eqref{eq:Crdivfemdof2d2} can be replaced by
% \begin{align*}
% \int_T \div \boldsymbol{v}\, q \dx, &\quad q\in \ \mathbb B_{k-1}(\bs r_2)/\mathbb R, T\in\mathcal T_h, \\
% \int_T \boldsymbol{v}\cdot\curl q\dx, &\quad q\in \ \mathbb B_{k+1}(\bs r_0), T\in\mathcal T_h.
% \end{align*}

\begin{theorem}\label{lm:femderhamcomplex}
Let $r_1^e\geq 0, r_1^{\texttt{v}}\geq 2 \, r_1^e + 1, k\geq 2 r_1^{\texttt{v}} + 2$, and let $ \bs r_0 = \bs r_1+1, \bs r_2 = \bs r_1-1$. The finite element complex
\begin{equation}\label{eq:femderhamcomplex1}
% \resizebox{1.0\hsize}{!}{$
\mathbb R\xrightarrow{\subset} \mathbb V^{\curl}_{k+1}(\mathcal T_h; \boldsymbol{r}_0)\xrightarrow{\curl}\mathbb V^{\div}_{k}(\mathcal T_h; \boldsymbol{r}_1) \xrightarrow{\div} \mathbb V^{L^2}_{k-1}(\mathcal T_h; \boldsymbol{r}_2)\xrightarrow{}0
% $}
\end{equation}
is exact.
\end{theorem}
\begin{proof}
By construction~\eqref{eq:femderhamcomplex1} is a complex, and  
$$
\curl\mathbb V^{\curl}_{k+1}(\mathcal T_h; \boldsymbol{r}_0)=\mathbb V^{\div}_{k}(\mathcal T_h; \boldsymbol{r}_1)\cap\ker(\div).
$$
By Lemma~\ref{lm:dimension} and the Euler's formula,
\begin{align*}
& \quad\; 1 - \dim\mathbb V^{\curl}_{k+1}(\mathcal T_h; \boldsymbol{r}_0) + \dim\mathbb V^{\div}_{k}(\mathcal T_h; \boldsymbol{r}_1) - \dim\mathbb V^{L^2}_{k-1}(\mathcal T_h; \boldsymbol{r}_2) \\
&=  1 - |\Delta_0(\mathcal T_h)| + |\Delta_1(\mathcal T_h)| - |\Delta_2(\mathcal T_h)| = 0.
%\\
%&=\dim\mathbb V^{\curl}_{k+1}(\mathcal T_h; \boldsymbol{r}_0)+\dim\mathbb V^{L^2}_{k-1}(\mathcal T_h; \boldsymbol{r}_2)-1.
\end{align*}
%Then
%\begin{align*}
%\dim\div\mathbb V^{\div}_{k}(\mathcal T_h; \boldsymbol{r}_1)
%&=\dim\mathbb V^{\div}_{k}(\mathcal T_h; \boldsymbol{r}_1)-\dim\curl\mathbb V^{\curl}_{k+1}(\mathcal T_h; \boldsymbol{r}_0) \\
%&=\dim\mathbb V^{L^2}_{k-1}(\mathcal T_h; \boldsymbol{r}_2).   
%\end{align*}
Therefore the exactness of complex~\eqref{eq:femderhamcomplex1} follows from Lemma~\ref{lm:abstract}.
\end{proof}

\begin{example}\rm
The two-dimensional finite element de Rham complexes constructed by Falk and Neilan~\cite{FalkNeilan2013} correspond to the case $r_1^e = 0, r_1^{\texttt{v}} = 1$, and $k \geq 4$:
%The lowest degree example is: $r_1^e = 0, r_1^{\texttt{v}} = 1$, and $k = 4$, and the complex is
\begin{equation*}%\label{eq:argyrishermite}
% \resizebox{1.0\hsize}{!}{$
\mathbb R\xrightarrow{\subset} {\rm Argy}_{k+1}( 
\begin{pmatrix}
2\\
1 
\end{pmatrix}
)\xrightarrow{\curl}{\rm Herm}_{k}( 
\begin{pmatrix}
1\\
0 
\end{pmatrix}
) \xrightarrow{\div} \mathbb V^{L^2}_{k-1}( 
\begin{pmatrix}
 0\\
 -1
\end{pmatrix}
)\xrightarrow{}0.
% $}
\end{equation*}
To fit the space, we skip $\mathcal T_h$ in the notation and write $\bs r_i = 
\begin{pmatrix}
 r_i^{\texttt{v}}\\
 r_i^e
\end{pmatrix}
$ as column vectors. 
The vector element is $C^0$ and thus the previous finite element is $C^1$ for which the lowest degree is the Argyris element with shape function space $\mathbb P_5$. The last one is discontinuous $\mathbb P_3$ but continuous at vertices. If we want to use a continuous element for the pressure, i.e., $\bs r_2 = (2,0)$, then $\bs r_1 = (3,1), \bs r_0 = (4,2)$ and $k\geq 8$, which may find an application in the strain gradient elasticity problem~\cite{ChenHuangHuang2023,LiaoMingXu2021taylorhood}. Later on, we will relax the relation $\bs r_2 = \bs r_1-1$ and construct relative low degree Stokes pair with continuous pressure elements. 
\end{example}

Notice that the pair $\bs r_1 = (0,0)$ and $\bs r_2 = (-1,-1)$ are not allowed since $\bs r_0 = \bs r_1+1 = (1,1)$ cannot define a $C^1$ element. Indeed the div stability for Stokes pair ${\rm Lagrange}_{k} - {\rm DG}_{k-1}$ is more subtle and not covered in our framework.
%\LC{Add a remark on the Stokes pair? Is it a simple consequence by adding boundary condition to the space.}

%\begin{remark}\rm
%Since $\dim\big(\mathbb B_{k-1}(\bs r_2)/\mathbb R\big)=\dim\mathbb B_{k-1}(\bs r_2)-1$,
%DoF~\eqref{eq:CrL2femdof2d2} can be replaced by
%\begin{align*}
%\int_T v\, q \dx, &\quad q\in \ \mathbb B_{k-1}(\bs r_2)/\mathbb R, T\in\mathcal T_h, \\
%\int_T v \dx, &\quad T\in\mathcal T_h.
%\end{align*}
%\end{remark}

%We can use the finite element pair $\mathbb V^{\div}_{k}(\mathcal T_h; \begin{pmatrix}
%1\\
%0 
%\end{pmatrix})\times\mathbb V^{L^2}_{k-1}(\mathcal T_h; \begin{pmatrix}
%0\\
%-1 
%\end{pmatrix})$ with $k\geq 4$ to discretize the Stokes problem, and the finite element pair $\mathbb V^{\div}_{k}(\mathcal T_h; \begin{pmatrix}
%3\\
%1 
%\end{pmatrix})\times\mathbb V^{L^2}_{k-1}(\mathcal T_h; \begin{pmatrix}
%2\\
%0 
%\end{pmatrix})$  with $k\geq 8$ to discretize the strain gradient elasticity problem~\cite{LiaoMingXu2021taylorhood}.
% \LC{Check the relation to the Falk-Neilan Stokes pair.}

\subsection{Normal continuous finite elements for vector functions}
%In this section, consider the case $r_1^e = -1$. Then it is the standard de Rham complex adding vertex smoothness. 
We continue to consider the case $r_1^e = -1$ and the smoothness on edges are fixed by:
$$
r_0^e = 0, \quad r_1^e = -1, \quad r_2^e = -1.
$$ 
The constraints on the vertex smoothness are
$$
r_0^{\texttt{v}} = r_1^{\texttt{v}} + 1, \quad r_1^{\texttt{v}} \geq -1, \quad r_2^{\texttt{v}} = \max\{ r_1^{\texttt{v}}-1, -1\}.
$$

The finite element spaces for scalar functions $\mathbb V^{\curl}_{k+1}(\mathcal T_h; \boldsymbol{r}_0) $ and $ \mathbb V^{L^2}_{k-1}(\mathcal T_h; \boldsymbol{r}_2)$ remain unchanged. 
%
%
%
%Given $\boldsymbol{r}_1=\begin{pmatrix}r_1^{\texttt{v}} \\ r_1^{e}\end{pmatrix}, \boldsymbol{r}_2=\begin{pmatrix}r_2^{\texttt{v}} \\ r_2^{e}\end{pmatrix}\in \mathbb N_{-1}^2$ such that
%Set $\boldsymbol{r}_0=\begin{pmatrix}r_0^{\texttt{v}} \\ r_0^{e}\end{pmatrix}:=\boldsymbol{r}_1+1\geq0$. 
%
We need to define the finite element space for $\boldsymbol{H}(\div,\Omega)$ with parameters $\bs r_1 = (r_1^{\texttt{v}}, -1)$. The vector function is not continuous on edges. But to be $H(\div)$-conforming, the normal component should be continuous. A refined notation for the smoothness parameter would be $\bs r_1 = \begin{pmatrix}
 r_1^{\texttt{v}}\\
-1,0
\end{pmatrix}
$
where the tangential component is $-1$ (discontinuous) and the normal component is $0$ (continuous);
notation $\begin{pmatrix}
 r_1^{\texttt{v}}\\
\frac{1}{2}
\end{pmatrix}$ is adopted in~\cite{huConstructionConformingFinite2021}. To simplify notation, we still use the simplified form $\bs r_1 = (r_1^{\texttt{v}}, -1)$ and understand that $r_1^e=-1$ for $\mathbb V^{\div}_{k}(\mathcal T_h; \boldsymbol{r}_1)$ space means the normal continuity. 
%So in the following DoFs, we separate the normal and tangential components. 

Take $\mathbb P_{k}(T;\mathbb R^2)$ with $k\geq \max\{2r_1^{\texttt{v}} + 2, 1\}$ as the space of shape functions. For $r_1^{\texttt{v}}\geq 0$, the DoFs are
\begin{subequations}\label{eq:divfemdof}
\begin{align}
\nabla^i\boldsymbol{v}(\texttt{v}), & \quad \texttt{v}\in \Delta_{0}(T), i=0,\ldots, r_1^{\texttt{v}}, \label{eq:divfemdof0}\\
\int_e \boldsymbol{v}\cdot\boldsymbol{n} \, q\dd s, &\quad  \boldsymbol{q}\in \mathbb P_{k- 2(r_1^{\texttt{v}} +1)} (e), e\in \Delta_{1}(T),\label{eq:divfemdof1}\\
\int_e \boldsymbol{v}\cdot\boldsymbol{t} \, q\dd s, &\quad  \boldsymbol{q}\in \mathbb P_{k- 2( r_1^{\texttt{v}} +1)} (e), e\in \Delta_{1}(T),\label{eq:divfemdof2}\\
\int_T \boldsymbol{v}\cdot\boldsymbol{q} \dx, &\quad \boldsymbol{q}\in \mathbb B_{k}^2((r_1^{\texttt{v}}, 0)). \label{eq:divfemdof3}
\end{align}
\end{subequations}
Although $r_1^e= -1$, we still use $\mathbb B_k((r_1^{\texttt{v}}, 0))$ not $\mathbb B_k((r_1^{\texttt{v}}, -1))$ as the interior moments so that we can have DoFs~\eqref{eq:divfemdof1}-\eqref{eq:divfemdof2} on edges. Namely locally we use the vector Hermite-type element with parameter $(r_1^{\texttt{v}}, 0)$. When defining the global $H(\div)$-conforming finite element space, the tangential component~\eqref{eq:divfemdof2} is considered as local, i.e., double valued on interior edges.

When $r_1^{\texttt{v}}= -1$, there is no DoFs on vertices and DoFs are
\begin{subequations}\label{eq:BDMdof}
\begin{align}
\int_e \boldsymbol{v}\cdot\boldsymbol{n} \, q\dd s, &\quad  \boldsymbol{q}\in \mathbb P_{k} (e), e\in \Delta_{1}(T),\label{eq:BDMdof1}\\
\int_e \boldsymbol{v}\cdot\boldsymbol{t} \, q\dd s, &\quad  \boldsymbol{q}\in \mathbb P_{k - 2} (e), e\in \Delta_{1}(T),\label{eq:BDMdof2}\\
\int_T \boldsymbol{v}\cdot\boldsymbol{q} \dx, &\quad \boldsymbol{q}\in \mathbb B_k^2((0, 0)). \label{eq:BDMdof3}
\end{align}
\end{subequations}
The normal component is the full degree polynomial $\mathbb P_k(e)$ but the tangential component is corresponding to the edge bubble $b_e \mathbb P_{k-2}(e)$. The interior moments become $\mathbb B_k((0, 0))$. Locally we use vector Lagrange finite element. At each edge, we use $t-n$ (tangential-normal) coordinate and at a vertex we use the coordinate formed by the two normal direction of two edges containing that vertex and merge into \eqref{eq:BDMdof1}. Then the uni-solvence in one triangle follows from that of vector Lagrange elements.

Define the global $H(\div)$-conforming finite element space
\begin{align*}
\mathbb V^{\div}_{k}(\mathcal T_h; (r_1^{\texttt{v}},-1)) = \{\boldsymbol{v}\in\boldsymbol{L}^2(\Omega;\mathbb R^2):&\, \boldsymbol{v}|_T\in\mathbb P_{k}(T;\mathbb R^2)\;\forall~T\in\mathcal T_h, \\
&\textrm{ all the DoFs~\eqref{eq:divfemdof0}-\eqref{eq:divfemdof1} are single-valued} \},
\end{align*}
for $r_1^{\texttt{v}}\geq0$, and
\begin{align*}
\mathbb V^{\div}_{k}(\mathcal T_h; (-1,-1)) = \{\boldsymbol{v}\in\boldsymbol{L}^2(\Omega;\mathbb R^2):&\, \boldsymbol{v}|_T\in\mathbb P_{k}(T;\mathbb R^2)\;\forall~T\in\mathcal T_h, \\
&\quad\textrm{ DoF~\eqref{eq:BDMdof1} is single-valued} \},
\end{align*}
where the tangential component~\eqref{eq:divfemdof2} and~\eqref{eq:BDMdof2} are considered as local and may be double-valued for each interior edge. 
%For both $r_1^{\texttt{v}}\geq 0$ and $r_1^{\texttt{v}}= - 1$, we have
%\begin{align*} 
%\dim\mathbb V^{\div}_{k}(\mathcal T_h; (r_1^{\texttt{v}},-1))&=|\Delta_2(\mathcal T_h)| \big(k^2-1-3r_1^{\texttt{v}}(r_1^{\texttt{v}}+1)\big) \\
%&\quad+|\Delta_1(\mathcal T_h)|(k-1-2r_1^{\texttt{v}})  +|\Delta_0(\mathcal T_h)|(r_1^{\texttt{v}}+1)(r_1^{\texttt{v}}+2).   
%\end{align*}

\begin{theorem}
Assume parameters $\bs r_0, \bs r_1, \bs r_2$ satisfy
\begin{align*}
& r_0^{\texttt{v}} = r_1^{\texttt{v}} + 1, \quad r_1^{\texttt{v}} \geq -1, \quad r_2^{\texttt{v}} = \max\{ r_1^{\texttt{v}}, 0\}-1,\\
& r_0^e = 0, \quad r_1^e = -1, \quad r_2^e = -1.
\end{align*}
Let $k\geq \max\{2r_1^{\texttt{v}}+2, 1\}$. The finite element complex
\begin{equation}\label{eq:femdivderhamcomplex}
% \resizebox{1.0\hsize}{!}{$
\mathbb R\xrightarrow{\subset} \mathbb V^{\curl}_{k+1}(\mathcal T_h; \boldsymbol{r}_0)\xrightarrow{\curl}\mathbb V^{\div}_{k}(\mathcal T_h; \boldsymbol{r}_1) \xrightarrow{\div} \mathbb V^{L^2}_{k-1}(\mathcal T_h; \boldsymbol{r}_2)\xrightarrow{}0
% $}
\end{equation}
 is exact.
\end{theorem}
\begin{proof}
Apparently~\eqref{eq:femdivderhamcomplex} is a complex, and  
$$
\curl\mathbb V^{\curl}_{k+1}(\mathcal T_h; \boldsymbol{r}_0)=\mathbb V^{\div}_{k}(\mathcal T_h; \boldsymbol{r}_1)\cap\ker(\div).
$$
Then we count the dimension. The dimension count in Lemma~\ref{lm:dimension} is still valid except $C_{11}= k - 1 - 2r_1^{\texttt{v}}$. As $C_{01} = k-2r_0^{\texttt{v}}$ and $C_{21} = 0$, the identity $C_{01} - C_{11} + C_{21} = -1$ still holds. 
%\begin{align*}
%\dim\div\mathbb V^{\div}_{k}(\mathcal T_h; \boldsymbol{r}_1)&=\dim\mathbb V^{\div}_{k}(\mathcal T_h; \boldsymbol{r}_1)-\dim\mathbb V^{\curl}_{k+1}(\mathcal T_h; \boldsymbol{r}_0)+1\\
%&=\frac{1}{2}\big(k(k+1)-3r_1^{\texttt{v}}(r_1^{\texttt{v}}+1)-2\big)|\Delta_2(\mathcal T_h)|+|\Delta_1(\mathcal T_h)|\\
%&\quad -|\Delta_0(\mathcal T_h)|+|\Delta_0(\mathcal T_h)|{r_1^{\texttt{v}}+1\choose2}+1.   
%\end{align*}
%Hence by $r_2^e = -1$ and $r_2^{\texttt{v}} = \max\{ r_1^{\texttt{v}}, 0\}-1$,
%\begin{align*}
%\dim\div\mathbb V^{\div}_{k}(\mathcal T_h; \boldsymbol{r}_1)
%&=\dim\stackrel{\circ}{\mathbb P}_{k-1}(\boldsymbol{r}_2)|\Delta_2(\mathcal T_h)|+|\Delta_0(\mathcal T_h)|{r_2^{\texttt{v}}+2\choose2} \\
%&\quad -|\Delta_2(\mathcal T_h)|+|\Delta_1(\mathcal T_h)|-|\Delta_0(\mathcal T_h)|+1 \\
%&=\dim\mathbb V^{L^2}_{k-1}(\mathcal T_h; \boldsymbol{r}_2)-|\Delta_2(\mathcal T_h)|+|\Delta_1(\mathcal T_h)|-|\Delta_0(\mathcal T_h)|+1.   
%\end{align*}
The rest of the proof is the same as that of Theorem~\ref{lm:femderhamcomplex}.
\end{proof}

\begin{example}\rm
%The lowest degree example is $r_0^{\texttt{v}} = 0, r_1^{\texttt{v}} = -1, r_2^{\texttt{v}} = -1$ and $k = 1$ which is the standard de Rham complex
%$$
%\mathbb R\xrightarrow{\subset} {\rm Lagrange}_{2}(\mathcal T_h; 
%\begin{pmatrix}
%0\\
%0 
%\end{pmatrix}
%) \xrightarrow{\curl} {\rm BDM}_{1}(\mathcal T_h; 
%\begin{pmatrix}
% -1\\
% -1
%\end{pmatrix}
%) \xrightarrow{\div} {\rm DG}_{0}(\mathcal T_h; 
%\begin{pmatrix}
% -1\\
% -1
%\end{pmatrix}
%)\xrightarrow{} 0.
%$$
For $k\geq 1$, $r_0^{\texttt{v}} = 0, r_1^{\texttt{v}} = -1, r_2^{\texttt{v}} = -1$, we recover the standard finite element de Rham complex
$$
\mathbb R\xrightarrow{\subset} {\rm Lagrange}_{k+1}( 
\begin{pmatrix}
0\\
0 
\end{pmatrix}
) \xrightarrow{\curl} {\rm BDM}_{k}( 
\begin{pmatrix}
 -1\\
 -1
\end{pmatrix}
) \xrightarrow{\div} {\rm DG}_{k-1}(
\begin{pmatrix}
 -1\\
 -1
\end{pmatrix}
)\xrightarrow{} 0.
$$
We can choose $r_0^{\texttt{v}} = 1, r_1^{\texttt{v}} = 0, r_2^{\texttt{v}} = -1$ and $k\geq 2$ to get
\begin{equation*}%\label{eq:argyrishermite}
% \resizebox{0.91\hsize}{!}{$
\mathbb R\xrightarrow{\subset} {\rm Herm}_{k+1}( 
\begin{pmatrix}
1\\
0 
\end{pmatrix}
)\xrightarrow{\curl}{\rm Sten}_{k}( 
\begin{pmatrix}
0\\
-1 
\end{pmatrix}
) \xrightarrow{\div} {\rm DG}_{k-1}(
\begin{pmatrix}
 -1\\
 -1
\end{pmatrix}
)\xrightarrow{}0,
% $}
\end{equation*}
which has been constructed in~\cite{Christiansen;Hu;Hu:2018finite}. 
%Our finite element de Rham complex~\eqref{eq:femdivderhamcomplex} is a generalization of~\eqref{eq:argyrishermite} to smoother cases at vertices.
\end{example}

%\subsection{Bubble complexes for continuous vector elements}

\subsection{General cases with inequality constraint}
We consider more general cases with an inequality constraint on the smoothness parameters $\bs r_1$ and $\bs r_2$:
\begin{equation}\label{eq:increaser2}
\bs r_0 = \bs r_1 + 1,  \quad \bs r_1\geq -1, \quad \quad \boldsymbol{r}_2\geq \boldsymbol{r}_1\ominus 1, \quad r_1^{\texttt{v}}\geq 2 \, r_1^e + 1, \quad r_2^{\texttt{v}}\geq 2 \, r_2^e.
\end{equation}
To define the finite element spaces, we further require
\begin{equation*}%\label{eq:r1vr2vk}
k\geq \max\{2r_1^{\texttt{v}}+2, 2r_2^{\texttt{v}}+2, 3r_2^e+4, (r_2^e+4)[r_2^{\texttt{v}}=0]\},
\end{equation*}
where the Iverson bracket $[ {\rm statement}] = 1$ if the statement inside the bracket is true and $0$ otherwise. 
%To ensure $$, we additionally assume $k\geq r_2^{\texttt{v}}+r_2^{e}+4$ when $\boldsymbol{r}_2=(2,1)$, $(0,0)$ and $(0,-1)$.

The finite element spaces for scalar functions $\mathbb V^{\curl}_{k+1}(\mathcal T_h; \boldsymbol{r}_0) $ and $\mathbb V^{L^2}_{k-1}(\mathcal T_h; \boldsymbol{r}_2)$ remain unchanged. 
%
%
%
%Given $\boldsymbol{r}_1=\begin{pmatrix}r_1^{\texttt{v}} \\ r_1^{e}\end{pmatrix}, \boldsymbol{r}_2=\begin{pmatrix}r_2^{\texttt{v}} \\ r_2^{e}\end{pmatrix}\in \mathbb N_{-1}^2$ such that
%Set $\boldsymbol{r}_0=\begin{pmatrix}r_0^{\texttt{v}} \\ r_0^{e}\end{pmatrix}:=\boldsymbol{r}_1+1\geq0$. 
%
Next we define a new finite element space for $\boldsymbol{H}(\div,\Omega)$. Take $\mathbb P_{k}(T;\mathbb R^2)$ as the space of shape functions.
The degrees of freedom are
\begin{subequations}\label{eq:2dCrdivfemdof}
\begin{align}
\nabla^i\boldsymbol{v}(\texttt{v}), & \quad \texttt{v}\in \Delta_{0}(T), i=0,\ldots, r_1^{\texttt{v}}, \label{eq:2dCrdivfemdof0}\\
\nabla^j\div\boldsymbol{v}(\texttt{v}),  & \quad \texttt{v}\in \Delta_{0}(T), j=\max\{r_1^{\texttt{v}},0\},\ldots, r_2^{\texttt{v}}, \label{eq:2dCrdivfemdof1}\\
\int_e \boldsymbol{v}\cdot\boldsymbol{n}\, q\dd s, &\quad  q\in \mathbb P_{k- 2(r_1^{\texttt{v}} +1)} (e), e\in \Delta_{1}(T),\label{eq:2dCrdivfemdof2}\\
\int_e \partial_n^i(\boldsymbol{v}\cdot\boldsymbol{t})\ q \dd s, &\quad  q\in \mathbb P_{k - 2(r_1^{\texttt{v}} +1)+i} (e), e\in \Delta_{1}(T),  i=0,\ldots, r_1^{e}, \label{eq:2dCrdivfemdof3}\\
\int_e \partial_n^i(\div\boldsymbol{v})\ q \dd s, &\quad  q\in \mathbb P_{k-1 - 2(r_2^{\texttt{v}} +1)+i} (e), e\in \Delta_{1}(T),  i=0,\ldots, r_2^{e}, \label{eq:2dCrdivfemdof4}\\
\int_T \div\boldsymbol{v}\, q\dx, &\quad q\in \mathbb B_{k-1}(\boldsymbol{r}_2)/\mathbb R, \label{eq:2dCrdivfemdof5}\\
\int_T \boldsymbol{v}\cdot\boldsymbol{q} \dx, &\quad \boldsymbol{q}\in \curl\mathbb B_{k+1}(\bs r_1+1). \label{eq:2dCrdivfemdof6}
\end{align}
\end{subequations}

We explain the change of DoFs. We add DoFs \eqref{eq:2dCrdivfemdof1},~\eqref{eq:2dCrdivfemdof4}, and \eqref{eq:2dCrdivfemdof5} on $\div \bs v$ to determine $\div \bs v\in \mathbb V^{L^2}_{k-1}(\mathcal T_h; \boldsymbol{r}_2)$. For interior moments, we use the bubble complex \eqref{eq:femderhamcomplex} to split it into range of $\div$ and its orthogonal complement. On edges, DoFs on $\div \bs v$ introduce some linear dependence of normal derivatives of the tangential and normal components and thus need to remove some redundancy. 

More precisely, for $i=0,1,\ldots, r_{1}^e-1$ with $r_2^e =r_1^e  - 1\geq0$, 
$$
\partial_n^i(\div\boldsymbol{v})=\partial_n^i(\partial_n(\boldsymbol{v}\cdot\boldsymbol{n})+\partial_t(\boldsymbol{v}\cdot\boldsymbol{t}))=\partial_n^{i+1}(\boldsymbol{v}\cdot\boldsymbol{n})+\partial_t(\partial_n^i(\boldsymbol{v}\cdot\boldsymbol{t})).
$$
The second term $\partial_t(\partial_n^i(\boldsymbol{v}\cdot\boldsymbol{t}))$ will be determined by~\eqref{eq:2dCrdivfemdof0} and~\eqref{eq:2dCrdivfemdof3}. The normal derivative of the normal component $\partial_n^{i+1}(\boldsymbol{v}\cdot\boldsymbol{n}), i\geq 0,$ is built into \eqref{eq:2dCrdivfemdof4} but not $\bs v\cdot \bs n$ which should be explicitly included in \eqref{eq:2dCrdivfemdof2}. A linear combination of~\eqref{eq:2dCrdivfemdof2},~\eqref{eq:2dCrdivfemdof3}, and~\eqref{eq:2dCrdivfemdof4} will determine
$$
\int_e \partial_n^i \boldsymbol{v}\cdot \bs q \dd s, \quad  \bs q\in \mathbb P^2_{k - 2(r_1^{\texttt{v}} +1)+i} (e), e\in \Delta_{1}(T_h),  i=0,1,\ldots, r_1^{e}.
$$
Consequently it returns to the smooth finite elements defined before. 
% The change of DoFs is illustrated in Fig.~\ref{fig:dofchange}.
%\end{remark}

% \LC{Define $\mathbb B^{\div}_{k}(T;\boldsymbol{r}_1)$ and count the dimension.}

\begin{lemma}\label{lem:2dCrdivfemunisolvence}
Assume $\bs r_1, \bs r_2$ satisfy \eqref{eq:increaser2}, and $k\geq \max\{2r_2^{\texttt{v}}+2, 3r_2^e+4, (r_2^e+4)[r_2^{\texttt{v}}=0]\}$. 
The DoFs~\eqref{eq:2dCrdivfemdof0}-\eqref{eq:2dCrdivfemdof6} are uni-solvent for $\mathbb P_{k}(T;\mathbb R^2)$.  
\end{lemma}
\begin{proof}
The condition $k\geq \max\{2r_2^{\texttt{v}}+2, 3r_2^e+4, (r_2^e+4)[r_2^{\texttt{v}}=0]\}$ ensures $\dim\mathbb B_{k-1}(\bs r_2)\geq1$ which can be verified by showing $|S_2(T)| > 0$ cf. Lemma \ref{lem:geodecomp2d}.  
%\mnote{ in the uni-solvence, we may relax the requirement of $k$. \XH{Unless DoF \eqref{eq:2dCrdivfemdof5} is replaced by others.}}

The number of DoFs~\eqref{eq:2dCrdivfemdof1} and~\eqref{eq:2dCrdivfemdof4}-\eqref{eq:2dCrdivfemdof5} on $\div \bs v$ is
$
\dim \mathbb P_{k-1}(T)-3{r_1^{\texttt{v}}+1\choose2}-1,
$
which is constant with respect to $\boldsymbol{r}_2$. Hence the number of DoFs~\eqref{eq:2dCrdivfemdof0}-\eqref{eq:2dCrdivfemdof6} is also constant with respect to $\boldsymbol{r}_2$. As a result the number of DoFs~\eqref{eq:2dCrdivfemdof0}-\eqref{eq:2dCrdivfemdof6} equals to $\dim\mathbb P_{k}(T;\mathbb R^2)$, which has been proved for case $\boldsymbol{r}_2= \boldsymbol{r}_1\ominus 1$.
% The number of DoFs~\eqref{eq:2dCrdivfemdof0}-\eqref{eq:2dCrdivfemdof6} is
% \begin{align*}
% & 6{r_1^{\texttt{v}}+2\choose2}+3{r_2^{\texttt{v}}+2\choose2}-3{r_1^{\texttt{v}}+1\choose2} + 3(k-2r_1^{\texttt{v}}-1)\\
% %&+ 3\sum_{i=0}^{r_1^{e}}(k-2r_1^{\texttt{v}}-1+i) + 3\sum_{i=0}^{r_2^{e}}(k-2r_2^{\texttt{v}}-2+i) \\
% &+{k+1\choose2}-1-3{r_2^{\texttt{v}}+2\choose2}+{k+3\choose2}-3{r_1^{\texttt{v}}+3\choose2}-3(k-2r_1^{\texttt{v}}-2)\\
% =&(k+1)(k+2)=\dim\mathbb P_{k}(T;\mathbb R^2).
% \end{align*}   

Take $\boldsymbol{v}\in\mathbb P_{k}(T;\mathbb R^2)$ and assume all the DoFs~\eqref{eq:2dCrdivfemdof0}-\eqref{eq:2dCrdivfemdof6} vanish. 
The vanishing DoF~\eqref{eq:2dCrdivfemdof2} implies $\div\boldsymbol{v}\in L_0^2(T)$.
By the vanishing DoFs~\eqref{eq:2dCrdivfemdof0}-\eqref{eq:2dCrdivfemdof1} and~\eqref{eq:2dCrdivfemdof4}-\eqref{eq:2dCrdivfemdof5}, we get $\div\boldsymbol{v}=0$.
And it follows from the vanishing DoFs~\eqref{eq:2dCrdivfemdof0} and~\eqref{eq:2dCrdivfemdof2}-\eqref{eq:2dCrdivfemdof3} that $\boldsymbol{v}\in\curl\mathbb B_{k+1}(\bs r_1+1)$. Therefore $\boldsymbol{v}=\boldsymbol{0}$ holds from the vanishing DoF~\eqref{eq:2dCrdivfemdof6}.  
\end{proof}

Define global $C^{r_1^{e}}$-continuous finite element space
\begin{align*}
\mathbb V^{\div}_{k}(\mathcal T_h; \boldsymbol{r}_1, \boldsymbol{r}_2) = \{\boldsymbol{v}\in \boldsymbol{L}^2(\Omega;\mathbb R^2):&\, \boldsymbol{v}|_T\in\mathbb P_{k}(T;\mathbb R^2)\;\forall~T\in\mathcal T_h, \\
&\textrm{ all the DoFs~\eqref{eq:2dCrdivfemdof0}-\eqref{eq:2dCrdivfemdof4} are single-valued} \}.
\end{align*}
 When $\bs r_2\geq \bs r_1\ominus 1$, we have $$\mathbb V^{\div}_{k}(\mathcal T_h; \boldsymbol{r}_1,\boldsymbol{r}_2) \subseteq \mathbb V^{\div}_{k}(\mathcal T_h; \boldsymbol{r}_1,\boldsymbol{r}_1\ominus 1).$$
Namely additional smoothness on $\div \bs v$ is imposed. We use Figure~\ref{fig:increaser2} to illustrate the exactness of the finite element de Rham complex~\eqref{eq:nodalfemderhamcomplex}, which is obtained by adding more constraints on $\div \bs v$.
\begin{figure}[htbp]
\begin{center}
\includegraphics[width=4in]{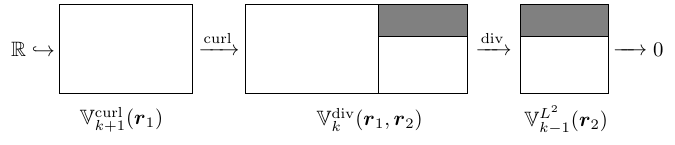}
\caption{Explanation of the smooth finite element de Rham complex with increased smoothness in pressure. 
% Spaces are corresponding to white boxes and the darker gray part denotes the extra constraints on $\div \bs v$.
}
\label{fig:increaser2}
\end{center}
\end{figure}

%\begin{remark}\rm 
%Note that both $\boldsymbol{r}_1$ and $\boldsymbol{r}_2$ are used in the definition.
%\begin{figure}[htbp]
%\begin{center}
%\includegraphics[width=7cm]{figures/divdof.pdf}
%\caption{Change of DoFs: DoF~\eqref{eq:Crdivfemdof2d1} is in the blue rectangle and~\eqref{eq:2dCrdivfemdof2},~\eqref{eq:2dCrdivfemdof3}, and~\eqref{eq:2dCrdivfemdof4} are in the red circle. They are equivalent when $r_2^e =r_1^e  - 1\geq0$.}
%\label{fig:dofchange}
%\end{center}
%\end{figure}

%Roughly speaking, the vector space is split into tangential $\bs v\cdot \bs t$ and normal $\bs v\cdot \bs n$. We use $\bs r_1$ for $\bs v\cdot \bs t$: DoFs~\eqref{eq:2dCrdivfemdof0},~\eqref{eq:2dCrdivfemdof3} and~\eqref{eq:2dCrdivfemdof6}, and $\bs r_2$ for $\bs v\cdot \bs n$ by the linear combination $\bs \div \bs v$. 

\begin{theorem}\label{thm:nodalfemderhamcomplex}
Let $\bs r_0 = \bs r_1 + 1, \bs r_1\geq -1, \boldsymbol{r}_2\geq \boldsymbol{r}_1\ominus 1
$ satisfying $r_1^{\texttt{v}}\geq 2 \, r_1^e + 1, r_2^{\texttt{v}}\geq 2 \, r_2^e$. Assume $k\geq \max\{2r_1^{\texttt{v}}+2, 2r_2^{\texttt{v}}+2, 3r_2^e+4, (r_2^e+4)[r_2^{\texttt{v}}=0]\}$. The finite element complex
\begin{equation}\label{eq:nodalfemderhamcomplex}
% \resizebox{1.0\hsize}{!}{$
\mathbb R\xrightarrow{\subset} \mathbb V^{\curl}_{k+1}(\mathcal T_h; \boldsymbol{r}_0)\xrightarrow{\curl}\mathbb V^{\div}_{k}(\mathcal T_h; \boldsymbol{r}_1,\boldsymbol{r}_2) \xrightarrow{\div} \mathbb V^{L^2}_{k-1}(\mathcal T_h; \boldsymbol{r}_2)\xrightarrow{}0
% $}
\end{equation}
 is exact.
\end{theorem}
\begin{proof}
It is straightforward to verify that~\eqref{eq:nodalfemderhamcomplex} is a complex by showing $\curl\mathbb V^{\curl}_{k+1}(\mathcal T_h; \boldsymbol{r}_0)\subseteq \mathbb V^{\div}_{k}(\mathcal T_h; \boldsymbol{r}_1,\boldsymbol{r}_2)$ and $\div \mathbb V^{\div}_{k}(\mathcal T_h; \boldsymbol{r}_1,\boldsymbol{r}_2)\subseteq \mathbb V^{L^2}_{k-1}(\mathcal T_h; \boldsymbol{r}_2)$. 
% Next we show that complex~\eqref{eq:nodalfemderhamcomplex} is exact. 
It is also obvious that 
$$
\curl\mathbb V^{\curl}_{k+1}(\mathcal T_h; \boldsymbol{r}_0)=\mathbb V^{\div}_{k}(\mathcal T_h; \boldsymbol{r}_1,\boldsymbol{r}_2)\cap\ker(\div).
$$
We have proved the exactness for $\bs r_2= \boldsymbol{r}_1\ominus 1$. When counting the dimension, only need to check the difference. 

The added vertex DoFs for $\mathbb V^{\div}_{k}(\mathcal T_h; \boldsymbol{r}_1,\boldsymbol{r}_2)$ and $\mathbb V^{L^2}_{k-1}(\mathcal T_h; \boldsymbol{r}_2)$ are equal, i.e., $$C_{10}(\bs r_2) - C_{10}(\boldsymbol{r}_1\ominus 1) = C_{20}(\bs r_2) - C_{20}(\boldsymbol{r}_1\ominus 1).$$ Same argument can be applied to edge DoFs. Therefore the alternating column sums remain the same and the proof of Theorem~\ref{lm:femderhamcomplex} can be still applied.
\end{proof}

We present two examples of the de Rham complex ending with the Lagrange element.
\begin{example}\rm
Consider the case $\bs r_1=(1,0)$, $\bs r_2=0$ and $k\geq 4$, which is also constructed as $H^1(\div)-H^1$ Stokes pair in Falk and Neilan~\cite{FalkNeilan2013}. Now we can choose continuous pressure space without increasing the polynomial degree. The complex is
\begin{equation*}%\label{eq:argyrishermitelagrange}
% \resizebox{1.0\hsize}{!}{$
\mathbb R\xrightarrow{\subset} {\rm Argy}_{k+1}( 
\begin{pmatrix}
2\\
1 
\end{pmatrix}
)\xrightarrow{\curl}
\mathbb V^{\div}_{k}(\begin{pmatrix}
1\\
0 
\end{pmatrix}
,
\begin{pmatrix}
 0\\
 0
\end{pmatrix}
) \xrightarrow{\div} {\rm Lagrange}_{k-1}( 
\begin{pmatrix}
 0\\
 0
\end{pmatrix}
)\xrightarrow{}0.
% $}
\end{equation*}
The velocity space is a reduced Hermite space with continuity of $\div \bs v$ at vertices and edges. With such modification, this $\mathbb P_{k} - \mathbb P_{k-1}$ Stokes pair with continuous pressure element is point-wise divergence free comparing to the Taylor-Hood element. 
%But the boundary condition should be imposed by Nistch technique. 
\end{example}

\begin{example}\rm 
Consider the case $\bs r_1=-1$ and $\bs r_2=0$, and $k\geq 4$.
The complex is
\begin{equation*}%\label{eq:argyrishermitelagrange}
%\resizebox{0.91\hsize}{!}{$
\mathbb R\xrightarrow{\subset} {\rm Lagrange}_{k+1}( 
\begin{pmatrix}
0\\
0 
\end{pmatrix}
)\xrightarrow{\curl}
\mathbb V^{\div}_{k}( \begin{pmatrix}
-1\\
-1 
\end{pmatrix}
,
\begin{pmatrix}
 0\\
 0
\end{pmatrix}
) \xrightarrow{\div} {\rm Lagrange}_{k-1}( 
\begin{pmatrix}
 0\\
 0
\end{pmatrix}
)\xrightarrow{}0,
%$}
\end{equation*}
which is the rotation of the finite element de Rham complex in~\cite[Section 5.2.1]{HuZhangZhang2020curlcurl}.
The space $\mathbb V^{\div}_{k}(\mathcal T_h; \bs r_1, \bs r_0
)$ can be used to discretize fourth-order div or curl equations~\cite{FanLiuZhang2019,HuZhangZhang2020curlcurl}. 
We can also apply the pair $\mathbb V^{\div}_{k}(\mathcal T_h;  \bs r_1, \bs r_0
)$ and ${\rm Lagrange}_{k-1}(\mathcal T_h; \bs r_0)$ to mixed finite element methods for Poisson equation $-\Delta u=f$, in which the discrete $u_h$ is continuous.
\end{example}

For simplicity, hereafter we will omit the triangulation $\mathcal T_h$ in the notation of global finite element spaces. For example, $\mathbb V^{\div}_{k}(\mathcal T_h; \boldsymbol{r}_1,\boldsymbol{r}_2)$ will be abbreviated as $\mathbb V^{\div}_{k}(\boldsymbol{r}_1,\boldsymbol{r}_2)$.

\section{Beyond the de Rham Complex}\label{sec:femcomplexbgg}
In this section, we shall construct more finite element complexes from copies of finite element de Rham complexes. 
% In two dimensions, the induced complexes are elasticity and divdiv complexes, and by rotation, Hessian and rot-rot complexes.

\subsection{Finite element curl\,div complexes}
Based on the finite element de Rham complex~\eqref{eq:nodalfemderhamcomplex}, we can obtain the finite element discretization of the curl\,div complex~\cite{Arnold;Hu:2020Complexes}
\begin{equation*}%\label{eq:curldivcomplex}
%\resizebox{0.95\hsize}{!}{$
\mathbb R\times\{0\}\xrightarrow{\subset} H^{1}( \Omega)\times\mathbb R \xrightarrow{(\curl, \boldsymbol{x})} \boldsymbol{H}(\curl\operatorname{div}, \Omega) \xrightarrow{\curl\operatorname{div}}\boldsymbol{H}(\operatorname{div}, \Omega) \xrightarrow{\operatorname{div}} L^{2}(\Omega) \rightarrow 0,
%$}
\end{equation*}
where $\boldsymbol{H}(\curl\operatorname{div}, \Omega):=\boldsymbol{H}^{0, 1}(\div,\Omega)=\{\boldsymbol{v}\in\boldsymbol{H}(\operatorname{div}, \Omega): \div\boldsymbol{v}\in H^1(\Omega)\}$, and the operator $(\curl, \boldsymbol{x})$ is defined by $(\curl, \boldsymbol{x})\begin{pmatrix}v \\ c\end{pmatrix}:=\curl v+c\boldsymbol{x}$ for $v\in H^{1}(\Omega)$ and $c\in\mathbb{R}$.

\begin{theorem}%\label{thm:nodalfemderhamcomplex}
Let $\bs r_0 = \bs r_1 + 1, \bs r_1\geq -1, \boldsymbol{r}_2\geq \max\{\boldsymbol{r}_1-1,0\}, \boldsymbol{r}_3\geq \max\{\boldsymbol{r}_2-2,-1\}
$ satisfying $r_1^{\texttt{v}}\geq 2 \, r_1^e + 1, r_2^{\texttt{v}}\geq 2 \, r_2^e, r_3^{\texttt{v}}\geq 2 \, r_3^e$. Assume $k\geq \max\{2r_1^{\texttt{v}}+2, 2r_2^{\texttt{v}}+2, 3r_2^e+4, (r_2^e+4)[r_2^{\texttt{v}}=0], 2r_3^{\texttt{v}}+4, 3r_3^e+6, (r_3^e+6)[r_3^{\texttt{v}}=0]\}$. The finite element complex
\begin{equation}\label{eq:femcurldivcomplex}
\resizebox{0.95\hsize}{!}{$
\mathbb R\times\{0\}\xrightarrow{\subset} \mathbb V^{\curl}_{k+1}(\boldsymbol{r}_0)\times\mathbb R\xrightarrow{(\curl,\boldsymbol{x})}\mathbb V^{\div}_{k}(\boldsymbol{r}_1,\boldsymbol{r}_2) \xrightarrow{\curl\div}\mathbb V^{\div}_{k-2}(\boldsymbol{r}_2-1,\boldsymbol{r}_3) \xrightarrow{\div} \mathbb V^{L^2}_{k-3}(\boldsymbol{r}_3)\xrightarrow{}0
$}
\end{equation}
 is exact.
\end{theorem}
\begin{proof}
By complex~\eqref{eq:nodalfemderhamcomplex}, clearly~\eqref{eq:femcurldivcomplex} is a complex, and $\div\mathbb V^{\div}_{k-2}(\boldsymbol{r}_2-1,\boldsymbol{r}_3)=\mathbb V^{L^2}_{k-3}(\boldsymbol{r}_3)$. We will focus on the exactness of complex~\eqref{eq:femcurldivcomplex}.

The condition $k\geq \max\{2r_2^{\texttt{v}}+2, 3r_2^e+4, (r_2^e+4)[r_2^{\texttt{v}}=0]\}$ implies $\dim\mathbb B_{k-1}(\bs r_2)\geq1$,  and $k\geq \max\{2r_3^{\texttt{v}}+4, 3r_3^e+6, (r_3^e+6)[r_3^{\texttt{v}}=0]\}$ implies $\dim\mathbb B_{k-3}(\bs r_3)\geq1$.
We get from the exactness of complex~\eqref{eq:nodalfemderhamcomplex} that
$$
\div\mathbb V^{\div}_{k}(\boldsymbol{r}_1,\boldsymbol{r}_2)=\mathbb V^{L^2}_{k-1}(\boldsymbol{r}_2), \quad \curl\mathbb V^{\curl}_{k-1}(\boldsymbol{r}_2)=\mathbb V^{\div}_{k-2}(\boldsymbol{r}_2-1,\boldsymbol{r}_3)\cap\ker(\div).
$$
Hence $\curl\div\mathbb V^{\div}_{k}(\boldsymbol{r}_1,\boldsymbol{r}_2)=\mathbb V^{\div}_{k-2}(\boldsymbol{r}_2-1,\boldsymbol{r}_3)\cap\ker(\div)$ follows from $\mathbb V^{\curl}_{k-1}(\boldsymbol{r}_2)=\mathbb V^{L^2}_{k-1}(\boldsymbol{r}_2)$ when $\boldsymbol{r}_2\geq0$.

For $\boldsymbol{v}\in\mathbb V^{\div}_{k}(\boldsymbol{r}_1,\boldsymbol{r}_2)\cap\ker(\curl\div)$, there exists constant $c$ such that $\div\boldsymbol{v}=2c$. Then we have $\div(\boldsymbol{v}-c\boldsymbol{x})=0$, i.e., $\boldsymbol{v}-c\boldsymbol{x}\in\mathbb V^{\div}_{k}(\boldsymbol{r}_1,\boldsymbol{r}_2)\cap\ker(\div)$. 
Therefore $\boldsymbol{v}\in\curl\mathbb V^{\curl}_{k+1}(\boldsymbol{r}_0)\oplus \boldsymbol{x}\mathbb R$ holds from the exactness of complex~\eqref{eq:nodalfemderhamcomplex}.
\end{proof}

\subsection{Finite element elasticity and Hessian complexes}
We first present two examples. Denote by 
$$\mathbb V_{k}^{\div}(\bs r_1, \bs r_2;\mathbb M) := \mathbb V_{k}^{\div}(\bs r_1, \bs r_2) \times \mathbb V_{k}^{\div}(\bs r_1, \bs r_2).
$$
We take two vector functions by row to form a matrix and each row belongs to $\mathbb V_{k}^{\div}(\bs r_1, \bs r_2)$.
To fit the space, we skip the constant space $\mathbb R$ in the beginning and $0$ at the end in the sequence, and $\mathcal T_h$ in the spaces. The first example has been presented in~\cite{Christiansen;Hu;Hu:2018finite} for $k\geq3$:
\begin{equation*}
\begin{tikzcd}
%\mathbb R \arrow{r}{\subset}
%&
{\rm Argy}_{k+2}
(\begin{pmatrix}
 2\\
 1
\end{pmatrix}
)
\arrow{r}{\curl}
 &
 {\rm Herm}_{k+1}(
\begin{pmatrix}
 1\\
 0
\end{pmatrix};\mathbb R^2)
   \arrow{r}{\div}
 &
 \mathbb V_{k}^{L^2}
 (
\begin{pmatrix}
 0\\
 -1
\end{pmatrix} )
% \arrow{r}{}
% & 0 \\
\\
 %
% \mathbb R \arrow{r}{\subset}
%&
 {\rm Herm}_{k+1}(
\begin{pmatrix}
 1\\
 0
\end{pmatrix};\mathbb R^2)
 \arrow[ur,swap,"{\rm id}"] \arrow{r}{\curl}
 & 
{\rm Sten}_k( 
 \begin{pmatrix}
 0\\
 -1
\end{pmatrix}; \mathbb M
) 
 \arrow[ur,swap,"{\rm -2\sskw}"] \arrow{r}{\div}
 & 
 {\rm DG}_{k-1}(
 \begin{pmatrix}
 -1\\
 -1
\end{pmatrix}; \mathbb R^2 )
%\arrow[r] 
% &0 
\end{tikzcd}.
\end{equation*}
This will lead to the elasticity complex
\begin{equation}\label{hessiancomplexchh}
% \resizebox{1.0\hsize}{!}{$
\bs{RM} \xrightarrow{\subset} {\rm Argy}_{k+2}(
\begin{pmatrix}
2\\
1 
\end{pmatrix}
)\xrightarrow{\air}{\rm HZ}_{k}( 
\begin{pmatrix}
0\\
-1 
\end{pmatrix}; \mathbb S
) \xrightarrow{\div} {\rm DG}_{k-1}(
\begin{pmatrix}
 -1\\
 -1
\end{pmatrix}; \mathbb R^2
)\xrightarrow{}\boldsymbol{0}.
% $}
\end{equation}
% The finite element elasticity complex based on Arnold-Winther element~\cite{Arnold;Winther:2002finite} was reconstructed in~\cite{ArnoldFalkWinther2006b} by a different discrete BGG diagram.

We then present another example with rotated differential operators and use $\mathbb V^{\rot}(\bs r_1,\bs r_2)$ to increase the smoothness of the last space. The finite element BGG diagram for $k\geq5$
\begin{equation*}
\begin{tikzcd}
%\mathbb R \arrow{r}{\subset}
%&
{\rm Argy}_{k+2}(
\begin{pmatrix}
 2\\
 1
\end{pmatrix})
\arrow{r}{\grad}
 &
 {\rm Herm}_{k+1}(
\begin{pmatrix}
 1\\
 0
\end{pmatrix};\mathbb R^2)
   \arrow{r}{\rot}
 &
 \mathbb V_{k}^{L^2}(
\begin{pmatrix}
 0\\
 -1
\end{pmatrix} )
% \arrow{r}{}
% & 0 \\
\\
 %
% \mathbb R \arrow{r}{\subset}
%&
 {\rm Herm}_{k+1}(
\begin{pmatrix}
 1\\
 0
\end{pmatrix};\mathbb R^2)
 \arrow[ur,swap,"{\rm id}"] \arrow{r}{\grad}
 & 
\mathbb V^{\rot}_k( 
 \begin{pmatrix}
 0\\
 -1
\end{pmatrix}, 
 \begin{pmatrix}
0\\
0
\end{pmatrix}
; \mathbb M
) 
 \arrow[ur,swap,"{\rm -2\sskw}"] \arrow{r}{\rot}
 & 
 {\rm Lagrange}_{k-1}(
 \begin{pmatrix}
 0\\
 0
\end{pmatrix}; \mathbb R^2 )
%\arrow[r] 
% &0 
\end{tikzcd}
\end{equation*}
will lead to the finite element Hessian complex constructed in~\cite{Chen;Huang:2021Finite}
\begin{equation}\label{eq:hesscomplex2dfem}
\mathbb P_1\xrightarrow{\subset} V_{k+2}^{\hess}\xrightarrow{\nabla^2}\bs\Sigma_{k}^{\rm rot}\xrightarrow{\rot} \bs V_{k-1}^{\grad}\xrightarrow{}\boldsymbol{0}.
\end{equation}
Note that complex \eqref{eq:hesscomplex2dfem} is not a rotation of complex \eqref{hessiancomplexchh} as complex \eqref{eq:hesscomplex2dfem} ends at a continuous Lagrange element.

We now present the general case.

\begin{theorem}
Let $\bs r_1\geq -1$ and $\bs r_2\geq \bs r_1 \ominus 1$ satisfying $r_1^{\texttt{v}}\geq 2 \, r_1^e + 2, r_2^{\texttt{v}}\geq 2 \, r_2^e,$ and let polynomial degree $k\geq \max\{2r_1^{\texttt{v}} + 3, 2r_2^{\texttt{v}}+2, 3r_2^e+5, 5[r_2^e=-1, r_2^{\texttt{v}}=1], 3[r_2^e=-1, r_2^{\texttt{v}}=0]\}$.
Then we have the BGG diagram
\begin{equation}\label{eq:BGGelasticity}
\begin{tikzcd}[column sep=0.7cm]
\mathbb R \arrow{r}{\subset}
&
\mathbb V_{k+2}^{\curl}(\bs r_1+2)
\arrow{r}{\curl}
 &
\mathbb V_{k+1}^{\div} (\bs r_1+1)
   \arrow{r}{\div}
 &
\mathbb V^{L^2}_k(\bs r_1)
 \arrow{r}{}
 & 0 \\
\mathbb R^2  \arrow[ur,swap, near end, "{\rm \cdot(-\bs x)^{\perp}}"]\arrow{r}{\subset}
&
\mathbb V_{k+1}^{\curl}(\bs r_1+1;\mathbb R^2)
 \arrow[ur,swap,"{\rm id}"] \arrow{r}{\curl}
 & 
\mathbb V_{k}^{\div}(\bs r_1, \bs r_2;\mathbb M) 
 \arrow[ur,swap,"{\rm -2\sskw}"] \arrow{r}{\div}
 & 
\mathbb V^{L^2}_{k-1}(\bs r_2; \mathbb R^2) 
\arrow[r] 
 &\boldsymbol{0} 
\end{tikzcd}
\end{equation}
which leads to the finite element elasticity complex
\begin{equation}\label{eq:elasticityfemcomplex}
{\mathbb P}_1\xrightarrow{\subset} \mathbb V_{k+2}^{\curl}(\bs r_1+2)\xrightarrow{\air}\mathbb V_{k}^{\div}(\bs r_1, \bs r_2; \mathbb S)\xrightarrow{\div}  \mathbb V^{L^2}_{k-1}(\bs r_2; \mathbb R^2)\xrightarrow{}\boldsymbol{0},
\end{equation}
where $\mathbb V_{k}^{\div}(\bs r_1, \bs r_2; \mathbb S):=\mathbb V_{k}^{\div}(\bs r_1, \bs r_2;\mathbb M)\cap\ker(\sskw)$.
\end{theorem}
\begin{proof}
First we show that $\sskw\mathbb V_{k}^{\div}(\bs r_1, \bs r_2; \mathbb M) =\mathbb V^{L^2}_k(\bs r_1)$. For $q\in\mathbb V^{L^2}_k(\bs r_1)$, by the exactness of the complex in the top line of~\eqref{eq:BGGelasticity}, there exists $\boldsymbol{v}_h\in\mathbb V_{k+1}^{\div} (\bs r_1)=\mathbb V_{k+1}^{\curl}(\bs r_1)$ such that $\div\boldsymbol{v}_h=q_h$. Then we get from the anti-commutative property~\eqref{eq:anticommutativeprop1} that $q_h=2\,\sskw(\curl \boldsymbol{v}_h)$.

Again condition $k\geq \max\{2r_2^{\texttt{v}}+2, 3r_2^e+5, 5[r_2^e=-1, r_2^{\texttt{v}}=1], 3[r_2^e=-1, r_2^{\texttt{v}}=0]\}$ ensures $\dim\mathbb B_{k-1}(\bs r_2)\geq3$.

We can apply the BGG framework in~\cite{Arnold;Hu:2020Complexes} to get the complex~\eqref{eq:elasticityfemcomplex} and its exactness. In two dimensions, we will provide a simple proof without invoking the machinery. 

Clearly~\eqref{eq:elasticityfemcomplex} is a complex. We prove the exactness of complex~\eqref{eq:elasticityfemcomplex} in two steps.

\medskip

{\it Step 1. Prove $\air \mathbb V_{k+2}^{\curl}(\bs r_1+2) = \mathbb V_{k}^{\div}(\bs r_1, \bs r_2; \mathbb S)\cap \ker(\div)$.} For $\boldsymbol{\tau}\in \mathbb V_{k}^{\div}(\bs r_1, \bs r_2; \mathbb S)\cap \ker(\div)$, by the bottom complex in~\eqref{eq:BGGelasticity}, there exists $\boldsymbol{v}\in\mathbb V_{k+1}^{\curl}(\bs r_1+1;\mathbb R^2)$ such that $\boldsymbol{\tau}=\curl\boldsymbol{v}$. Then it follows from~\eqref{eq:anticommutativeprop1} that
$$
\div\boldsymbol{v}=2\,\sskw(\curl \boldsymbol{v})=2\,\sskw\boldsymbol{\tau}=0.
$$
By the exactness of the top de Rham complex, there exists $q\in\mathbb V_{k+2}^{\curl}(\bs r_1+2)$ such that $\boldsymbol{v}=\curl q$. Thus $\boldsymbol{\tau}=\curl\boldsymbol{v}=\air q\in\air \mathbb V_{k+2}^{\curl}(\bs r_1+2)$.

{\it Step 2. Prove $\div\mathbb V_{k}^{\div}(\bs r_1, \bs r_2; \mathbb S)=\mathbb V^{L^2}_{k-1}(\bs r_2; \mathbb R^2)$.}
% First 
% \begin{equation}\label{eq:AirkerdivS}
% \air \mathbb V_{k+2}^{\curl}(\bs r_1+2) = \mathbb V_{k}^{\div}(\bs r_1, \bs r_2; \mathbb S)\cap \ker(\div)
% \end{equation}
% using the fact $\curl q\in \mathbb P_k(T)$ iff $q\in \mathbb P_{k+1}(T)$. 
% Then it suffices to prove the mapping
% $$
% \div: \mathbb V_{k}^{\div}(\bs r_1, \bs r_2; \mathbb S) \to  \mathbb V^{L^2}_{k-1}(\bs r_2; \mathbb R^2)
% $$
% is surjective. 
As $\div: \mathbb V_{k}^{\div}(\bs r_1, \bs r_2; \mathbb M) \to  \mathbb V^{L^2}_{k-1}(\bs r_2; \mathbb R^2)$ is surjective, given a $\bs q\in \mathbb V^{L^2}_{k-1}(\bs r_2; \mathbb R^2)$, we can find $\bs \tau \in  \mathbb V_{k}^{\div}(\bs r_1, \bs r_2; \mathbb M)$ such that $\div \bs \tau = \bs q.$ By the diagram~\eqref{eq:BGGelasticity}, we can find $\bs v\in \mathbb V_{k+1}^{\div}(\bs r_1+1)$ such that $\div \bs v = -2\sskw\bs \tau$. Set $\bs \tau^s = \bs \tau + \curl \bs v$. Then $\div \bs \tau^s = \div \bs \tau = \bs q$ and $2\sskw\bs \tau^s = 2\sskw\, \bs \tau + 2\sskw\, \curl \bs v = 2\sskw\bs \tau +\div \bs v = 0$, i.e. $\bs \tau^s$ is symmetric. Therefore we find $\bs \tau^s \in \mathbb V_{k}^{\div}(\bs r_1, \bs r_2; \mathbb S)$ and $\div \bs \tau^s = \bs q$. 
% The degrees of freedom for Hu-Zhang element are
% \begin{align}
% \boldsymbol \tau (\delta) & \quad\forall~\delta\in \mathcal V(K), \label{HdivSBDMfemdof1}\\
% (\boldsymbol \tau\boldsymbol n, \boldsymbol q)_e & \quad\forall~\boldsymbol q\in \mathbb P_{k-2}^2(e),  e\in\mathcal E(K),\label{HdivSBDMfemdof2} \\
% (\boldsymbol{t}^{\intercal}\boldsymbol \tau\boldsymbol{t}, q)_e & \quad\forall~q\in \mathbb P_{k-2}(e),  e\in\mathcal E(K),\label{HdivSBDMfemdof3} \\
% (\boldsymbol \tau, \boldsymbol q)_K &\quad \forall~\boldsymbol q\in \mathbb P_{k-3}(K;\mathbb S)
% \label{HdivSBDMfemdof4}
% \end{align}
\end{proof}

In~\eqref{eq:elasticityfemcomplex}, $\mathbb V_{k}^{\div}(\bs r_1, \bs r_2; \mathbb S)$ is defined as $\mathbb V_{k}^{\div}(\bs r_1, \bs r_2;\mathbb M)\cap\ker(\sskw)$.
Next we give the finite element description of space $\mathbb V_{k}^{\div}(\bs r_1, \bs r_2; \mathbb S)$ and thus can obtain locally supported basis. On each triangle, we take $\mathbb P_{k}(T;\mathbb S)$ as the shape function space. By symmetrizing DoFs~\eqref{eq:2dCrdivfemdof0}-\eqref{eq:2dCrdivfemdof6}, we propose the following local DoFs for space $\mathbb V_{k}^{\div}(\bs r_1, \bs r_2; \mathbb S)$
\begin{subequations}\label{eq:smoothdivSdof}
\begin{align}
\nabla^i\boldsymbol{\tau}(\texttt{v}), & \quad \texttt{v}\in \Delta_{0}(T), i=0,\ldots, r_1^{\texttt{v}}, \label{eq:divSdof1}\\
\nabla^j\div\boldsymbol{\tau}(\texttt{v}),  & \quad \texttt{v}\in \Delta_{0}(T), j=r_1^{\texttt{v}},\ldots, r_2^{\texttt{v}}, \label{eq:divSdof2}\\
\int_e (\boldsymbol{\tau}\boldsymbol{n})\cdot\boldsymbol{q}\dd s, &\quad \boldsymbol{q}\in \mathbb P_{k- 2(r_1^{\texttt{v}} +1)}^2 (e), e\in \Delta_{1}(T),\label{eq:divSdof3}\\
\int_e \partial_n^i(\boldsymbol{t}^{\intercal}\boldsymbol{\tau}\boldsymbol{t})\ q \dd s, &\quad  q\in \mathbb P_{k - 2(r_1^{\texttt{v}} +1)+i} (e), e\in \Delta_{1}(T),  i=0,\ldots, r_1^{e}, \label{eq:divSdof4}\\
\int_e \partial_n^i(\div\boldsymbol{\tau})\cdot\boldsymbol{q} \dd s, &\quad \boldsymbol{q}\in \mathbb P_{k-1 - 2(r_2^{\texttt{v}} +1)+i}^2(e), e\in \Delta_{1}(T),  i=0,\ldots, r_2^{e}, \label{eq:divSdof5}\\
\int_T \div\boldsymbol{\tau} \cdot\boldsymbol{q}\dx, &\quad \boldsymbol{q}\in \mathbb B_{k-1}^2(\boldsymbol{r}_2)/\boldsymbol{RM}, \label{eq:divSdof6}\\
\int_T \boldsymbol{\tau}:\boldsymbol{q} \dx, &\quad \boldsymbol{q}\in \air\mathbb B_{k+2}(\bs r_1+2). \label{eq:divSdof7}
\end{align}
\end{subequations}

\begin{lemma}\label{lem:PkSunisolvence}
The DoFs~\eqref{eq:divSdof1}-\eqref{eq:divSdof7} are uni-solvent for $\mathbb P_{k}(T;\mathbb S)$.
\end{lemma}
\begin{proof}
The number of DoFs~\eqref{eq:divSdof2} and~\eqref{eq:divSdof5}-\eqref{eq:divSdof6} is
$
2\dim\mathbb P_{k-1}(T)-3-6{r_1^{\texttt{v}}+1\choose2}.
$
Then the number of DoFs~\eqref{eq:divSdof1}-\eqref{eq:divSdof7} is
\begin{align*}
&\quad 9{r_1^{\texttt{v}}+2\choose 2}+6(k-1- 2r_1^{\texttt{v}}) + 3\sum_{i=0}^{r_1^e}(k-1- 2r_1^{\texttt{v}}+i) + \dim \mathbb B_{k+2}(\bs r_1+2)
 \\
&\quad +2\dim\mathbb P_{k-1}(T)-3-6{r_1^{\texttt{v}}+1\choose2} =\dim\mathbb P_{k+2}(T)+2\dim\mathbb P_{k-1}(T)-3,
\end{align*}
by~\eqref{eq:dimensionpolyderham},
% \begin{align*}
% &9{r_1^{\texttt{v}}+2\choose 2}+6{r_2^{\texttt{v}}+2\choose 2}-6{r_1^{\texttt{v}}+1\choose 2} +6(k-1- 2r_1^{\texttt{v}}) \\ 
% &+ 3\sum_{i=0}^{r_1^e}(k-1- 2r_1^{\texttt{v}}+i) 
% +6\sum_{i=0}^{r_2^e}(k-2- 2r_2^{\texttt{v}}+i) \\
% &+2\left[{k-2-3r_2^e \choose 2}-3{r_2^{\texttt{v}}-2r_2^e \choose 2}\right]-3+{k-5-3r_1^e \choose 2}-3{r_1^{\texttt{v}}-2r_1^e-2\choose 2} \\
% =&\frac{3}{2}(k+1)(k+2),
% \end{align*}
which equals to $\dim\mathbb P_k(T;\mathbb S)$.

Take $\boldsymbol{\tau}\in\mathbb P_k(T;\mathbb S)$, and assume all the DoFs~\eqref{eq:divSdof1}-\eqref{eq:divSdof7} vanish. It follows from the integration by parts and~\eqref{eq:divSdof3} that
$$
(\div\boldsymbol{\tau},\boldsymbol{q})_T=0\quad\forall~\boldsymbol{q}\in\boldsymbol{RM}.
$$
Thanks to DoFs~\eqref{eq:divSdof1}-\eqref{eq:divSdof2} and~\eqref{eq:divSdof5}-\eqref{eq:divSdof6}, we get $\div\boldsymbol{\tau}=\boldsymbol{0}$.
On each edge $e$,
\begin{align*}
\partial_n^i(\div\boldsymbol{\tau})&=\partial_n^i(\div(\boldsymbol{n}^{\intercal}\boldsymbol{\tau}))\boldsymbol{n}+\partial_n^i(\div(\boldsymbol{t}^{\intercal}\boldsymbol{\tau}))\boldsymbol{t} \\
&=\partial_n^i\big(\partial_n(\boldsymbol{n}^{\intercal}\boldsymbol{\tau}\boldsymbol{n})+\partial_t(\boldsymbol{n}^{\intercal}\boldsymbol{\tau}\boldsymbol{t})\big)\boldsymbol{n}+\partial_n^i\big(\partial_n(\boldsymbol{t}^{\intercal}\boldsymbol{\tau}\boldsymbol{n})+\partial_t(\boldsymbol{t}^{\intercal}\boldsymbol{\tau}\boldsymbol{t})\big)\boldsymbol{t} \\
&=\big(\partial_n^{i+1}(\boldsymbol{n}^{\intercal}\boldsymbol{\tau}\boldsymbol{n})+\partial_t\partial_n^i(\boldsymbol{n}^{\intercal}\boldsymbol{\tau}\boldsymbol{t})\big)\boldsymbol{n}+\big(\partial_n^{i+1}(\boldsymbol{t}^{\intercal}\boldsymbol{\tau}\boldsymbol{n})+\partial_t\partial_n^i(\boldsymbol{t}^{\intercal}\boldsymbol{\tau}\boldsymbol{t})\big)\boldsymbol{t}.
\end{align*}
Then we acquire from DoFs~\eqref{eq:divSdof1}-\eqref{eq:divSdof5} that $\boldsymbol{\tau}\in\air\mathbb B_{k+2}(\bs r_1+2)$. Finally we get $\boldsymbol{\tau}=\boldsymbol{0}$ from the vanishing DoF~\eqref{eq:divSdof7}.
\end{proof}

Next we define the global finite element space and show it is $\mathbb V_{k}^{\div} (\bs r_1, \bs r_2; \mathbb S)$.
\begin{lemma}
It holds  
\begin{align}
\mathbb V_{k}^{\div} (\bs r_1, \bs r_2; \mathbb S) = \mathbb V:= \{\boldsymbol{\tau}\in \boldsymbol{L}^2(\Omega;\mathbb S):&\, \boldsymbol{\tau}|_T\in\mathbb P_{k}(T;\mathbb S)\;\forall~T\in\mathcal T_h, \notag\\
&\textrm{ all the DoFs~\eqref{eq:divSdof1}-\eqref{eq:divSdof5} are single-valued} \}. 
\notag%\label{eq:Vkdivr1r2S}
\end{align}
\end{lemma}
\begin{proof}
%Denote by $$ the space in the right hand side of~\eqref{eq:Vkdivr1r2S} for simplicity.
Apparently $\mathbb V\subseteq\mathbb V_{k}^{\div} (\bs r_1, \bs r_2; \mathbb S)$. By comparing DoFs and direct computation, we can show $\dim\mathbb V_{k}^{\div} (\bs r_1, \bs r_2; \mathbb S)=\dim\mathbb V$ and the desired result follows.
\end{proof}

%
%\begin{example}\rm
%The pair, for $k\geq 5$, which  is necessary to ensure $\dim\mathbb B_{k-1}(\boldsymbol{r}_2)\geq3$,
%% $\mathbb P_1\subseteq\, \stackrel{\circ}{\mathbb P}_{k-1 - 3(r_2^{\texttt{v}} + 1) }(\bs r_2)$,
%\begin{equation*}
%\mathbb V^{\div}_k( 
% \begin{pmatrix}
% 0\\
% -1
%\end{pmatrix}, 
% \begin{pmatrix}
%0\\
%0
%\end{pmatrix}
%; \mathbb S
%) 
%\stackrel{\div}{\longrightarrow}
% {\rm Lagrange}_{k-1}(
% \begin{pmatrix}
% 0\\
% 0
%\end{pmatrix}; \mathbb R^2 )
%\end{equation*}
%can be used to discretize the linear elasticity in the mixed form. The space for the symmetric stress is Hu-Zhang element with constraint $\div \bs\sigma$ continuous at vertices and edges. The displacement space is continuous. The obtained saddle point system will have smaller dimension compared with the Hu-Zhang element and discontinuous displacement pair. 
%%The requirement $k\geq 5$. \mnote{ do we have to?}
%\end{example}

\subsection{Finite element divdiv complexes}
%\LC{In BGG, the smoothness of the Sobolev spaces between mskw are the same. In discrete level, the smoothness at vertices and edges should match. So $\tilde r_1+1 = \bs r_1$.}
We first consider the case: the tensor finite element space is continuous. 
Let $\bs r_1\geq 0$ and $\bs r_2\geq \max\{\bs r_1 - 2, -1\}$. We introduce the space $\mathbb V_{k}^{\div\div^+} (\bs r_1, \bs r_2; \mathbb M)\subseteq \mathbb V_{k}^{\div} (\bs r_1)\times \mathbb V_{k}^{\div} (\bs r_1)$ with constraint on $\div\div \bs \tau$. 
The shape function space is $\mathbb P_{k}(T; \mathbb M)$ with $k\geq \max\{2r_1^{\texttt{v}} + 2, 2r_2^{\texttt{v}}+3\}$ and DoFs are
\begin{subequations}\label{eq:divdivdof}
\begin{align}
\nabla^i\boldsymbol{\tau}(\texttt{v}), & \quad \texttt{v}\in \Delta_{0}(T), i=0,\ldots, r_1^{\texttt{v}}, \label{eq:divdivdof0}\\
\nabla^j\div \div\boldsymbol{\tau}(\texttt{v}),  & \quad \texttt{v}\in \Delta_{0}(T), j= \max\{r_1^{\texttt{v}}-1,0\},\ldots, r_2^{\texttt{v}}, \label{eq:divdivdof2}\\
\int_e \boldsymbol{\tau}\boldsymbol{n} \cdot \bs q\dd s, &\quad  \bs q\in \mathbb P^2_{k- 2(r_1^{\texttt{v}} +1)} (e), e\in \Delta_{1}(T),\label{eq:divdivdof3}\\
\int_e \partial_n^i(\boldsymbol{\tau}\boldsymbol{t}) \cdot \bs q \dd s, &\quad \bs q\in \mathbb P_{k - 2(r_1^{\texttt{v}} +1)+i}^2(e), e\in \Delta_{1}(T),  i=0,\ldots, r_1^{e}, \label{eq:divdivdof31}\\
\int_e \boldsymbol{n}^{\intercal}\div\boldsymbol{\tau} q\dd s, &\quad  q\in \mathbb P_{k-1- 2r_1^{\texttt{v}}} (e), e\in \Delta_{1}(T),\label{eq:divdivdof32}\\
\int_e \partial_n^i(\boldsymbol{t}^{\intercal}\div\boldsymbol{\tau}) q\dd s, &\quad  q\in \mathbb P_{k-1- 2r_1^{\texttt{v}}+i} (e), e\in \Delta_{1}(T),  i=0,\ldots, r_1^{e}-1,\label{eq:divdivdof33}\\
\int_e \partial_n^i(\div\div\boldsymbol{\tau})\ q \dd s, &\quad  q\in \mathbb P_{k-2 - 2(r_2^{\texttt{v}} +1)+i} (e), e\in \Delta_{1}(T),  i=0,\ldots, r_2^{e}, \label{eq:divdivdof44}\\
\int_T \div\boldsymbol{\tau}\cdot \bs q\dx, &\quad \bs q\in \curl\mathbb B_{k}(\bs r_1), \label{eq:divdivdof5}\\
\int_T \div\div\boldsymbol{\tau}\, q\dx, &\quad q\in \mathbb B_{k-2}(\boldsymbol{r}_2)/\mathbb P_1(T), \label{eq:divdivdof55}\\
\int_T \boldsymbol{\tau}:\boldsymbol{q} \dx, &\quad \boldsymbol{q}\in \curl\mathbb B_{k+1}(\bs r_1+1;\mathbb R^2). \label{eq:divdivdof6}
\end{align}
\end{subequations}

%Twist these DoFs to make sure the top one is the following diagram is an exact sequence. 
%It is like adding constraint on $\div\div \bs \tau$ to the space $\mathbb V_{k-1}^{\div}(\bs r_1, \tilde{\bs r}_1) $ to match the dimension of the space $\mathbb V_{k-1}^{\div}(\tilde{\bs r}_1, \bs r_2).$
\begin{lemma}\label{lm:divdivM}
Assume $k\geq \max\{2r_1^{\texttt{v}} + 2, 2r_2^{\texttt{v}}+3, 3r_2^e+6, 6[r_2^e=-1, r_2^{\texttt{v}}=1], 4[r_2^e=-1, r_2^{\texttt{v}}=0]\}$.
The DoFs~\eqref{eq:divdivdof0}-\eqref{eq:divdivdof6} are uni-solvent for $\mathbb P_{k}(T; \mathbb M)$.  
\end{lemma}
\begin{proof}
The number of DoFs~\eqref{eq:divdivdof2}, \eqref{eq:divdivdof44} and \eqref{eq:divdivdof55} is
\begin{align*}  
\dim\mathbb P_{k-2}(T)-3-3{r_1^{\texttt{v}}\choose2} & = 3\sum_{i=1}^{r_1^e-1}(k-2r_1^v+i) + \dim\mathbb B_{k-2}(\max\{\boldsymbol{r}_1-2,-1\})-3.
\end{align*}
And
the number of DoFs~\eqref{eq:2dCrdivfemdof0}-\eqref{eq:2dCrdivfemdof6} for $\mathbb V_{k}^{\div} (\bs r_1, \bs r_1-1)$ minus the number of DoFs \eqref{eq:divdivdof0}, \eqref{eq:divdivdof3}-\eqref{eq:divdivdof33}, \eqref{eq:divdivdof5} and \eqref{eq:divdivdof6} is 
\begin{align*}
&\quad 3\sum_{i=1}^{r_1^e-1}(k-2r_1^v+i)+ 2\dim\mathbb B_{k-1}(\boldsymbol{r}_1-1)-\dim\mathbb B_{k}(\boldsymbol{r}_1)-2 \\
&=2\dim\mathbb P_{k-1}(T)-\dim\mathbb P_{k}(T)-6{r_1^{\texttt{v}}+1\choose2}+3{r_1^{\texttt{v}}+2\choose2}-6\sum_{i=0}^{r_1^e-1}(k-2r_1^v+i)\\
&\quad +3\sum_{i=0}^{r_1^e}(k-2r_1^v-1+i)+3\sum_{i=1}^{r_1^e-1}(k-2r_1^v+i)-2 \\
&=2\dim\mathbb P_{k-1}(T)-\dim\mathbb P_{k}(T)-3{r_1^{\texttt{v}}\choose2}-2,
\end{align*}
by \eqref{eq:dimensionpolyderham}, which equals to $\dim\mathbb P_{k-2}(T)-3-3{r_1^{\texttt{v}}\choose2}$. Hence the number of DoFs~\eqref{eq:divdivdof0}-\eqref{eq:divdivdof6} equals to $\dim\mathbb P_{k}(T; \mathbb M)$.

Take $\boldsymbol{\tau}\in\mathbb P_k(T;\mathbb M)$, and assume all the DoFs~\eqref{eq:divdivdof0}-\eqref{eq:divdivdof6} vanish. 
% Applying the integration by parts, it follows from \eqref{eq:divdivdof3} and \eqref{eq:divdivdof32} that
% $$
% (\div\div\boldsymbol{\tau}, q)_T=0\quad\forall~q\in\mathbb P_1(T).
% $$
% Thanks to DoFs~\eqref{eq:divdivdof0}-\eqref{eq:divdivdof2}, \eqref{eq:divdivdof44} and \eqref{eq:divdivdof55}, we get $\div\div\boldsymbol{\tau}=0$.
Let $\boldsymbol{v}=\div\boldsymbol{\tau}\in\mathbb P_{k-1}(T;\mathbb R^2)$.
Applying the integration by parts, it follows from \eqref{eq:divdivdof3} and \eqref{eq:divdivdof32} that
$$
(\div\boldsymbol{v}, q)_T=(\div\div\boldsymbol{\tau}, q)_T=0\quad\forall~q\in\mathbb P_1(T).
$$
Applying Lemma~\ref{lem:2dCrdivfemunisolvence}, i.e. the unisolvence of space $\mathbb V^{\div}_{k-1}(\boldsymbol{r}_1-1, \boldsymbol{r}_2)$, it follows from DoFs~\eqref{eq:divdivdof0}-\eqref{eq:divdivdof2} and~\eqref{eq:divdivdof32}-\eqref{eq:divdivdof55} that $\div\boldsymbol{\tau}=\boldsymbol{v}=\boldsymbol{0}$. Then $\boldsymbol{\tau}=\curl\boldsymbol{w}$ with some $\boldsymbol{w}\in\mathbb P_{k+1}(T;\mathbb R^2)$. Thanks to Theorem~\ref{thm:Cr2dfemunisolvence}, we derive $\boldsymbol{w}=\boldsymbol{0}$ and $\boldsymbol{\tau}=\boldsymbol{0}$ from DoFs \eqref{eq:divdivdof0}, \eqref{eq:divdivdof3}-\eqref{eq:divdivdof31} and \eqref{eq:divdivdof6}.
\end{proof}

%\LC{Global space: all tangential part are local.}

Define global $H(\div\div)$-conforming finite element space
\begin{align*}
\mathbb V_{k}^{\div\div^+} (\bs r_1, \bs r_2; \mathbb M) = \{\boldsymbol{\tau}\in \boldsymbol{L}^2(\Omega;\mathbb M):&\, \boldsymbol{\tau}|_T\in\mathbb P_{k}(T;\mathbb M)\;\forall~T\in\mathcal T_h, \\
&\textrm{ all the DoFs~\eqref{eq:divdivdof0}-\eqref{eq:divdivdof44} are single-valued} \}.
\end{align*} 
The super-script in $\div\div^+$ indicates the smoothness is more than $\div\div$-conforming. Indeed we have $\mathbb V_{k}^{\div\div^+} (\bs r_1, \bs r_2; \mathbb M)\subset \bs H(\div,\Omega;\mathbb M)\cap \bs H(\div\div,\Omega;\mathbb M)$. 

\begin{theorem}\label{th:divdiv0}
Let $\bs r_1\geq 0$ and $\bs r_2\geq \max\{\bs r_1 - 2,-1\}$. Assume $k\geq \max\{2r_1^{\texttt{v}} + 2, 2r_2^{\texttt{v}}+3, 3r_2^e+6, 6[r_2^e=-1, r_2^{\texttt{v}}=1], 4[r_2^e=-1, r_2^{\texttt{v}}=0]\}$. The BGG diagram
 \begin{equation}\label{eq:femdivdivbgg}
\begin{tikzcd}[column sep=0.378cm]
\mathbb R^2 \arrow{r}{\subset}
&
\mathbb V_{k+1}^{\curl}(\bs r_1+1; \mathbb R^2)
\arrow{r}{\curl}
 &
\mathbb V_{k}^{\div\div^+} (\bs r_1, \bs r_2; \mathbb M)
   \arrow{r}{\div}
 &
\mathbb V^{\div}_{k-1}(\bs r_1 - 1, \bs r_2)
 \arrow{r}{}
 & \boldsymbol{0} \\
 \mathbb R \arrow[ur,swap,"{\rm -\bs x}"] \arrow{r}{\subset}
&
\mathbb V_{k}^{\curl}( \bs r_1)
 \arrow[ur,swap,"{\rm mskw}"] \arrow{r}{\curl}
 & 
\mathbb V_{k-1}^{\div}(\bs r_1 - 1, \bs r_2) 
 \arrow[ur,swap,"{\rm id}"] \arrow{r}{\div}
 & 
\mathbb V^{L^2}_{k-2}(\bs r_2) 
\arrow[r] 
 &0 
\end{tikzcd}
\end{equation}
which leads to the finite element divdiv complex
\begin{equation}\label{eq:continuousdivdivcomplex}
{\bf RT}\xrightarrow{\subset} 
\mathbb V_{k+1}^{\curl}(\bs r_1+1;\mathbb R^2)
\xrightarrow{{\sym\curl}}
\mathbb V_{k}^{\div\div^+} (\bs r_1, \bs r_2; \mathbb S)\xrightarrow{\div\div}  
\mathbb V^{L^2}_{k-2}(\bs r_2) 
\xrightarrow{}0,
\end{equation}
where $\mathbb V_{k}^{\div\div^+}(\bs r_1, \bs r_2; \mathbb S):=\mathbb V_{k}^{\div\div^+}(\bs r_1, \bs r_2; \mathbb M)/\mskw\mathbb V_{k}^{\curl}( \bs r_1)$.
\end{theorem}
\begin{proof}
By the anti-commutative property $\div(\mskw v)=-\curl v$, we can conclude complex
\eqref{eq:continuousdivdivcomplex} from the BGG framework in~\cite{Arnold;Hu:2020Complexes}. 

In the following we give a self-contained proof without invoking the BGG framework. 
Clearly~\eqref{eq:continuousdivdivcomplex} is a complex. As $\div \div \mskw\mathbb V_{k}^{\curl}( \bs r_1) = -\div \curl\mathbb V_{k}^{\curl}( \bs r_1) = 0$, we have 
$$
\div\div\mathbb V_{k}^{\div\div^+} (\bs r_1, \bs r_2; \mathbb S)=\div\div\mathbb V_{k}^{\div\div^+} (\bs r_1, \bs r_2; \mathbb M)=\mathbb V^{L^2}_{k-2}(\bs r_2).
$$
By two complexes in diagram~\eqref{eq:femdivdivbgg}, we have
\begin{align*}
\dim \mathbb V_{k}^{\div\div^+} (\bs r_1, \bs r_2; \mathbb M) &=\dim \mathbb V_{k+1}^{\curl}(\bs r_1+1; \mathbb R^2)+\dim \mathbb V^{\div}_{k-1}(\bs r_1 - 1, \bs r_2)-2,\\
\dim\mathbb V_{k-1}^{\div}(\bs r_1 - 1, \bs r_2) &=\dim\mathbb V_{k}^{\curl}( \bs r_1)+\dim\mathbb V^{L^2}_{k-2}(\bs r_2) -1.
\end{align*}
Combining the last two equations yields 
$$
\dim \mathbb V_{k}^{\div\div^+} (\bs r_1, \bs r_2; \mathbb M)=\dim \mathbb V_{k+1}^{\curl}(\bs r_1+1; \mathbb R^2)+\dim\mathbb V_{k}^{\curl}( \bs r_1)+\dim\mathbb V^{L^2}_{k-2}(\bs r_2)-3.
$$
Hence
$$
\dim \mathbb V_{k}^{\div\div^+} (\bs r_1, \bs r_2; \mathbb S)=\dim \mathbb V_{k+1}^{\curl}(\bs r_1+1; \mathbb R^2)+\dim\mathbb V^{L^2}_{k-2}(\bs r_2)-3.
$$
Therefore the exactness of complex~\eqref{eq:continuousdivdivcomplex} follows from Lemma~\ref{lm:abstract}.
\end{proof}

Next we give the finite element characterization of $\mathbb V_{k}^{\div\div^+} (\bs r_1, \bs r_2; \mathbb S)$. We choose $\mathbb P_k(T;\mathbb S)$ as the shape function space. 
By symmetrizing DoFs~\eqref{eq:divdivdof0}-\eqref{eq:divdivdof6}, we propose the following local DoFs:
\begin{subequations}\label{eq:divdiv+dof}
\begin{align}
\nabla^i\boldsymbol{\tau}(\texttt{v}), & \quad \texttt{v}\in \Delta_{0}(T), i=0,\ldots, r_1^{\texttt{v}}, \label{eq:symdivdiv+dof1}\\
\nabla^j\div \div\boldsymbol{\tau}(\texttt{v}),  & \quad \texttt{v}\in \Delta_{0}(T), j= \max\{r_1^{\texttt{v}}-1,0\},\ldots, r_2^{\texttt{v}}, \label{eq:symdivdiv+dof2}\\
\int_e \boldsymbol{\tau}\boldsymbol{n} \cdot \bs q\dd s, &\quad  \bs q\in \mathbb P^2_{k- 2(r_1^{\texttt{v}} +1)} (e), e\in \Delta_{1}(T),\label{eq:symdivdiv+dof3}\\
\int_e \partial_n^i(\bs t^{\intercal}\boldsymbol{\tau}\boldsymbol{t}) \ q \dd s, &\quad  q\in \mathbb P_{k - 2(r_1^{\texttt{v}} +1)+i}(e), e\in \Delta_{1}(T),  i=0,\ldots, r_1^{e}, \label{eq:symdivdiv+dof31}\\
\int_e \boldsymbol{n}^{\intercal}\div\boldsymbol{\tau} q\dd s, &\quad  q\in \mathbb P_{k-1- 2r_1^{\texttt{v}}} (e), e\in \Delta_{1}(T),\label{eq:symdivdiv+dof32}\\
\int_e \partial_n^i(\boldsymbol{t}^{\intercal}\div\boldsymbol{\tau}) q\dd s, &\quad  q\in \mathbb P_{k-1- 2r_1^{\texttt{v}}+i} (e), e\in \Delta_{1}(T),  i=0,\ldots, r_1^{e}-1,\label{eq:symdivdiv+dof33}\\
\int_e \partial_n^i(\div\div\boldsymbol{\tau})\ q \dd s, &\quad  q\in \mathbb P_{k- 2(r_2^{\texttt{v}} +2)+i} (e), e\in \Delta_{1}(T),  i=0,\ldots, r_2^{e}, \label{eq:symdivdiv+dof44}\\
\int_T \div\boldsymbol{\tau}\cdot \bs q\dx, &\quad \bs q\in \curl\mathbb B_{k}(\bs r_1) / \bs x^{\perp}, \label{eq:symdivdiv+dof5}\\
\int_T \div\div\boldsymbol{\tau} q\dx, &\quad q\in \mathbb B_{k-2}(\boldsymbol{r}_2)/\mathbb P_1(T), \label{eq:symdivdiv+dof55}\\
\int_T \boldsymbol{\tau}:\boldsymbol{q} \dx, &\quad \boldsymbol{q}\in \air\mathbb B_{k+2}(\bs r_1+2). \label{eq:symdivdiv+dof6}
\end{align}
\end{subequations}

% \LC{In the following uni-solvence result, $\bs r_1\geq -1$ is allowed.} The requirement $\bs r_1\geq 0$ is for the BGG diagram.
Using a similar proof as that in Lemma \ref{lm:divdivM}, we can prove the unisolvence. 
\begin{lemma}\label{lem:divdivunisolvence} 
Let $\bs r_1\geq 0$ and $\bs r_2\geq \max\{\bs r_1 - 2,-1\}$. Assume $k\geq \max\{2r_1^{\texttt{v}} + 2, 2r_2^{\texttt{v}}+3, 3r_2^e+6, 6[r_2^e=-1, r_2^{\texttt{v}}=1], 4[r_2^e=-1, r_2^{\texttt{v}}=0]\}$.
The DoFs~\eqref{eq:symdivdiv+dof1}-\eqref{eq:symdivdiv+dof6} are uni-solvent for $\mathbb P_{k}(T;\mathbb S)$.
\end{lemma}

\begin{lemma}
Let $\bs r_1\geq 0$ and $\bs r_2\geq \max\{\bs r_1 - 2,-1\}$. Assume $k\geq \max\{2r_1^{\texttt{v}} + 2, 2r_2^{\texttt{v}}+3, 3r_2^e+6, 6[r_2^e=-1, r_2^{\texttt{v}}=1], 4[r_2^e=-1, r_2^{\texttt{v}}=0]\}$. It holds that 
\begin{align}
\mathbb V_{k}^{\div\div^+}(\bs r_1, \bs r_2; \mathbb S) = \mathbb V:= \{\boldsymbol{\tau}\in \boldsymbol{L}^2(\Omega;\mathbb S):&\, \boldsymbol{\tau}|_T\in\mathbb P_{k}(T;\mathbb S)\;\forall~T\in\mathcal T_h, \notag\\
&\textrm{ all the DoFs~\eqref{eq:symdivdiv+dof1}-\eqref{eq:symdivdiv+dof44} are single-valued} \} 
\notag%\label{eq:Vkdivdivr1r2S}
\end{align}
and $\mathbb V_{k}^{\div\div^+}(\bs r_1, \bs r_2; \mathbb S)\subset \bs H(\div,\Omega;\mathbb S)\cap \bs H(\div\div,\Omega;\mathbb S)$.
\end{lemma}
\begin{proof}
%Denote by $\mathbb V$ the space in the right hand side of~\eqref{eq:Vkdivdivr1r2S} for simplicity.
Apparently $\mathbb V\subseteq\mathbb V_{k}^{\div\div^+} (\bs r_1, \bs r_2; \mathbb S)$. It suffices to prove $\dim\mathbb V_{k}^{\div\div^+} (\bs r_1, \bs r_2; \mathbb S)=\dim\mathbb V$ which can be verified by a direct computation and the Euler's formula.
%By comparing DoFs and direct computation,
%\begin{align*}
%&\quad\dim\mathbb V -\dim\mathbb V^{L^2}_{k-2}(\bs r_2) \\
%&=
%3{r_1^{\texttt{v}}+2\choose 2}|\Delta_0(\mathcal T_h)|-{r_1^{\texttt{v}}\choose 2}|\Delta_0(\mathcal T_h)| +(2k-2- 4r_1^{\texttt{v}}-r_1^e)|\Delta_1(\mathcal T_h)| \\
%&\quad +|\Delta_1(\mathcal T_h)|\sum_{i=0}^{r_1^e}(2k-1- 4r_1^{\texttt{v}}+2i) \\
%&\quad +|\Delta_2(\mathcal T_h)|\big(\dim\mathbb B_{k}(\bs r_1)+\dim\mathbb B_{k+2}(\bs r_1+2)-4\big) \\
%&=2|\Delta_0(\mathcal T_h)| {r_1^{\texttt{v}}+3 \choose 2} + 2|\Delta_{1}(\mathcal T_h)|(r_1^e+2) \left(k-2r_1^{\texttt{v}}-3/2+r_1^e/2\right) \\
%&\quad+2|\Delta_2(\mathcal T_h)|\dim\mathbb B_{k+1}(\bs r_1+1)-3\big(|\Delta_0(\mathcal T_h)|-|\Delta_1(\mathcal T_h)|+|\Delta_2(\mathcal T_h)|\big) \\
%&=\dim\mathbb V_{k+1}^{\curl}(\bs r_1+1;\mathbb R^2)-3\big(|\Delta_0(\mathcal T_h)|-|\Delta_1(\mathcal T_h)|+|\Delta_2(\mathcal T_h)|\big),
%\end{align*}
%which together with the Euler's formula yields
%\begin{align*}  
%\dim\mathbb V &= \dim\mathbb V^{L^2}_{k-2}(\bs r_2)+\dim\mathbb V_{k+1}^{\curl}(\bs r_1+1;\mathbb R^2)-3 \\
%&=\dim\mathbb V^{L^2}_{k-2}(\bs r_2)+\dim\sym\curl\mathbb V_{k+1}^{\curl}(\bs r_1+1;\mathbb R^2).
%\end{align*}
%On the other side, by complex~\eqref{eq:continuousdivdivcomplex}, 
%$$
%\dim\mathbb V_{k}^{\div\div^+}(\bs r_1, \bs r_2; \mathbb S)=\dim\mathbb V^{L^2}_{k-2}(\bs r_2)+\dim\sym\curl\mathbb V_{k+1}^{\curl}(\bs r_1+1;\mathbb R^2).
%$$
%Therefore $\dim\mathbb V_{k}^{\div\div^+} (\bs r_1, \bs r_2; \mathbb S)=\dim\mathbb V$.
\end{proof}

%\begin{remark}\rm
%We assume $\bs r_1\geq 0$ and thus $\bs \tau|_e$ is continuous. Then through linear combination DoF~\eqref{eq:symdivdiv+dof32} can be replaced by 
%$$
%\int_{e} \tr_2^{\div\div}(\bs \tau)q\dd s,  \quad q\in \mathbb P_{k-1- 2r_1^{\texttt{v}}} (e),
%$$
%where $\tr_2^{\div\div}(\bs \tau) = \partial_{t}(\bs t^{\intercal}\bs \tau \bs n) + \boldsymbol{n}^{\intercal}\div\boldsymbol{\tau}$ is one of the trace operators of $\div\div$; see~\cite{ChenHuang2020}. 
%\end{remark}

\begin{example}\rm 
We choose $\bs r_1 = (1, 0), \bs r_2 = (0,0)$ to get the divdiv complex constructed in~\cite{Chen;Huang:2021Finite} for $k\geq6$
\begin{equation*}
\resizebox{1.0\hsize}{!}{$
{\bf RT}\xrightarrow{\subset} 
{\rm Argy}_{k+1}(
\begin{pmatrix}
 2\\
 1
\end{pmatrix}
;\mathbb R^2)
\xrightarrow{{\sym\curl}}
\mathbb V_{k}^{\div\div^+} (
\begin{pmatrix}
 1\\
 0
\end{pmatrix}
, 
\begin{pmatrix}
 0\\
 0
\end{pmatrix}
; \mathbb S)\xrightarrow{\div\div}  
{\rm Lagrange}_{k-2}(
\begin{pmatrix}
 0\\
 0
\end{pmatrix}
) 
\xrightarrow{}0.
$}
\end{equation*}

\end{example}

The finite element divdiv complexes presented in~\cite{Hu;Ma;Zhang:2020family,ChenHuang2020} with $\bs r_1=(0,-1),\bs r_2=(-1,-1)$ are not included in complex \eqref{eq:continuousdivdivcomplex} due to the mis-match of the smoothness. In~\eqref{eq:femdivdivbgg}, $\sskw(\bs \tau)$ is discontinuous for $\bs \tau \in \mathbb V_{k}^{\div} (\begin{pmatrix}
 r_1^{\texttt{v}}\\
 -1
\end{pmatrix},
%\begin{pmatrix}
% -1\\
% -1
%\end{pmatrix},
\bs r_2; \mathbb M).
$ The operator $\mskw$ is still injective. But it is unclear if $\mathbb V_{k}^{\div\div^+}(\begin{pmatrix}
 r_1^{\texttt{v}}\\
 -1
\end{pmatrix}, \bs r_2; \mathbb M)/\mskw\mathbb V_{k}^{\curl}(\begin{pmatrix}
 r_1^{\texttt{v}}\\
 0
\end{pmatrix} )$ consists of symmetric matrix functions with desirable normal continuity. 
The continuous version of the divdiv complex is~\cite{Hu;Ma;Zhang:2020family}
\begin{equation}\label{eq:divdivcomplexH1div}
%\resizebox{0.91\hsize}{!}{$
{\bf RT}\xrightarrow{\subset} \boldsymbol{H}^{1,1}\left(\operatorname{div}, \Omega\right) \stackrel{\sym\curl}{\longrightarrow} \boldsymbol{H}(\operatorname{divdiv}, \Omega ; \mathbb{S}) \cap \boldsymbol{H}(\operatorname{div}, \Omega ; \mathbb{S}) \stackrel{\operatorname{divdiv}}{\longrightarrow} L^{2}(\Omega) \rightarrow 0.
%$}
\end{equation}

Now we consider the finite element discretization of the divdiv complex \eqref{eq:divdivcomplexH1div} by using the BGG framework.
For the case $\bs r_1 = (r_1^{\texttt{v}}, -1)$ with $r_1^{\texttt{v}}\geq 0$, $\bs r_0=\bs r_1+1$ and $\bs r_2\geq \max\{\bs r_1 - 2, -1 \}$, we refine the BGG diagram \eqref{eq:femdivdivbgg} to
\begin{equation}\label{eq:femdivdivbggexample1}
\begin{tikzcd}[column sep=0.378cm]
\mathbb R^2 \arrow{r}{\subset}
&
\mathbb V_{k+1}^{\div}(\bs r_0, 
\begin{pmatrix}
r_1^{\texttt{v}}\\
0 
\end{pmatrix})
\arrow{r}{\curl}
 &
\hat{\mathbb V}_{k}^{\div\div^+} (\bs r_1,
%\begin{pmatrix}
% -1\\
% -1
%\end{pmatrix},
\bs r_2; \mathbb M)
   \arrow{r}{\div}
 &
\mathbb V^{\div}_{k-1}(\begin{pmatrix}
r_1^{\texttt{v}} -1\\
 -1
\end{pmatrix},\bs r_2)
 \arrow{r}{}
 & 0 
\\
 \mathbb R \arrow[ur,swap,"{\rm -\bs x}"] \arrow{r}{\subset}
&
\mathbb V_{k}^{\curl}(\begin{pmatrix}
 r_1^{\texttt{v}}\\
 0
\end{pmatrix})
 \arrow[ur,swap,"{\rm mskw}"] \arrow{r}{\curl}
 & 
\mathbb V_{k-1}^{\div}(\begin{pmatrix}
r_1^{\texttt{v}} -1\\
 -1
\end{pmatrix},\bs r_2) 
 \arrow[ur,swap,"{\rm id}"] \arrow{r}{\div}
 & 
\mathbb V^{L^2}_{k-2}(\bs r_2) 
\arrow[r] 
 &0 
\end{tikzcd}.
\end{equation}
Here the space $\hat{\mathbb V}_{k}^{\div\div^+} (\bs r_1,\bs r_2; \mathbb M)$ is the subspace of $\mathbb V_{k}^{\div\div^+} (\bs r_1, \bs r_2; \mathbb M)$ defined by
$$
\hat{\mathbb V}_{k}^{\div\div^+} (\bs r_1,\bs r_2; \mathbb M):=\big\{\bs\tau\in\mathbb V_{k}^{\div\div^+} (\bs r_1, \bs r_2; \mathbb M): \sskw\bs\tau\in\mathbb V_{k}^{\curl}(\begin{pmatrix}
 r_1^{\texttt{v}}\\
 0
\end{pmatrix})\big\}.
$$
The diagram \eqref{eq:femdivdivbggexample1} will lead to the finite element divdiv complex \eqref{eq:femdivdivcomplex} with the finite element space $\mathbb V_{k}^{\div\div^+} (\bs r_1, \bs r_2; \mathbb S)=\hat{\mathbb V}_{k}^{\div\div^+} (\bs r_1,\bs r_2; \mathbb M)/\mskw\mathbb V_{k}^{\curl}(\begin{pmatrix}
 r_1^{\texttt{v}}\\
 0
\end{pmatrix})$.

Next we prove the exactness of the derived finite element divdiv complex directly rather than using the BGG framework. The space $\mathbb V_{k}^{\div\div^+} (\bs r_1, \bs r_2; \mathbb S)$ is still defined by DoFs \eqref{eq:divdiv+dof} and recall that $i=0,\ldots, -1$ means empty and thus \eqref{eq:symdivdiv+dof31} and \eqref{eq:symdivdiv+dof33} are not present. 

\begin{theorem}
Let $\bs r_1 = (r_1^{\texttt{v}}, -1), r_1^{\texttt{v}}\geq 0$ and $\bs r_2\geq \max\{\bs r_1 - 2, -1 \}$. 
Assume $k\geq \max\{2r_1^{\texttt{v}} + 2, 2r_2^{\texttt{v}}+3, 3r_2^e+6, 6[r_2^e=-1, r_2^{\texttt{v}}=1], 4[r_2^e=-1, r_2^{\texttt{v}}=0]\}$.
The following finite element divdiv complex is exact
\begin{equation}\label{eq:femdivdivcomplex}
\resizebox{0.935\hsize}{!}{$
{\bf RT}\xrightarrow{\subset} 
\mathbb V_{k+1}^{\div}(\begin{pmatrix}
r_1^{\texttt{v}} +1\\
 0
\end{pmatrix}, 
\begin{pmatrix}
r_1^{\texttt{v}}\\
0 
\end{pmatrix})
\xrightarrow{{\sym\curl}}\mathbb V_{k}^{\div\div^+}(\begin{pmatrix}
 r_1^{\texttt{v}}\\
 -1
\end{pmatrix}, \bs r_2; \mathbb S)\xrightarrow{\div\div}  \mathbb V^{L^2}_{k-2}(\bs r_2)\xrightarrow{}0.
$}
\end{equation}
\end{theorem}
\begin{proof}
It is easy to check that~\eqref{eq:femdivdivcomplex} is a complex. We will prove the exactness of complex~\eqref{eq:femdivdivcomplex}.

By divdiv complex~\eqref{eq:continuousdivdivcomplex}, we have $\div\div\mathbb V_{k}^{\div\div^+}(\begin{pmatrix}
 r_1^{\texttt{v}}\\
 0
\end{pmatrix}, \bs r_2; \mathbb S)=\mathbb V^{L^2}_{k-2}(\bs r_2)$. Noting that
$$
\div\div\mathbb V_{k}^{\div\div^+}(\begin{pmatrix}
 r_1^{\texttt{v}}\\
 0
\end{pmatrix}, \bs r_2; \mathbb S)\subseteq
\div\div\mathbb V_{k}^{\div\div^+}(\begin{pmatrix}
 r_1^{\texttt{v}}\\
 -1
\end{pmatrix}, \bs r_2; \mathbb S)\subseteq
\mathbb V^{L^2}_{k-2}(\bs r_2),
$$
hence $\div\div\mathbb V_{k}^{\div\div^+}(\begin{pmatrix}
 r_1^{\texttt{v}}\\
 -1
\end{pmatrix}, \bs r_2; \mathbb S)=\mathbb V^{L^2}_{k-2}(\bs r_2)$. On the other side,
\begin{align*}
% &\quad\dim \mathbb V_{k}^{\div\div^+}(\begin{pmatrix}
%  r_1^{\texttt{v}}\\
%  -1
% \end{pmatrix}, \bs r_2; \mathbb S)\cap\ker(\div\div)\\
&\quad\dim \mathbb V_{k}^{\div\div^+}(\begin{pmatrix}
 r_1^{\texttt{v}}\\
 -1
\end{pmatrix}, \bs r_2; \mathbb S)-\dim\mathbb V^{L^2}_{k-2}(\bs r_2) \\
&=3{r_1^{\texttt{v}}+2\choose2}|\Delta_0(\mathcal T_h)|-{r_1^{\texttt{v}}\choose2}|\Delta_0(\mathcal T_h)| +(3k-6r_1^{\texttt{v}}-2)|\Delta_1(\mathcal T_h)| \\
&\quad + |\Delta_2(\mathcal T_h)|\dim\mathbb B_k(\begin{pmatrix}
 r_1^{\texttt{v}}\\
 0
\end{pmatrix})+|\Delta_2(\mathcal T_h)|\dim\mathbb B_{k+2}(\begin{pmatrix}
 r_1^{\texttt{v}}+2\\
 1
\end{pmatrix})-4|\Delta_2(\mathcal T_h)| \\
&=2{r_1^{\texttt{v}}+3\choose2}|\Delta_0(\mathcal T_h)| +(3k-6r_1^{\texttt{v}}-5)|\Delta_1(\mathcal T_h)| \\
&\quad + |\Delta_2(\mathcal T_h)|\dim\mathbb B_k(\begin{pmatrix}
 r_1^{\texttt{v}}\\
 0
\end{pmatrix})+|\Delta_2(\mathcal T_h)|\dim\mathbb B_{k+2}(\begin{pmatrix}
 r_1^{\texttt{v}}+2\\
 1
\end{pmatrix})-|\Delta_2(\mathcal T_h)| \\
&\quad-3(|\Delta_0(\mathcal T_h)|-|\Delta_1(\mathcal T_h)|+|\Delta_2(\mathcal T_h)|).
\end{align*}
Thanks to the DoFs~\eqref{eq:2dCrdivfemdof0}-\eqref{eq:2dCrdivfemdof6} for $\mathbb V_{k+1}^{\div}(\begin{pmatrix}
r_1^{\texttt{v}} +1\\
 0
\end{pmatrix}, 
\begin{pmatrix}
r_1^{\texttt{v}}\\
0 
\end{pmatrix})$ and the Euler's formula,
$$
\dim \mathbb V_{k}^{\div\div^+}(\begin{pmatrix}
 r_1^{\texttt{v}}\\
 -1
\end{pmatrix}, \bs r_2; \mathbb S)-\dim\mathbb V^{L^2}_{k-2}(\bs r_2)=\dim\mathbb V_{k+1}^{\div}(\begin{pmatrix}
r_1^{\texttt{v}} +1\\
 0
\end{pmatrix}, 
\begin{pmatrix}
r_1^{\texttt{v}}\\
0 
\end{pmatrix})-3,
$$
which together with Lemma~\ref{lm:abstract} indicates the exactness of complex~\eqref{eq:femdivdivcomplex}.
\end{proof}

\begin{example}\rm
When $r_1^{\texttt{v}} = 0$ and $\bs r_2=-1$, we recover the finite element divdiv complex constructed in~\cite{Hu;Ma;Zhang:2020family} for $k\geq3$
\begin{equation*}
{\bf RT}\xrightarrow{\subset} 
\mathbb V_{k+1}^{\div}(\begin{pmatrix}
1\\
 0
\end{pmatrix}, 
\begin{pmatrix}
0\\
0 
\end{pmatrix})
\xrightarrow{{\sym\curl}}
{\rm HMZ}_k(
\begin{pmatrix}
 0\\
 -1
\end{pmatrix}
) 
\xrightarrow{\div\div}  
{\rm DG}_{k-2}(
\begin{pmatrix}
 -1\\
 -1
\end{pmatrix}
) 
\xrightarrow{}0.
\end{equation*}

\end{example}

%where $\boldsymbol{H}^{1}\left(\operatorname{div}, \Omega\right):=\{\boldsymbol{v}\in\boldsymbol{H}^{1}\left(\Omega ; \mathbb{R}^{2}\right): \div\boldsymbol{v}\in H^1(\Omega)\}$.

Another modification is to relax the smoothness $\bs H(\div\div,\Omega;\mathbb S)\cap \bs H(\div, \Omega; \mathbb S)$ to $\bs H(\div\div,\Omega;\mathbb S)$ only. We will modify \eqref{eq:divdiv+dof} by replacing \eqref{eq:symdivdiv+dof3}-\eqref{eq:symdivdiv+dof33} with
\begin{subequations}\label{eq:divdivdof}
\begin{align}
\int_e \bs n^{\intercal}\boldsymbol{\tau}\boldsymbol{n} \ q\dd s, &\quad  q\in \mathbb P_{k- 2(r_1^{\texttt{v}} +1)} (e), e\in \Delta_{1}(T),\label{eq:symdivdivdof3}\\
\int_e \bs t^{\intercal}\boldsymbol{\tau}\boldsymbol{n} \ q\dd s, &\quad  q\in \mathbb P_{k- 2(r_1^{\texttt{v}} +1)} (e), e\in \Delta_{1}(T),\label{eq:symdivdivdof3t}\\
\int_e \tr_{2}^{\div\div}(\boldsymbol{\tau}) \ q\dd s, &\quad  q\in \mathbb P_{k-1- 2r_1^{\texttt{v}}} (e), e\in \Delta_{1}(T),\label{eq:symdivdivdof32}
%\\
%\int_e \partial_n^i(\div\div\boldsymbol{\tau})\ q \dd s, &\quad  q\in \mathbb P_{k- 2(r_2^{\texttt{v}} +2)+i} (e), e\in \Delta_{1}(T),  i=0,\ldots, r_2^{e}, \label{eq:symdivdivdof44}
\end{align}
\end{subequations}
where $\tr_2^{\div\div}(\bs \tau) = \partial_{t}(\bs t^{\intercal}\bs \tau \bs n) + \boldsymbol{n}^{\intercal}\div\boldsymbol{\tau}$ is one of the trace operators of $\div\div$; see~\cite{ChenHuang2020}. 

Define
\begin{align}
&\mathbb V_{k}^{\div\div}(\begin{pmatrix}
 r_1^{\texttt{v}}\\
 -1
\end{pmatrix}, \bs r_2; \mathbb S) = \mathbb V:= \{ \boldsymbol{\tau}\in \boldsymbol{L}^2(\Omega;\mathbb S):\, \boldsymbol{\tau}|_T\in\mathbb P_{k}(T;\mathbb S)\;\forall~T\in\mathcal T_h, \notag\\
&\textrm{ all the DoFs~ \eqref{eq:divdiv+dof} by replacing \eqref{eq:symdivdiv+dof3}-\eqref{eq:symdivdiv+dof33} with \eqref{eq:divdivdof} except~\eqref{eq:symdivdivdof3t} are single-valued} \}. \notag%\label{eq:VkdivdivS}
\end{align}
As $\bs t^{\intercal}\bs \tau \bs n$ is local, the vector $\bs \tau\bs n$ is not continuous across edges. But $\tr_1^{\div\div}(\bs \tau)=\bs n^{\intercal}\boldsymbol{\tau}\boldsymbol{n}$ and $\tr_2^{\div\div}(\bs \tau)$ are continuous. So the space $\mathbb V_{k}^{\div\div}(\begin{pmatrix}
 r_1^{\texttt{v}}\\
 -1
\end{pmatrix}, \bs r_2; \mathbb S) \subset \bs H(\div\div,\Omega;\mathbb S)$ but not in $\bs H(\div,\Omega;\mathbb S)$. It cannot be derived from the BGG diagram~\eqref{eq:femdivdivbggexample1} as the induced space should be in $\bs H(\div,\Omega; \mathbb S)$.  
%, the dimension of space is
%\begin{align*}
%\dim \mathbb V_{k}^{\div\div}(\begin{pmatrix}
% r_1^{\texttt{v}}\\
% -1
%\end{pmatrix}, \bs r_2; \mathbb S)
%= 
%\dim \mathbb V_{k}^{\div\div^+}(\begin{pmatrix}
% r_1^{\texttt{v}}\\
% -1
%\end{pmatrix}, \bs r_2; \mathbb S)
%& - |\Delta_1(\mathcal T_h)| (k- 2r_1^{\texttt{v}} -1)\\
%& + 3 |\Delta_2(\mathcal T_h)| (k- 2r_1^{\texttt{v}} -1).
%\end{align*}

\begin{theorem}
The following finite element divdiv complex is exact
\begin{equation}\label{eq:divdivfemcomplex}
\resizebox{0.91\hsize}{!}{$
{\bf RT}\xrightarrow{\subset} \mathbb V_{k+1}^{\curl}(\begin{pmatrix}
r_1^{\texttt{v}} +1\\
 0
\end{pmatrix}; \mathbb R^2)
\xrightarrow{{\sym\curl}}\mathbb V_{k}^{\div\div}(\begin{pmatrix}
 r_1^{\texttt{v}}\\
 -1
\end{pmatrix}, \bs r_2; \mathbb S)\xrightarrow{\div\div}  \mathbb V^{L^2}_{k-2}(\bs r_2)\xrightarrow{}0.
$}
\end{equation}
\end{theorem}
\begin{proof}
For $\boldsymbol{\tau}=\sym\curl\boldsymbol{v}$, we have~\cite[Lemma 2.2]{ChenHuang2020}
$$
\boldsymbol{n}^{\intercal}\boldsymbol{\tau}\boldsymbol{n}=\partial_t(\boldsymbol{v}\cdot\boldsymbol{n}),\quad  \tr_2^{\div\div}(\bs \tau)=\partial_{tt}(\boldsymbol{v}\cdot\boldsymbol{t}).
$$
Then it is obvious that~\eqref{eq:divdivfemcomplex} is a complex. We will show the exactness of complex~\eqref{eq:divdivfemcomplex}.

Noting that $\mathbb V_{k}^{\div\div^+}(\begin{pmatrix}
 r_1^{\texttt{v}}\\
 -1
\end{pmatrix}, \bs r_2; \mathbb S)\subseteq\mathbb V_{k}^{\div\div}(\begin{pmatrix}
 r_1^{\texttt{v}}\\
 -1
\end{pmatrix}, \bs r_2; \mathbb S)$, by the exactness of complex~\eqref{eq:femdivdivcomplex}, we have
$$
\div\div\mathbb V_{k}^{\div\div}(\begin{pmatrix}
 r_1^{\texttt{v}}\\
 -1
\end{pmatrix}, \bs r_2; \mathbb S)= \mathbb V^{L^2}_{k-2}(\bs r_2).
$$
On the other side,
\begin{align*}
% &\quad\dim \mathbb V_{k}^{\div\div}(\begin{pmatrix}
%  r_1^{\texttt{v}}\\
%  -1
% \end{pmatrix}, \bs r_2; \mathbb S)\cap\ker(\div\div)\\
&\quad\dim \mathbb V_{k}^{\div\div}(\begin{pmatrix}
 r_1^{\texttt{v}}\\
 -1
\end{pmatrix}, \bs r_2; \mathbb S)-\dim\mathbb V^{L^2}_{k-2}(\bs r_2) \\
% &=3{r_1^{\texttt{v}}+2\choose2}|\Delta_0(\mathcal T_h)|-{r_1^{\texttt{v}}\choose2}|\Delta_0(\mathcal T_h)| +(3k-6r_1^{\texttt{v}}-2)|\Delta_1(\mathcal T_h)| \\
% &\quad + |\Delta_2(\mathcal T_h)|\dim\mathbb B_k(\begin{pmatrix}
%  r_1^{\texttt{v}}\\
%  0
% \end{pmatrix})+|\Delta_2(\mathcal T_h)|\dim\mathbb B_{k+2}(\begin{pmatrix}
%  r_1^{\texttt{v}}+2\\
%  1
% \end{pmatrix})-4|\Delta_2(\mathcal T_h)| \\
&=2{r_1^{\texttt{v}}+3\choose2}|\Delta_0(\mathcal T_h)| +2(k-2r_1^{\texttt{v}}-2)|\Delta_1(\mathcal T_h)|  \\
&\quad + |\Delta_2(\mathcal T_h)|\dim\mathbb B_k(\begin{pmatrix}
 r_1^{\texttt{v}}\\
 -1
\end{pmatrix})+|\Delta_2(\mathcal T_h)|\dim\mathbb B_{k+2}(\begin{pmatrix}
 r_1^{\texttt{v}}+2\\
 1
\end{pmatrix})-|\Delta_2(\mathcal T_h)| \\
&\quad-3(|\Delta_0(\mathcal T_h)|-|\Delta_1(\mathcal T_h)|+|\Delta_2(\mathcal T_h)|).
\end{align*}
Thanks to the DoFs~\eqref{eq:C12d0}-\eqref{eq:C12d2} for $\mathbb V_{k+1}^{\curl}(\begin{pmatrix}
r_1^{\texttt{v}} +1\\
 0
\end{pmatrix}; \mathbb R^2)$ and the Euler's formula,
$$
\dim \mathbb V_{k}^{\div\div}(\begin{pmatrix}
 r_1^{\texttt{v}}\\
 -1
\end{pmatrix}, \bs r_2; \mathbb S)-\dim\mathbb V^{L^2}_{k-2}(\bs r_2)=\dim\mathbb V_{k+1}^{\curl}(\begin{pmatrix}
r_1^{\texttt{v}} +1\\
 0
\end{pmatrix}; \mathbb R^2)-3,
$$
which together with Lemma~\ref{lm:abstract} ends the proof.
\end{proof}

\begin{example}\rm
When $r_1^{\texttt{v}} = 0$ and $\bs r_2=-1$, we recover the finite element divdiv complex constructed in~\cite{ChenHuang2020} for $k\geq3$
\begin{equation*}
{\bf RT}\xrightarrow{\subset} 
{\rm Herm}_{k+1}(\begin{pmatrix}
1\\
 0
\end{pmatrix};
\mathbb R^2
)
\xrightarrow{{\sym\curl}}
{\rm CH}_k(
\begin{pmatrix}
 0\\
 -1
\end{pmatrix}
) 
\xrightarrow{\div\div}  
{\rm DG}_{k-2}(
\begin{pmatrix}
 -1\\
 -1
\end{pmatrix}
) 
\xrightarrow{}0.
\end{equation*}
\end{example}

The first finite element divdiv complex in~\cite{Chen;Hu;Huang:2018Multigrid}
is based on the distributional divdiv complex 
\begin{equation*}%\label{eq:divdivcomplexH-1}
{\bf RT}\xrightarrow{\subset} \boldsymbol{H}^{1}\left(\Omega ; \mathbb{R}^{2}\right) \xrightarrow{\sym\curl} \boldsymbol{H}^{-1}(\operatorname{divdiv}, \Omega ; \mathbb{S}) \xrightarrow{\div\div} H^{-1}(\Omega) \rightarrow 0,
\end{equation*}
and not covered in this paper. 

\section{Conclusion and future work}

In recent years, there have been several advancements in the construction of finite element Hessian complexes, elasticity complexes, and divdiv complexes, as documented in~\cite{ChenHuang2020,Chen;Huang:2020Finite,Chen;Huang:2021Finite,Christiansen;Gopalakrishnan;Guzman;Hu:2020discrete,HuLiang2021,Hu;Liang;Ma:2021Finite,Hu;Liang;Ma;Zhang:2022conforming,Hu;Ma;Zhang:2020family}. Our primary objective is to extend the BGG construction to finite element complexes, unifying these findings and producing more systematic results. In this work, we have achieved this goal in two dimensions. However, the extension to three dimensions presents several challenges.

One of the challenges is the existence of finite element de Rham complexes with varying degrees of smoothness in three dimensions, which we will discuss in a forthcoming work~\cite{Chen;Huang:2022FEMcomplex3D}. Additionally, there is a mismatch in the continuity of Sobolev spaces $H^1(\Omega), H(\curl,\Omega)$, and $H(\div,\Omega)$. The main obstacle to generalizing BGG to the discrete case is the mismatch of tangential or normal continuity of $H(\curl)$ or $H(\div)$ conforming finite element spaces, respectively. In~\cite{Arnold;Hu:2020Complexes}, these spaces are replaced by $H^s(\Omega)$ Sobolev spaces with matching indices $s$. We will investigate further solutions in our future work. Moreover, edge-type finite elements in three dimensions are the most complex elements and require additional investigation.

To facilitate a clear and effective discussion, we will separate the two-dimensional and three-dimensional cases. Although the two-dimensional case is more straightforward and provides some insight into the three-dimensional case, treating them simultaneously in a simple and effective way is not possible due to the differences between the two cases. For instance, the proof of the div stability $\mathbb V^{\div}_{k}(\mathcal T_h; \boldsymbol{r}_1,\boldsymbol{r}_2) \xrightarrow{\div} \mathbb V^{L^2}_{k-1}(\mathcal T_h; \boldsymbol{r}_2)\to 0$ can be established by dimension count in 2D, but is much more technical in 3D.

\bibliographystyle{abbrv}
\bibliography{paper,refgeodecomp}
\end{document}